\pdfoutput=1
\documentclass[a4paper,10pt]{amsart}

\usepackage[utf8]{inputenc}
\usepackage{url}

\usepackage[expansion=false]{microtype}
\usepackage{amssymb,amsmath,amsopn,amsxtra,amsthm,amsfonts}
\usepackage{xr}
\usepackage[mathcal]{euscript}
\usepackage{mathalfa}

\usepackage[pdfusetitle,unicode]{hyperref}

\makeatletter
\let\c@equation\c@subsubsection

\makeatother

\newtheorem{cor}[subsubsection]{Corollary}
\newtheorem{lem}[subsubsection]{Lemma}
\newtheorem{prop}[subsubsection]{Proposition}

\newtheorem{thm}[subsubsection]{Theorem}

\theoremstyle{definition}
\newtheorem{defn}[subsubsection]{Definition}
\newtheorem{rem}[subsubsection]{Remark}
\newtheorem{constr}[subsubsection]{Construction}
\newtheorem{ex}[subsubsection]{Example}
\newtheorem{quest}[subsubsection]{Question}

\theoremstyle{remark}

\newcommand{\ssecref}[1]{\sectsign\ref{#1}}

\renewcommand{\eqref}[1]{(\ref{#1})}

\usepackage{tikz}
\usetikzlibrary{matrix}
\usepackage{tikz-cd}
\tikzset{shorten <>/.style={shorten >=#1,shorten <=#1}}

\usepackage{xspace}

\usepackage{xifthen}

\usepackage{xparse}

\newcommand{\nc}{\newcommand}
\nc{\renc}{\renewcommand}
\nc{\ssec}{\subsection}
\nc{\sssec}{\subsubsection}
\nc{\on}{\operatorname}
\nc{\term}[1]{#1\xspace}

\DeclareMathSymbol{A}{\mathalpha}{operators}{`A}
\DeclareMathSymbol{B}{\mathalpha}{operators}{`B}
\DeclareMathSymbol{C}{\mathalpha}{operators}{`C}
\DeclareMathSymbol{D}{\mathalpha}{operators}{`D}
\DeclareMathSymbol{E}{\mathalpha}{operators}{`E}
\DeclareMathSymbol{F}{\mathalpha}{operators}{`F}
\DeclareMathSymbol{G}{\mathalpha}{operators}{`G}
\DeclareMathSymbol{H}{\mathalpha}{operators}{`H}
\DeclareMathSymbol{I}{\mathalpha}{operators}{`I}
\DeclareMathSymbol{J}{\mathalpha}{operators}{`J}
\DeclareMathSymbol{K}{\mathalpha}{operators}{`K}
\DeclareMathSymbol{L}{\mathalpha}{operators}{`L}
\DeclareMathSymbol{M}{\mathalpha}{operators}{`M}
\DeclareMathSymbol{N}{\mathalpha}{operators}{`N}
\DeclareMathSymbol{O}{\mathalpha}{operators}{`O}
\DeclareMathSymbol{P}{\mathalpha}{operators}{`P}
\DeclareMathSymbol{Q}{\mathalpha}{operators}{`Q}
\DeclareMathSymbol{R}{\mathalpha}{operators}{`R}
\DeclareMathSymbol{S}{\mathalpha}{operators}{`S}
\DeclareMathSymbol{T}{\mathalpha}{operators}{`T}
\DeclareMathSymbol{U}{\mathalpha}{operators}{`U}
\DeclareMathSymbol{V}{\mathalpha}{operators}{`V}
\DeclareMathSymbol{W}{\mathalpha}{operators}{`W}
\DeclareMathSymbol{X}{\mathalpha}{operators}{`X}
\DeclareMathSymbol{Y}{\mathalpha}{operators}{`Y}
\DeclareMathSymbol{Z}{\mathalpha}{operators}{`Z}

\nc{\sA}{\ensuremath{\mathcal{A}}\xspace}
\nc{\sB}{\ensuremath{\mathcal{B}}\xspace}
\nc{\sC}{\ensuremath{\mathcal{C}}\xspace}
\nc{\sD}{\ensuremath{\mathcal{D}}\xspace}
\nc{\sE}{\ensuremath{\mathcal{E}}\xspace}
\nc{\sF}{\ensuremath{\mathcal{F}}\xspace}
\nc{\sG}{\ensuremath{\mathcal{G}}\xspace}
\nc{\sH}{\ensuremath{\mathcal{H}}\xspace}
\nc{\sI}{\ensuremath{\mathcal{I}}\xspace}
\nc{\sJ}{\ensuremath{\mathcal{J}}\xspace}
\nc{\sK}{\ensuremath{\mathcal{K}}\xspace}
\nc{\sL}{\ensuremath{\mathcal{L}}\xspace}
\nc{\sM}{\ensuremath{\mathcal{M}}\xspace}
\nc{\sN}{\ensuremath{\mathcal{N}}\xspace}
\nc{\sO}{\ensuremath{\mathcal{O}}\xspace}
\nc{\sP}{\ensuremath{\mathcal{P}}\xspace}
\nc{\sQ}{\ensuremath{\mathcal{Q}}\xspace}
\nc{\sR}{\ensuremath{\mathcal{R}}\xspace}
\nc{\sS}{\ensuremath{\mathcal{S}}\xspace}
\nc{\sT}{\ensuremath{\mathcal{T}}\xspace}
\nc{\sU}{\ensuremath{\mathcal{U}}\xspace}
\nc{\sV}{\ensuremath{\mathcal{V}}\xspace}
\nc{\sW}{\ensuremath{\mathcal{W}}\xspace}
\nc{\sX}{\ensuremath{\mathcal{X}}\xspace}
\nc{\sY}{\ensuremath{\mathcal{Y}}\xspace}
\nc{\sZ}{\ensuremath{\mathcal{Z}}\xspace}

\nc{\bA}{\ensuremath{\mathbf{A}}\xspace}
\nc{\bB}{\ensuremath{\mathbf{B}}\xspace}
\nc{\bC}{\ensuremath{\mathbf{C}}\xspace}
\nc{\bD}{\ensuremath{\mathbf{D}}\xspace}
\nc{\bE}{\ensuremath{\mathbf{E}}\xspace}
\nc{\bF}{\ensuremath{\mathbf{F}}\xspace}
\nc{\bG}{\ensuremath{\mathbf{G}}\xspace}
\nc{\bH}{\ensuremath{\mathbf{H}}\xspace}
\nc{\bI}{\ensuremath{\mathbf{I}}\xspace}
\nc{\bJ}{\ensuremath{\mathbf{J}}\xspace}
\nc{\bK}{\ensuremath{\mathbf{K}}\xspace}
\nc{\bL}{\ensuremath{\mathbf{L}}\xspace}
\nc{\bM}{\ensuremath{\mathbf{M}}\xspace}
\nc{\bN}{\ensuremath{\mathbf{N}}\xspace}
\nc{\bO}{\ensuremath{\mathbf{O}}\xspace}
\nc{\bP}{\ensuremath{\mathbf{P}}\xspace}
\nc{\bQ}{\ensuremath{\mathbf{Q}}\xspace}
\nc{\bR}{\ensuremath{\mathbf{R}}\xspace}
\nc{\bS}{\ensuremath{\mathbf{S}}\xspace}
\nc{\bT}{\ensuremath{\mathbf{T}}\xspace}
\nc{\bU}{\ensuremath{\mathbf{U}}\xspace}
\nc{\bV}{\ensuremath{\mathbf{V}}\xspace}
\nc{\bW}{\ensuremath{\mathbf{W}}\xspace}
\nc{\bX}{\ensuremath{\mathbf{X}}\xspace}
\nc{\bY}{\ensuremath{\mathbf{Y}}\xspace}
\nc{\bZ}{\ensuremath{\mathbf{Z}}\xspace}

\nc{\dA}{\ensuremath{\mathds{A}}\xspace}
\nc{\dB}{\ensuremath{\mathds{B}}\xspace}
\nc{\dC}{\ensuremath{\mathds{C}}\xspace}
\nc{\dD}{\ensuremath{\mathds{D}}\xspace}
\nc{\dE}{\ensuremath{\mathds{E}}\xspace}
\nc{\dF}{\ensuremath{\mathds{F}}\xspace}
\nc{\dG}{\ensuremath{\mathds{G}}\xspace}
\nc{\dH}{\ensuremath{\mathds{H}}\xspace}
\nc{\dI}{\ensuremath{\mathds{I}}\xspace}
\nc{\dJ}{\ensuremath{\mathds{J}}\xspace}
\nc{\dK}{\ensuremath{\mathds{K}}\xspace}
\nc{\dL}{\ensuremath{\mathds{L}}\xspace}
\nc{\dM}{\ensuremath{\mathds{M}}\xspace}
\nc{\dN}{\ensuremath{\mathds{N}}\xspace}
\nc{\dO}{\ensuremath{\mathds{O}}\xspace}
\nc{\dP}{\ensuremath{\mathds{P}}\xspace}
\nc{\dQ}{\ensuremath{\mathds{Q}}\xspace}
\nc{\dR}{\ensuremath{\mathds{R}}\xspace}
\nc{\dS}{\ensuremath{\mathds{S}}\xspace}
\nc{\dT}{\ensuremath{\mathds{T}}\xspace}
\nc{\dU}{\ensuremath{\mathds{U}}\xspace}
\nc{\dV}{\ensuremath{\mathds{V}}\xspace}
\nc{\dW}{\ensuremath{\mathds{W}}\xspace}
\nc{\dX}{\ensuremath{\mathds{X}}\xspace}
\nc{\dY}{\ensuremath{\mathds{Y}}\xspace}
\nc{\dZ}{\ensuremath{\mathds{Z}}\xspace}

\nc{\bbA}{\ensuremath{\mathbb{A}}\xspace}
\nc{\bbB}{\ensuremath{\mathbb{B}}\xspace}
\nc{\bbC}{\ensuremath{\mathbb{C}}\xspace}
\nc{\bbD}{\ensuremath{\mathbb{D}}\xspace}
\nc{\bbE}{\ensuremath{\mathbb{E}}\xspace}
\nc{\bbF}{\ensuremath{\mathbb{F}}\xspace}
\nc{\bbG}{\ensuremath{\mathbb{G}}\xspace}
\nc{\bbH}{\ensuremath{\mathbb{H}}\xspace}
\nc{\bbI}{\ensuremath{\mathbb{I}}\xspace}
\nc{\bbJ}{\ensuremath{\mathbb{J}}\xspace}
\nc{\bbK}{\ensuremath{\mathbb{K}}\xspace}
\nc{\bbL}{\ensuremath{\mathbb{L}}\xspace}
\nc{\bbM}{\ensuremath{\mathbb{M}}\xspace}
\nc{\bbN}{\ensuremath{\mathbb{N}}\xspace}
\nc{\bbO}{\ensuremath{\mathbb{O}}\xspace}
\nc{\bbP}{\ensuremath{\mathbb{P}}\xspace}
\nc{\bbQ}{\ensuremath{\mathbb{Q}}\xspace}
\nc{\bbR}{\ensuremath{\mathbb{R}}\xspace}
\nc{\bbS}{\ensuremath{\mathbb{S}}\xspace}
\nc{\bbT}{\ensuremath{\mathbb{T}}\xspace}
\nc{\bbU}{\ensuremath{\mathbb{U}}\xspace}
\nc{\bbV}{\ensuremath{\mathbb{V}}\xspace}
\nc{\bbW}{\ensuremath{\mathbb{W}}\xspace}
\nc{\bbX}{\ensuremath{\mathbb{X}}\xspace}
\nc{\bbY}{\ensuremath{\mathbb{Y}}\xspace}
\nc{\bbZ}{\ensuremath{\mathbb{Z}}\xspace}

\nc{\mrm}[1]{\ensuremath{\mathrm{#1}}\xspace}
\nc{\mbf}[1]{\ensuremath{\mathbf{#1}}\xspace}
\nc{\mcal}[1]{\ensuremath{\mathcal{#1}}\xspace}
\nc{\msc}[1]{\ensuremath{\mathscr{#1}}\xspace}

\renc{\bar}[1]{\overline{#1}}

\let\sectsign\S
\let\S\relax

\nc{\sub}{\subset}
\nc{\too}{\longrightarrow}
\nc{\hook}{\hookrightarrow}
\nc*{\hooklongrightarrow}{\ensuremath{\lhook\joinrel\relbar\joinrel\rightarrow}}
\nc{\hooklong}{\hooklongrightarrow}
\nc{\twoheadlongrightarrow}{\relbar\joinrel\twoheadrightarrow}
\nc{\shiso}{\approx}
\nc{\isoto}{\xrightarrow{\sim}}
\nc{\isofrom}{\xleftarrow{\sim}}
\renc{\ge}{\geqslant}
\renc{\le}{\leqslant}
\renc{\geq}{\geqslant}
\renc{\leq}{\leqslant}

\nc{\id}{\mathrm{id}}

\DeclareMathOperator{\rk}{\mathrm{rk}}
\DeclareMathOperator{\Hom}{\mathrm{Hom}}
\nc{\uHom}{\underline{\smash{\Hom}}}
\DeclareMathOperator{\Maps}{\mathrm{Maps}}
\DeclareMathOperator{\Aut}{\mathrm{Aut}}
\DeclareMathOperator{\End}{\mathrm{End}}
\DeclareMathOperator{\Sym}{\mathrm{Sym}}
\nc{\Pre}{\mathrm{PSh}{}}
\nc{\uEnd}{\underline{\smash{\End}}}

\renc{\lim}{\operatorname*{lim}}
\nc{\colim}{\operatorname*{colim}}
\nc{\Cofib}{\on{Cofib}}
\nc{\Fib}{\on{Fib}}
\nc{\initial}{\varnothing}
\nc{\op}{\mathrm{op}}

\let\bigcoprod=\coprod
\renc{\coprod}{\sqcup}

\nc{\bDelta}{\mbf{\Delta}}
\nc{\DM}{\mbf{DM}}
\nc{\eff}{\mathrm{eff}}
\nc{\veff}{\mathrm{veff}}
\nc{\cyc}{{\mrm{cyc}}}
\nc{\corr}{{\on{corr}}}
\nc{\fet}{{\mrm{f\acute et}}}
\nc{\fsyn}{{\mrm{fsyn}}}
\nc{\fqs}{{\mrm{fqs}}}
\nc{\syn}{{\mrm{syn}}}
\nc{\Perf}{\mbf{Perf}}
\nc{\Pic}{\mrm{Pic}}
\nc{\perf}{\mrm{perf}}
\nc{\oblv}{\on{oblv}}
\nc{\exact}{\on{exact}}

\nc{\F}{{\on{F}}}
\nc{\clopen}{{\mrm{clopen}}}
\nc{\B}{\mrm{B}}
\nc{\D}{\mrm{D}}
\nc{\Fin}{\on{Fin}}
\nc{\Cut}{\on{Cut}}
\nc{\Cart}{\on{Cart}}
\nc{\pairs}{\mathsf{pairs}}
\nc{\Pairs}{\mathrm{Pair}}
\nc{\Trip}{\mathrm{Trip}}
\nc{\Lab}{\mathrm{Lab}}
\nc{\coCart}{\mathrm{coCart}}
\nc{\RKE}{\mathrm{RKE}}
\nc{\strict}{\mathrm{strict}}
\nc{\Emb}{\mathrm{Emb}}
\nc{\Split}{\mathrm{Split}}
\nc{\Set}{\mathrm{Set}}
\nc{\sSets}{\mathrm{sSets}}
\nc{\pb}{\mathrm{pb}}
\nc{\fib}{\mathrm{fib}}
\nc{\cofib}{\mathrm{cofib}}
\nc{\diff}{\mrm{diff}}
\nc{\gp}{\mrm{gp}}
\nc{\chr}{\mrm{char}}
\nc{\mgp}{\mrm{mot-gp}}
\nc{\FSyn}{\mrm{FSyn}}
\nc{\FEt}{\mrm{FEt}}
\nc{\Spc}{\mrm{Spc}}
\nc{\Ob}{\mrm{Ob}}
\nc{\Spt}{\mrm{Spt}}
\nc{\T}{\bT}
\nc{\suspinf}{\Sigma^\infty}
\nc{\h}{\mrm{h}}
\nc{\uhom}{\underline{\mathrm{Hom}}}
\nc{\umap}{\underline{\mathrm{Maps}}}
\renc{\H}{\bH}
\nc{\Einfty}{{\sE_\infty}}
\nc{\Eone}{{\sE_1}}
\nc{\Stab}{\mrm{Stab}}
\nc{\lax}{{\mrm{lax}}}
\nc{\cocart}{{\mrm{cocart}}}
\nc{\Sch}{\mrm{Sch}}
\nc{\dSch}{\mrm{dSch}}
\nc{\Aff}{\mrm{Aff}}
\nc{\SmAff}{\mrm{SmAff}}
\nc{\dAff}{\mrm{dAff}}
\nc{\Fr}{\on{Fr}}
\nc{\A}{\mathbf{A}}
\nc{\N}{\mathbf{N}}
\nc{\Z}{\mathbf{Z}}
\nc{\Q}{\mathbf{Q}}
\nc{\Oo}{\mathcal{O}} 
\nc{\red}{{\on{red}}}
\nc{\Voev}{{\on{Voev}}}
\nc{\Corr}{\mrm{Corr}}
\nc{\Span}{\mathbf{Corr}}
\nc{\Gap}{\mrm{Gap}}
\nc{\Filt}{\mrm{Filt}}
\nc{\Corrfr}{\Corr^{\fr}}
\nc{\Corrvfr}{\Corr^{\Vfr}}
\nc{\Spec}{\on{Spec}}
\nc{\Sm}{\mrm{Sm}}
\nc{\QSm}{\mrm{QSm}}
\nc{\Gm}{\mathbf{G}_{\mrm{m}}}
\renc{\P}{\bP}
\nc{\nis}{\mathrm{nis}}
\nc{\zar}{\mathrm{zar}}
\nc{\et}{\mathrm{\acute et}}
\nc{\all}{\mathrm{all}}
\nc{\fold}{\mathrm{fold}}
\nc{\Fun}{\mathrm{Fun}}
\nc{\Ho}{\mathrm{Ho}}
\nc{\Segal}{\mathrm{Segal}}
\nc{\Mon}{\mrm{Mon}{}}
\nc{\Ab}{\mrm{Ab}}
\nc{\Gr}{\mrm{Gr}}
\nc{\Sh}{\on{Sh}}
\nc{\M}{\mrm{M}}
\nc{\Lhtp}{L_{\A^1}}
\nc{\Lmot}{L_{\mrm{mot}}}
\nc{\mot}{\mrm{mot}}
\nc{\SH}{\mbf{SH}}
\nc{\RR}{\mbf{R}}
\nc{\CC}{\mbf{C}}
\nc{\Mod}{\mbf{Mod}}
\nc{\QCoh}{\mbf{QCoh}}
\nc{\MonUnit}{\mbf{1}}
\nc{\tr}{\on{tr}}
\nc{\vop}{\mrm{vop}}
\nc{\fr}{{\on{fr}}}
\nc{\Ar}{\mrm{Ar}}
\nc{\Vfr}{\on{Vfr}}
\nc{\frdiff}{{\on{frdiff}}}
\nc{\frGys}{\on{frGys}}
\nc{\SHfr}{\SH^{\fr}}
\nc{\SHfrdiff}{\SH^{\frdiff}}
\nc{\SHfrGys}{\SH^{\frGys}}
\nc{\InftyCat}{\infty\textnormal{-}\mrm{Cat}}
\nc{\TriCat}{\mathrm{TriCat}}
\nc{\Cat}{\mathrm{1\textnormal{-}Cat}}
\nc{\Th}{\on{Th}}
\def\G{\bG}
\nc{\CMon}{\mrm{CMon}{}}
\nc{\CAlg}{\mrm{CAlg}{}}
\nc{\MGL}{\mrm{MGL}}
\nc{\MSL}{\mrm{MSL}}
\nc{\MSp}{\mrm{MSp}}
\nc{\Seg}{\mrm{Seg}{}}
\nc{\Tw}{\mrm{Tw}}
\nc{\sslash}{/\mkern-6mu/}
\nc{\PrL}{\mrm{Pr}^\mrm{L}}
\nc{\PrR}{\mrm{Pr}^\mrm{R}}
\nc{\pr}{\mrm{pr}}
\nc{\efr}{\mrm{efr}}
\nc{\nfr}{\mrm{nfr}}
\nc{\dfr}{\mrm{fr}}
\nc{\tfr}{\mrm{tfr}}
\nc{\Vect}{\mrm{Vect}}
\nc{\sVect}{\mrm{sVect}}
\nc{\fix}{\mrm{fix}}
\nc{\Hilb}{\mathrm{Hilb}}
\nc{\flci}{\mathrm{flci}}
\nc{\Isom}{\mathrm{Isom}}
\nc{\GL}{\mathrm{GL}}
\nc{\BGL}{\mathrm{BGL}}
\nc{\SL}{\mathrm{SL}}
\nc{\Sp}{\mathrm{Sp}}
\nc{\fin}{\mathrm{fin}}
\nc{\cl}{\mathrm{cl}}
\nc{\cn}{\mathrm{cn}}
\nc{\sm}{\mathrm{sm}}
\nc{\heart}{\heartsuit}
\nc{\1}{\mbf{1}}
\nc{\ornt}{\mrm{or}}
\nc{\GW}{\mrm{GW}}
\nc{\ev}{\mrm{ev}}
\nc{\FSYN}{\mathcal{FS}\mrm{yn}}
\nc{\FQSM}{\mathcal{FQS}\mathrm{m}}
\nc{\fc}{\mrm{fc}{}}

\let\phi\varphi

\let\emptyset\varnothing

\nc{\inftyCat}{\term{$\infty$-category}}
\nc{\inftyCats}{\term{$\infty$-categories}}

\nc{\inftyOneCat}{\term{$(\infty,1)$-category}}
\nc{\inftyOneCats}{\term{$(\infty,1)$-categories}}

\nc{\inftyGrpd}{\term{$\infty$-groupoid}}
\nc{\inftyGrpds}{\term{$\infty$-groupoids}}

\nc{\inftyTop}{\term{$\infty$-topos}}
\nc{\inftyTops}{\term{$\infty$-toposes}}

\nc{\inftyTwoCat}{\term{$(\infty,2)$-category}}
\nc{\inftyTwoCats}{\term{$(\infty,2)$-categories}}

\title{Modules over algebraic cobordism}

\author[E. Elmanto]{Elden Elmanto}
\address{Department of Mathematics\\
Harvard University\\
1 Oxford St.\\
Cambridge, MA 02138\\
USA}
\email{\href{mailto:elmanto@math.harvard.edu}{elmanto@math.harvard.edu}}
\urladdr{\url{https://www.eldenelmanto.com/}}

\author[M. Hoyois]{Marc Hoyois}
\address{Fakultät für Mathematik\\
Universität Regensburg\\
Universitätsstr. 31\\
93040 Regensburg\\
Germany}
\email{\href{mailto:marc.hoyois@ur.de}{marc.hoyois@ur.de}}
\urladdr{\url{http://www.mathematik.ur.de/hoyois/}}
\thanks{M.H.\ was partially supported by NSF grant DMS-1761718}

\author[A. A. Khan]{Adeel A. Khan}
\address{IHES\\
35 route de Chartres\\
91440 Bures-sur-Yvette\\
France}
\email{\href{mailto:khan@ihes.fr}{khan@ihes.fr}}
\urladdr{\url{https://www.preschema.com}}

\author[V. Sosnilo]{Vladimir Sosnilo}
\address{Laboratory ``Modern Algebra and Applications''\\
St. Petersburg State University\\
14th line, 29B\\
199178 Saint Petersburg\\
Russia}
\email{\href{mailto:vsosnilo@gmail.com}{vsosnilo@gmail.com}}

\author[M. Yakerson]{Maria Yakerson}
\address{Institute for Mathematical Research (FIM)\\
ETH Z\"urich \\
R\"amistr. 101\\  
8092 Z\"urich\\
Switzerland}
\email{\href{mailto:maria.yakerson@math.ethz.ch}{maria.yakerson@math.ethz.ch}}
\urladdr{\url{https://www.muramatik.com}}
\thanks{M.Y.\ was supported by DFG - SPP 1786 ``Homotopy Theory and Algebraic Geometry''}

\date{\today}

\begin{document}

\begin{abstract}
We prove that the $\infty$-category of $\MGL$-modules over any scheme is equivalent to the $\infty$-category of motivic spectra with finite syntomic transfers. Using the recognition principle for infinite $\P^1$-loop spaces, we deduce that very effective $\MGL$-modules over a perfect field are equivalent to grouplike motivic spaces with finite syntomic transfers.

Along the way, we describe any motivic Thom spectrum built from virtual vector bundles of nonnegative rank in terms of the moduli stack of finite quasi-smooth derived schemes with the corresponding tangential structure.
In particular, over a regular equicharacteristic base, we show that $\Omega^\infty_{\P^1}\MGL$ is the $\A^1$-homotopy type of the moduli stack of virtual finite flat local complete intersections, and that for $n>0$, $\Omega^\infty_{\P^1} \Sigma^n_{\P^1} \MGL$ is the $\A^1$-homotopy type of the moduli stack of finite quasi-smooth derived schemes of virtual dimension $-n$.
\end{abstract}

\maketitle

\vspace{0em}
\parskip 0.2cm

\parskip 0pt
\tableofcontents

\parskip 0.2cm

\section{Introduction}

This article contains two main results:
\begin{itemize}
	\item a concrete description of the $\infty$-category of modules over Voevodsky's algebraic cobordism spectrum $\MGL$ (and its variants such as $\MSL$);
	\item a computation of the infinite $\P^1$-loop spaces of effective motivic Thom spectra in terms of finite quasi-smooth derived schemes with tangential structure and cobordisms.
\end{itemize}
Both results have an incarnation over arbitrary base schemes and take a more concrete form over perfect fields. We discuss them in more details in \ssecref{ssec:intro-modules} and \ssecref{ssec:intro-thom} below.

\ssec{Modules over algebraic cobordism}\label{ssec:intro-modules}
If $S$ is a regular scheme over a field with resolutions of singularities, there are well-known equivalences of $\infty$-categories
\begin{gather*}
	\Mod_{H\Z}(\SH(S))\simeq\DM(S),\\
	\Mod_{H\tilde{\Z}}(\SH(S))\simeq \widetilde{\DM}(S),
\end{gather*}
expressing Voevodsky's $\infty$-category of motives $\DM(S)$ and its Milnor–Witt refinement $\widetilde{\DM}(S)$ as $\infty$-categories of modules over the motivic $\Einfty$-ring spectra $H\Z$ and $H\tilde{\Z}$, which represent motivic cohomology and Milnor–Witt motivic cohomology; see \cite{Rondigs:2008,CDintegral,BachmannFasel,ElmantoKolderup}.
These equivalences mean that a structure of $H\Z$-module (resp.\ of $H\tilde{\Z}$-module) on a motivic spectrum is equivalent to a structure of transfers in the sense of Voevodsky (resp.\ a structure of Milnor–Witt transfers in the sense of Calmès and Fasel). 

Our main result, Theorem~\ref{thm:MGL-modules}, gives an analogous description of modules over Voevodsky's algebraic cobordism spectrum $\MGL$: we construct an equivalence between $\MGL$-modules and motivic spectra with \emph{finite syntomic transfers}:
\begin{equation}\label{eqn:intro}
	\Mod_{\MGL}(\SH(S))\simeq \SH^\fsyn(S).
\end{equation}
Notably, we do not need resolutions of singularities and are able to prove this over arbitrary schemes $S$. Furthermore, we obtain similar results for other motivic Thom spectra; for example, $\MSL$-modules are equivalent to motivic spectra with transfers along finite syntomic morphisms with trivialized canonical sheaf (Theorem~\ref{thm:MSL-modules}). 
It is worth noting that, even though both sides of~\eqref{eqn:intro} only involve classical schemes, our construction of the equivalence uses derived algebraic geometry in an essential way.

Over a perfect field $k$, we prove a cancellation theorem for finite syntomic correspondences. This allows us to refine~\eqref{eqn:intro} to an equivalence
\[
 \Mod_\MGL(\SH^\veff(k))\simeq\H^\fsyn(k)^\gp
\]
between \emph{very effective} $\MGL$-modules and grouplike motivic \emph{spaces} with finite syntomic transfers (see Theorem~\ref{thm:MGL-modules2}).

\ssec{Framed correspondences and the motivic recognition principle}\label{ssec:intro-frames}

The starting point of the proof of~\eqref{eqn:intro} is a description of motivic Thom spectra in terms of \emph{framed correspondences}. The notion of framed correspondence was introduced by Voevodsky \cite{voevodsky2001notes} and later developed by Garkusha, Panin, Ananyevskiy, and Neshitov \cite{garkusha2014framed,hitr,agp,gnp}. Subsequently, a more flexible formalism of framed correspondences was developed by the authors in \cite{EHKSY1}.
 There we construct a symmetric monoidal $\infty$-category $\Span^\fr(\Sm_S)$ whose objects are smooth $S$-schemes and whose morphisms are spans
\begin{equation}\label{eqn:intro-span}
\begin{tikzcd}
   & Z \ar[swap]{ld}{f}\ar{rd}{g} & \\
  X &   & Y
\end{tikzcd}
\end{equation}
with $f$ finite syntomic, together with an equivalence $\sL_f\simeq 0$ in $K(Z)$, where $\sL_f$ is the cotangent complex of $f$ and $K(Z)$ is the algebraic K-theory space of $Z$. Starting with the $\infty$-category $\Span^\fr(\Sm_S)$, we can define the symmetric monoidal $\infty$-categories $\H^\fr(S)$ of \emph{framed motivic spaces} and $\SH^\fr(S)$ of \emph{framed motivic spectra}. The \emph{reconstruction theorem} states that there is an equivalence
\begin{equation}\label{eqn:intro-reconstruction}
	\SH(S)\simeq \SH^\fr(S)
\end{equation}
between motivic spectra and framed motivic spectra over any scheme $S$ \cite[Theorem 18]{framed-loc}. This equivalence can be regarded as the ``sphere spectrum version'' of the equivalences discussed in \ssecref{ssec:intro-modules}.
Although the right-hand side of~\eqref{eqn:intro-reconstruction} is obviously more complicated than the left-hand side, the point of the reconstruction theorem is that many motivic spectra of interest admit simpler descriptions as framed motivic spectra. For example, the motivic cohomology spectrum $H\Z_S$ turns out to be the framed suspension spectrum of the constant sheaf $\Z$ \cite[Theorem~21]{framed-loc}.

Over a perfect field $k$, the \emph{motivic recognition principle} states that the framed suspension functor
\[
\Sigma^\infty_{\T,\fr}\colon \H^\fr(k) \to \SH^\fr(k)\simeq \SH(k)
\]
is fully faithful when restricted to grouplike objects, and that its essential image is the subcategory of very effective motivic spectra \cite[Theorem 3.5.14]{EHKSY1}. In particular, the functor $\Omega^\infty_{\T,\fr}\Sigma^\infty_{\T,\fr}$ is computed as group completion on framed motivic spaces. Thus, if a motivic spectrum $E$ over $k$ is shown to be the framed suspension spectrum of a framed motivic space $X$, then its infinite $\P^1$-loop space $\Omega^\infty_\T E$ is the group completion $X^\gp$. This will be our strategy to compute the infinite $\P^1$-loop spaces of motivic Thom spectra.

\ssec{Geometric models of motivic Thom spectra}\label{ssec:intro-thom}

The general notion of motivic Thom spectrum was introduced in \cite[Section 16]{norms}. In particular, there is a motivic Thom spectrum $M\beta$ associated with any natural transformation $\beta\colon B\to K$, where $K$ is the presheaf of K-theory spaces on smooth schemes. For example, if $\beta$ is the inclusion of the rank $n$ summand of K-theory, then $M\beta\simeq \Sigma^n_\T\MGL$. The motivic spectrum $M\beta$ is very effective when $\beta$ lands in the rank $\geq 0$ summand $K_{\geq 0}\subset K$.

In this paper, we show that the motivic Thom spectrum $M\beta$ of any morphism $\beta\colon B\to K_{\geq 0}$ is the framed suspension spectrum of an explicit framed motivic space.
Before stating the general result more precisely, we mention some important special cases:
\begin{enumerate}
	\item If $Y$ is smooth over $S$ and $\xi\in K(Y)$ is a K-theory element of rank $\geq 0$, the Thom spectrum $\Th_{Y/S}(\xi)\in \SH(S)$ is the framed suspension spectrum of the presheaf on $\Span^\fr(\Sm_S)$ sending $X$ to the $\infty$-groupoid of spans~\eqref{eqn:intro-span} where $Z$ is a \emph{derived} scheme and $f$ is finite and quasi-smooth, together with an equivalence $\sL_f\simeq -g^*(\xi)$ in $K(Z)$.
	\item The algebraic cobordism spectrum $\MGL_S$ is the framed suspension spectrum of the moduli stack $\FSYN_S$ of finite syntomic $S$-schemes. More generally, for $n\geq 0$, $\Sigma^n_\T \MGL_S$ is the framed suspension spectrum of the moduli stack $\FQSM_S^{n}$ of finite quasi-smooth derived $S$-schemes of relative virtual dimension $-n$. 
	\item The special linear algebraic cobordism spectrum $\MSL_S$ is the framed suspension spectrum of the moduli stack $\FSYN_S^\ornt$ of finite syntomic $S$-schemes with trivialized canonical sheaf. More generally, for $n\geq 0$, $\Sigma^n_\T \MSL_S$ is the framed suspension spectrum of the moduli stack $\FQSM_S^{\ornt,n}$ of finite quasi-smooth derived $S$-schemes of relative virtual dimension $-n$ with trivialized canonical sheaf.
\end{enumerate}
We now explain the general paradigm. Given a natural transformation $\beta\colon B\to K$ of presheaves on smooth $S$-schemes, we define a \emph{$\beta$-structure} on a morphism $f\colon Z\to X$ between smooth $S$-schemes to be a lift of its shifted cotangent complex to $B(Z)$:
\[
\begin{tikzcd}
	& B \ar{d}{\beta} \\
	Z \ar[dashed]{ur} \ar{r}[swap]{-\sL_f} & K\rlap.
\end{tikzcd}
\]
More generally, if $f\colon Z\to X$ is a morphism between quasi-smooth derived $S$-schemes, we define a $\beta$-structure on $f$ to be a lift of $-\sL_f$ to $\tilde B(Z)$, where $\tilde B$ is the left Kan extension of $B$ to quasi-smooth derived $S$-schemes (although left Kan extension is a rather abstract procedure, it turns out that $\tilde B$ admits a concrete description for every $B$ of interest). Then, for any $\beta\colon B\to K_{\geq 0}$, we show that the motivic Thom spectrum $M\beta \in \SH(S)$ is the framed suspension spectrum of the moduli stack $\FQSM_S^\beta$ of finite quasi-smooth derived $S$-schemes with $\beta$-structure (Theorem~\ref{thm:thom-general}).

Over a perfect field $k$, this result becomes much more concrete. Indeed, using the motivic recognition principle, we deduce that $\Omega^\infty_\T M\beta$ is the motivic homotopy type of the group completion of the moduli stack $\FQSM_k^\beta$ (Corollary~\ref{cor:thom-general}). Even better, the group completion is redundant if $\beta$ lands in the positive-rank summand of K-theory.

The case of the algebraic cobordism spectrum $\MGL_k$ deserves more elaboration. For its infinite $\P^1$-loop space, we obtain an equivalence
\[
\Omega^\infty_\T\MGL_k \simeq L_\zar\Lhtp\FSYN_k^\gp
\]
of $\Einfty$-ring spaces with framed transfers. For $n>0$, we obtain equivalences of $\FSYN_k$-modules
\[
\Omega^\infty_\T\Sigma^n_\T\MGL_k \simeq L_\nis\Lhtp\FQSM_k^{n}
\]
(see Corollary~\ref{cor:MGL}). Finally, using an algebraic version of Whitney's embedding theorem for finite schemes, we can replace the moduli stack $\FSYN_k$ by the Hilbert scheme $\Hilb^\flci(\A^\infty_k)$ of finite local complete intersections in $\A^\infty_k$, which is a smooth ind-scheme and a commutative monoid up to $\A^1$-homotopy (see Theorem~\ref{thm:MGL-Hilb}):
\[
\Omega^\infty_\T\MGL \simeq L_\zar(\Lhtp\Hilb^\flci(\A^\infty_k))^\gp.
\]

\ssec{Analogies with topology}\label{ssec:intro-topology}

A cobordism between two compact smooth manifolds $M$ and $N$ is a smooth manifold $W$ with a proper morphism $W\to \bR$ whose fibers over $0$ and $1$ are identified with $M$ and $N$; see for example \cite[\sectsign 1]{Quillen:1971}. In other words, cobordisms are paths in the moduli stack of compact smooth manifolds (though one needs the notion of quasi-smooth derived manifold to solve the transversality issues in the definition of this moduli stack). The direct algebro-geometric analog of a cobordism is thus an $\A^1$-path in the moduli stack of proper quasi-smooth schemes. As in topology, one can also consider moduli stacks of schemes with some stable tangential structure (see \ssecref{ssec:intro-thom}).
The $\A^1$-localization of such a moduli stack is then analogous to a \emph{cobordism space} of structured manifolds (a space in which points are manifolds, paths are cobordisms, homotopies between paths are cobordisms between cobordisms, etc.).

From this perspective, our computation of $\Omega^\infty_\T M\beta$ for $\beta$ of rank $0$ is similar to the following computation of Galatius, Madsen, Tillmann, and Weiss in topology \cite{GMTW}: given a morphism of spaces $\beta\colon B\to \mathrm{BO}$, the infinite loop space of the Thom spectrum $M\beta$ is the cobordism space of zero-dimensional compact smooth manifolds with $\beta$-structured stable normal bundle. The only essential difference is that in topology these cobordism spaces are already grouplike, due to the existence of nontrivial cobordisms to the empty manifold. In $\FSYN_k^\beta$, however, the empty scheme is not $\A^1$-homotopic to any nonempty scheme, so group completion is necessary.

To our knowledge, when $\beta$ is of positive rank, the topological analog of our computation is not recorded in the literature. One expects for example an identification, for $M$ a smooth manifold and $n>0$, of the mapping space $\Maps(M,\Omega^\infty\Sigma^{n} \mathrm{MU})$ with the cobordism space of complex-oriented submanifolds of $M$ of codimension $n$.
 In the other direction, it is an interesting problem to extend our computation of $\Omega^\infty_\T M\beta$ to $\beta$ of \emph{negative} rank, where the topological story suggests a relationship with a moduli stack of proper quasi-smooth schemes of positive dimension. 
 
 Finally, we note that our description of $\MGL$-modules can be understood as a more coherent version of Quillen's geometric universal property of $\mathrm{MU}$ \cite[Proposition 1.10]{Quillen:1971}.

\ssec{Related work}
Similar computations of motivic Thom spectra in terms of framed correspondences were obtained independently by Garkusha and Neshitov \cite{GarkushaNeshitov}. 
Our approach differs in that we work with tangentially framed correspondences as defined in \cite{EHKSY1} rather than framed correspondences in the sense of Voevodsky (which are called equationally framed correspondences in \cite{EHKSY1}). 
Because tangentially framed correspondences are the morphisms in a symmetric monoidal $\infty$-category, we are able to make much more structured computations, which are crucial for describing $\infty$-categories of modules over motivic Thom spectra.
Our notion of motivic Thom spectrum is also strictly more general than the one in \cite{GarkushaNeshitov}.

The fact that $\MGL$-cohomology groups admit finite syntomic transfers is well known (see \cite{PaninOrientedII} for the case of finite transfers between smooth schemes and \cite{DegliseOrientation} for the general case). They are also an essential feature in the algebraic cobordism theory of Levine and Morel \cite{Levine:2007}. Such transfers were further constructed by Navarro \cite{NavarroGysin} on $E$-cohomology groups for any $\MGL$-module $E$, and Déglise, Jin, and Khan \cite{DJKFundamental} showed that these transfers exist at the level of spaces. Our main result implies that $E$-cohomology spaces admit \emph{coherent} finite syntomic transfers, and that this structure even characterizes $\MGL$-modules. 

In \cite{LowreySchurg}, Lowrey and Schürg give a presentation of the algebraic cobordism groups $\Omega^n(X)\simeq \MGL^{2n,n}(X)$ with projective quasi-smooth derived $X$-schemes as generators (for $X$ smooth over a field of characteristic zero). Our results give a comparable presentation for $n\geq 0$ of the whole $\infty$-groupoid $(\Omega^\infty_\T\Sigma^{n}_\T\MGL)(X)$, which holds also in positive characteristic but is only Zariski-local. We hope that there is a common generalization of both results, namely, a global description of the sheaves $\Omega^\infty_\T\Sigma^{n}_\T\MGL$, for all $n\in \Z$ and in arbitrary characteristic, in terms of quasi-smooth derived schemes.

\ssec{Conventions and notation}
This paper is a continuation of \cite{EHKSY1}, and we use the same notation as in \emph{op.\ cit.} In particular, $\Pre(C)$ denotes the $\infty$-category of presheaves on an $\infty$-category $C$, $\Pre_\Sigma(C)\subset \Pre(C)$ is the full subcategory of presheaves that transform finite sums into finite products (called \emph{$\Sigma$-presheaves} for short), and $L_\Sigma$ is the associated localization functor. 
We denote by $L_\zar$ and $L_\nis$ the Zariski and Nisnevich sheafification functors, by $\Lhtp$ the (naive) $\A^1$-localization functor, and by $L_\mot$ the motivic localization functor.

In addition, we use derived algebraic geometry throughout the paper, following \cite{LurieThesis}, \cite[Chapter 2.2]{HAG2}, and \cite[Chapter 25]{SAG}. By a \emph{derived commutative ring} we mean an object of $\CAlg^\Delta=\Pre_\Sigma(\mathrm{Poly})$, where $\mathrm{Poly}$ is the category of polynomial rings $\Z[x_1,\dotsc,x_n]$ for $n\geq 0$; these are often called simplicial commutative rings in the literature, for historical reasons that will not be relevant here. Given a derived commutative ring $R$, we denote by $\CAlg_R^\Delta$ the $\infty$-category $(\CAlg^\Delta)_{R/}$ and by $\CAlg_R^\sm\subset \CAlg_R^\Delta$ the full subcategory of smooth $R$-algebras. If $R$ is discrete, we further denote by $\CAlg_R^\heart\subset\CAlg_R^\Delta$ the full subcategory of discrete $R$-algebras, which is a 1-category containing $\CAlg_R^\sm$.

We write $\dAff$ for the $\infty$-category of derived affine schemes and $\dSch$ for that of derived schemes.
If $X$ is a derived scheme, we denote by $X_\cl$ its underlying classical scheme. Every morphism $f\colon Y\to X$ in $\dSch$ admits a cotangent complex $\sL_f\in \QCoh(Y)$ \cite[\sectsign 25.3]{SAG}. We say that $f$ is \emph{quasi-smooth} if it is locally of finite presentation and $\sL_f$ is perfect of Tor-amplitude $\leq 1$; the rank of $\sL_f$ is called the \emph{relative virtual dimension} of $f$.

Throughout this paper, $S$ denotes a fixed base scheme, arbitrary unless otherwise specified.

\ssec{Acknowledgments}
We would like to thank Andrei Druzhinin, Joachim Jelisiejew, Jacob Lurie, Akhil Mathew, Fabien Morel, and David Rydh for many useful discussions that helped us at various stages in the writing of this paper.

We would also like to express our gratitude to Grigory Garkusha, Ivan Panin, Alexey Ananyevskiy, and Alexander Neshitov, who realized Voevodsky's ideas about framed correspondences in a series of groundbreaking articles. Our present work would not have been possible without theirs.

The final stages of this work were supported by the National Science Foundation under grant DMS-1440140, while the first two authors were in residence at the Mathematical Sciences Research Institute in Berkeley, California, during the ``Derived Algebraic Geometry'' program in spring 2019.

Finally, this paper was completed during a two-week stay at the Institute for Advanced Study in Princeton in July 2019. We thank the Institute for excellent working conditions.

\section{Twisted framed correspondences}
\label{sec:twisted-frames}

Recall that a framed correspondence from $X$ to $Y$ is a span
\[
\begin{tikzcd}
   & Z \ar[swap]{ld}{f}\ar{rd}{g} & \\
  X &   & Y
\end{tikzcd}
\]
where $f$ is finite syntomic, together with an equivalence $\sL_f\simeq 0$ in $K(Z)$ \cite[Definition 2.3.4]{EHKSY1}. Given $\xi\in K(Y)$ of rank $0$, one can consider a ``twisted'' version of this definition by instead requiring an equivalence $\sL_f+ g^*(\xi)\simeq 0$ in $K(Z)$. In this section, we study this notion of twisted framed correspondence, which we define more generally for $\xi\in K(Y)$ of rank $\geq 0$ using derived algebraic geometry. The connection with Thom spectra in motivic homotopy theory will be made in Section~\ref{sec:thom}.

We begin in \ssecref{ssec:sVect} with some recollections about the relationship between stable vector bundles and K-theory. In \ssecref{ssec:twisted-frames}, we introduce the presheaves $\h^\fr_S(Y,\xi)$, $\h_S^\nfr(Y,\xi)$, and $\h_S^\efr(Y,\xi)$ of $\xi$-twisted tangentially framed, normally framed, and equationally framed correspondences, and in \ssecref{ssec:descent} we prove some of their basic properties. In \ssecref{ssec:fr-comparison}, we show that when $Y$ is smooth the presheaves $\h^\fr_S(Y,\xi)$, $\h_S^\nfr(Y,\xi)$, and $\h_S^\efr(Y,\xi)$ are motivically equivalent. Most proofs in \ssecref{ssec:descent} and \ssecref{ssec:fr-comparison} are almost identical to the proofs of the analogous results for untwisted framed correspondences in \cite[Section 2]{EHKSY1}. Finally, in \ssecref{ssec:base-change}, we prove that the restriction of the presheaf $\h^\fr_S(Y,\xi)$ to smooth $S$-schemes is compatible, up to motivic equivalence, with any base change $S'\to S$; this is a key technical result that will allow us to establish our main results over arbitrary base schemes (rather than just over fields).

\ssec{Stable vector bundles and K-theory}
\label{ssec:sVect}

For $X$ a derived scheme, we denote by $\Vect(X)$ the $\infty$-groupoid of finite locally free sheaves on $X$ \cite[\sectsign 2.9.3]{SAG}.
We define the $\infty$-groupoid $\sVect(X)$ as the colimit of the sequence
\[
\Vect(X)\xrightarrow{\oplus\sO_X} \Vect(X) \xrightarrow{\oplus\sO_X} \Vect(X) \to \dotsb.
\]
Thus, an object of $\sVect(X)$ is a pair $(\sE,m)$ where $\sE\in\Vect(X)$ and $m\geq 0$, which we can think of as the formal difference $\sE-\sO^m_X$; the \emph{rank} of such a pair is $\rk(\sE)-m$. Two pairs $(\sE,m)$ and $(\sE',m')$ are equivalent if they have the same rank and there exists $n\geq 0$ such that $\sE\oplus \sO^{m'}\oplus \sO^n \simeq \sE'\oplus \sO^m\oplus\sO^n$.

Note that $\sVect(X)$ is a $(\Vect(X),\oplus)$-module and the canonical map $\Vect(X)\to K(X)$ factors through $\sVect(X)$.

For $R$ a derived commutative ring, $\sVect(R)$ is noncanonically a sum of copies of $B\GL(R)$. Indeed, if $\xi=(E,m)\in\sVect(R)$, the space of automorphisms of $\xi$ is
\[
\Omega_\xi(\sVect(R)) \simeq \colim_n \Aut_R(E\oplus R^{n}).
\]
If we choose an equivalence $E\oplus F\simeq R^r$, computing the colimit of the sequence
\[
\Aut_R(E) \to \Aut_R(E\oplus F) \to \Aut_R(E\oplus F\oplus E) \to \dotsb
\]
in two ways yields an equivalence $\colim_n \Aut_R(E\oplus R^{n})\simeq \colim_n \Aut_R(R^n) =\GL(R)$.

Recall that a morphism of spaces is \emph{acyclic} if its fibers have no reduced homology. We refer to \cite{Raptis} for a review of the main properties of acyclic morphisms. In particular, by \cite[Theorem 3.3]{Raptis}, the class of acyclic morphisms is closed under colimits and base change, and if $X$ is a space whose fundamental groups have no nontrivial perfect subgroups, then every acyclic morphism $X\to Y$ is an equivalence.

\begin{lem}\label{lem:quillen}
	Let $R$ be a derived commutative ring. Then the canonical map $\sVect(R) \to K(R)$ is a plus construction in the sense of Quillen, i.e., it is the universal map that kills the commutator subgroup of $\pi_1(\sVect(R), \xi)\simeq \GL(\pi_0R)$ for all $\xi\in\sVect(R)$. In particular, it is acyclic.
\end{lem}

\begin{proof}
	This is an instance of the McDuff–Segal group completion theorem. We refer to \cite[Theorem 1.1]{ORW} or \cite[Theorem 9]{Nikolaus} for modern treatments.
\end{proof}

\begin{rem}\label{rem:KSL-plus}
	Similarly, if $\Vect^\SL(R)$ denotes the monoidal groupoid of finite locally free $R$-modules with trivialized determinant and $K^\SL(R)$ is its group completion, then $\sVect^\SL(R) \to K^\SL(R)$ is a plus construction.
	Here, the group completion theorem \cite[Theorem 1.1]{ORW} does not directly apply because $\Vect^\SL(R)$ is not homotopy commutative. However, we can apply \cite[Proposition 3.1]{deloop4} by viewing $\Vect^\SL(R)$ as a module over its subgroupoid of bundles of even rank, which is an $\Einfty$-monoid.
\end{rem}

\ssec{Presheaves of twisted framed correspondences}
\label{ssec:twisted-frames}

Let $X$ and $Y$ be derived $S$-schemes and let $\xi\in K(Y)$ be of rank $\geq 0$.
The $\infty$-groupoid of \emph{$\xi$-twisted framed correspondences} from $X$ to $Y$ over $S$ is defined by:
\[
\h^{\fr}_S(Y,\xi)(X) = \left\{
\begin{tikzcd}
   & Z \ar[swap]{ld}{f}\ar{rd}{g} & \\
  X &   & Y
\end{tikzcd}+\quad \parbox{3.5cm}{$f$ finite quasi-smooth\\$\sL_f\simeq -g^*(\xi)$ in $K(Z)$}
\right\}.
\]
When $\xi=0$ and $X,Y\in\Sch_S$, this definition recovers the notion of tangentially framed correspondence considered in \cite[\sectsign 2.1]{EHKSY1}, by the following lemma:

\begin{lem}\label{lem:classical}
	Let $f\colon Z\to X$ be a quasi-finite quasi-smooth morphism of derived schemes of relative virtual dimension $0$.
	Then $f$ is flat. In particular, if $X$ is classical, then $Z$ is classical.
\end{lem}

\begin{proof}
	By \cite[Theorem 7.2.2.15]{HA}, $f$ is flat if and only if $Z_\cl\simeq Z\times_XX_\cl$ and $f_\cl$ is flat, so we can assume $X$ classical.
	The question being local on $Z$, we can assume that $X$ is affine and that $Z$ is a subscheme of $\A^n_X$.
	Since $Z$ is quasi-smooth of relative virtual dimension $0$ over $X$, it is cut out by exactly $n$ equations locally on $\A^n_X$ \cite[Proposition 2.3.8]{KhanVCD}. Thus, we may assume that $Z$ is cut out by $n$ equations in $(\A^n_X)_a$ for some $a\in \sO(\A^n_X)$, hence cut out by $n+1$ equations in $\A^{n+1}_X$. Since $f$ is quasi-finite, this implies that $f_\cl\colon Z_\cl\to X$ is a relative global complete intersection in the sense of \cite[Tag 00SP]{stacks}, hence syntomic \cite[Tag 00SW]{stacks}. In particular, the defining equations of $Z$ in $\A^{n+1}_X$ locally form a regular sequence, whence $Z_\cl=Z$ \cite[Example 2.3.2]{KhanVCD}.
\end{proof}

Note that $\h_S^\fr(Y,\xi)$ is a presheaf on $\dSch_S$. In fact, it is a $\Sigma$-presheaf on the $\infty$-category $\Span^\fr(\dSch_S)$ of framed correspondences (see Appendix~\ref{app:category}). It is also covariantly functorial in the pair $(Y,\xi)$.

To relate $\xi$-twisted framed correspondences with motivic homotopy theory, we will also need to study two auxiliary versions of twisted framed correspondences, namely an ``equationally framed'' version and a ``normally framed'' version.

Let $X,Y\in\dSch_S$ and let $\xi=(\sE,m)\in \sVect(Y)$ be of rank $\geq 0$.
We define:
\[
\h^{\efr}_S(Y,\xi)(X) =\colim_{n\to\infty} \left\{
\begin{tikzcd}
   & Z \ar[swap]{ld}{f} \ar{rd}{g} & \\
  X &   & Y
\end{tikzcd}+\quad \parbox{5.7cm}{$f$ finite\\$i\colon Z\to \A^n_X$ closed immersion over $X$\\$\phi\colon (\A^n_X)_Z^h\to \A^{n-m}\times \bV(\sE)$\\$\phi^{-1}(0 \times Y)\simeq Z$, $\phi|_{Z}\simeq g$}
\right\},
\]
\[
\h^{\nfr}_S(Y,\xi)(X) =\colim_{n\to\infty} \left\{
\begin{tikzcd}
   & Z \ar[swap]{ld}{f}\ar{rd}{g} & \\
  X &   & Y
\end{tikzcd}+\quad \parbox{5.7cm}{$f$ finite quasi-smooth\\$i\colon Z\to \A^n_X$ closed immersion over $X$\\$\sN_i\simeq \sO_Z^{n-m}\oplus g^*(\sE)$}
\right\}.
\]
More precisely, for each $m$, the right-hand sides are functors $\Vect(Y)_{\geq m} \to \Pre(\dSch_S)$ in the variable $\sE$. As $m$ varies, these functors fit together in a cone over the sequence
\[
\Vect(Y)\xrightarrow{\oplus\sO_Y} \Vect_{\geq 1}(Y) \xrightarrow{\oplus\sO_Y} \Vect_{\geq 2}(Y) \to \dotsb,
\]
which induces $\sVect_{\geq 0}(Y)\to \Pre(\dSch_S)$. Here, the notation $X_Z^h$ for $X$ a derived scheme and $Z\subset X$ a closed subset refers to the pro-object of étale neighborhoods of $Z$ in $X$, see \cite[A.1.1]{EHKSY1}, and $\bV(\sE)=\Spec(\Sym(\sE))$ is the vector bundle over $Y$ associated with $\sE$.

For $\xi=0$ and $X,Y\in\Sch_S$, these definitions recover the notions of equationally framed and normally framed correspondences from \cite[Section 2]{EHKSY1} (using Lemma~\ref{lem:classical} for the latter). Unlike in \emph{op.\ cit.}, we will not discuss the ``level $n$'' versions of $\h^\efr$ and $\h^\nfr$, for simplicity. However, it is clear that many of the results below hold at finite level (a notable exception is Proposition~\ref{prop:additivity}).

There are evident forgetful maps
\[
\h^{\efr}_S(Y,\xi)\longrightarrow \h^{\nfr}_S(Y,\xi) \longrightarrow \h^{\fr}_S(Y,\xi)
\]
in $\Pre_\Sigma(\dSch_S)$, similar to the case $\xi=0$ considered in \cite[\sectsign 2]{EHKSY1} (see also the discussion before Proposition~\ref{prop:nfr-to-dfr} for more details on the second morphism).
Note that $\h^{\efr}_S(Y,\xi)(X)$ is discrete (i.e., a set) when $X$ and $Y$ are classical, since $\phi$ determines $g$ as well as the derived closed subscheme $Z$. On the other hand, $\h^{\nfr}_S(Y,\xi)(X)$ is usually not discrete when $\rk \xi\geq 1$.

\begin{rem}\label{rem:voev-lemma}
	Let $X,Y\in\dSch_S$ and $\xi=(\sE,m)\in \sVect_{\geq 0}(Y)$. Then there is a natural equivalence
	\[
	\h^\efr_S(Y,\xi)(X) \simeq \colim_{n\to\infty} \Maps\left(X_+\wedge (\P^1)^{\wedge n},L_\nis\left(\frac{\bV(\sO_Y^{n-m}\oplus\sE)}{\bV(\sO_Y^{n-m}\oplus\sE)-0}\right)\right).
	\]
	If $X$ and $Y$ are classical this is a special case of Voevodsky's Lemma \cite[Corollary A.1.5]{EHKSY1}, which is easily generalized to derived schemes. Indeed, as in the proof of \cite[Proposition A.1.4]{EHKSY1}, it suffices to show the following: if $Y$ is a derived scheme and $Z\subset Y$ is a closed subset, then $Y'\mapsto (Y')^h_Z$ is an étale cosheaf on étale $Y$-schemes; by the topological invariance of the étale $\infty$-topos \cite[Remark B.6.2.7]{SAG}, this follows from the underived statement proved in \cite[Proposition A.1.4]{EHKSY1}.
\end{rem}

\ssec{Descent and additivity}
\label{ssec:descent}

We say that a presheaf $\sF\colon \dSch_S^\op \to\Spc$ satisfies \emph{closed gluing} if $\sF(\emptyset)\simeq *$ and if for any diagram of closed immersions $X\hookleftarrow Z \hook Y$ in $\dSch_S$, the canonical map
\[
\sF(X\sqcup_{Z} Y) \to \sF(X)\times_{\sF(Z)}\sF(Y)
\]
is an equivalence.

We say that a presheaf $\sF\colon \dSch_S^\op\to\Spc$ is \emph{finitary} if, for every cofiltered diagram $(X_\alpha)$ of qcqs derived schemes with affine transition maps, the canonical map
\[
\colim_\alpha\sF(X_\alpha)\to \sF(\lim_\alpha X_\alpha)
\]
is an equivalence.

\begin{prop}\label{prop:fr-descent}
Let $Y\in\dSch_S$ and let $\xi\in K(Y)$ be of rank $\geq 0$.

\noindent{\em(i)}
$\h^{\fr}_S(Y,\xi)$ is a Nisnevich sheaf on qcqs derived schemes.

\noindent{\em(ii)}
If $Y$ is locally finitely presented over $S$, then $\h^{\fr}_S(Y,\xi)$ is finitary.

\noindent{\em(iii)}
Let $R\hook Y$ be a Nisnevich (resp.\ étale) covering sieve generated by a single map. Then $\h^\fr_S(R,\xi)\to \h^\fr_S(Y,\xi)$ is a Nisnevich (resp.\ étale) equivalence.
\end{prop}

\begin{proof}
	(i) and (ii) follow from the corresponding properties of algebraic K-theory (see for example \cite[Proposition A.15]{CMNN} and \cite[Lemma 7.3.5.13]{HA}, respectively).
	Let us prove (iii). Since $\h^{\fr}_S(R,\xi)\to \h^{\fr}_S(Y,\xi)$ is a monomorphism, it suffices to show that it is a Nisnevich (resp.\ étale) effective epimorphism \cite[Example 5.2.8.16]{HTT}. Refining the sieve $R$ if necessary, we can assume that it is generated by a single étale map $Y'\to Y$. 
	Given a span $X\leftarrow Z\to Y$ in $\h^{\fr}_S(Y,\xi)(X)$, we must show that the sieve on $X$ consisting of all maps $X'\to X$ such that $X'\times_XZ\to Z\to Y$ factors through $Y'$ is covering in the Nisnevich (resp.\ étale) topology. In other words, we must show that if $X_\cl$ is local and henselian (resp.\ strictly henselian), then the étale map $Z\times_YY'\to Z$ has a section. By the topological invariance of the étale $\infty$-topos \cite[Remark B.6.2.7]{SAG}, it is equivalent to show that the étale map $Z_\cl\times_YY'\to Z_\cl$ has a section.
	Since $Z$ is finite over $X$, $Z_\cl$ is a sum of henselian (resp.\ strictly henselian) local schemes \cite[Proposition 18.5.10]{EGA4-4}, whence the result.
\end{proof}

\begin{prop}\label{prop:fr-additivity}
	Let $Y_1,\dots,Y_k\in\dSch_S$, let $\xi \in K(Y_1\sqcup \dotsb\sqcup Y_k)$ have rank $\geq 0$, and let $\xi_i$ be the restriction of $\xi$ to $Y_i$.
	Then the canonical map 
	\[
		\h^{\fr}_S(Y_1\sqcup \dotsb\sqcup Y_k,\xi) \to \h^{\fr}_S(Y_1,\xi_1)\times\dotsb\times \h^{\fr}_S(Y_k,\xi_k)
	\]
	is an equivalence.
\end{prop}

\begin{proof}
	Clear.
\end{proof}

\begin{prop}\label{prop:efr-descent}
Let $Y\in\dSch_S$ and $\xi\in\sVect_{\geq 0}(Y)$.

\noindent{\em(i)}
$\h^\efr_S(Y,\xi)$ is a sheaf for the quasi-compact étale topology.

\noindent{\em(ii)}
$\h^{\efr}_S(Y,\xi)$ satisfies closed gluing.

\noindent{\em(iii)}
If $Y$ is locally finitely presented over $S$, then $\h^{\efr}_S(Y,\xi)$ is finitary.

\noindent{\em(iv)}
Let $R\hook Y$ be a Nisnevich (resp.\ étale) covering sieve generated by a single map. Then $\h^\efr_S(R,\xi)\to \h^\efr_S(Y,\xi)$ is a Nisnevich (resp.\ étale) equivalence.
\end{prop}

\begin{proof}
	The proof of each point is exactly the same as the corresponding point of \cite[Proposition 2.1.5(i)]{EHKSY1}.
\end{proof}

\begin{lem}\label{lem:closed-gluing}
	The presheaf
	\[
	\dSch_S^\op \to \Spc,\quad X\mapsto \{\text{quasi-smooth derived $X$-schemes}\}
	\]
	satisfies closed gluing.
\end{lem}

\begin{proof}
	The assertion without the quasi-smoothness condition follows from \cite[Theorem 16.2.0.1]{SAG} (as in the proof of \cite[Theorem 16.3.0.1]{SAG}, which is the spectral analog).
	It remains to prove the following: if $Y$ is a derived scheme and $f\colon Y\to X_0\sqcup_{X_{01}}X_1$ is a morphism whose base change $f_*\colon Y_*\to X_*$ is quasi-smooth for each $*\in\{0,1,01\}$, then $f$ is quasi-smooth. We have that $f$ is locally finitely presented by \cite[Proposition 16.3.2.1(3)]{SAG}. It remains to show that the cotangent complex $\sL_f$ is perfect and has Tor-amplitude $\leq 1$. Since the cotangent complex is stable under base change, the pullback of $\sL_f$ to $X_*$ is $\sL_{f_*}$. The claim now follows from \cite[Proposition 16.2.3.1(3,7)]{SAG}.
\end{proof}

\begin{prop}\label{prop:nfr-descent}
Let $Y\in\dSch_S$ and $\xi\in\sVect_{\geq 0}(Y)$.

\noindent{\em(i)}
$\h^{\nfr}_S(Y,\xi)$ satisfies Nisnevich excision.

\noindent{\em(ii)}
$\h^{\nfr}_S(Y,\xi)$ satisfies closed gluing.

\noindent{\em(iii)}
If $Y$ is locally finitely presented over $S$, then $\h^{\nfr}_S(Y,\xi)$ is finitary.

\noindent{\em(iv)}
Let $R\hook Y$ be a Nisnevich (resp.\ étale) covering sieve generated by a single map. Then $\h^\nfr_S(R,\xi)\to \h^\nfr_S(Y,\xi)$ is a Nisnevich (resp.\ étale) equivalence.
\end{prop}

\begin{proof}
	(i) By definition, the presheaf $\h^\nfr_S(Y,\xi)$ is a filtered colimit over $n$ of its level $n$ versions, which are clearly sheaves for the fpqc topology. We conclude since the property of Nisnevich excision is preserved by filtered colimits.
	
	(ii) This follows from the closed gluing property for connective quasi-coherent sheaves \cite[Theorem 16.2.0.1]{SAG} and Lemma~\ref{lem:closed-gluing}.
	
	(iii) Clear.
	
	(iv) The proof is identical to that of Proposition~\ref{prop:fr-descent}(iii).
\end{proof}

\begin{prop}
	\label{prop:additivity}
	Let $Y_1,\dots,Y_k\in\dSch_S$, let $\xi\in\sVect(Y_1\sqcup\dotsb\sqcup Y_k)$ be of rank $\geq 0$, and let $\xi_i$ be the restriction of $\xi$ to $Y_i$.
	Then the canonical maps 
	\begin{gather*}
		\h_S^{\efr}(Y_1\sqcup \dotsb\sqcup Y_k,\xi) \to \h_S^{\efr}(Y_1,\xi_1)\times\dotsb\times \h_S^{\efr}(Y_k,\xi_k) \\
		\h_S^{\nfr}(Y_1\sqcup \dotsb\sqcup Y_k,\xi) \to \h_S^{\nfr}(Y_1,\xi_1)\times\dotsb\times \h_S^{\nfr}(Y_k,\xi_k)
	\end{gather*}
	are $\A^1$-equivalences.
\end{prop}

\begin{proof}
	Same as the proof of \cite[Proposition 2.2.11]{EHKSY1}.
\end{proof}

It follows from Proposition~\ref{prop:additivity} that the presheaves $\Lhtp\h_S^{\efr}(Y,\xi)$ and $\Lhtp\h_S^{\nfr}(Y,\xi)$ have canonical structures of $\Einfty$-objects (cf.\ \cite[2.2.9]{EHKSY1}).

\ssec{Comparison theorems}
\label{ssec:fr-comparison}

If $f\colon A\to B$ is a morphism of derived commutative rings, there is a canonical $B$-linear map
\[
\epsilon_f\colon \cofib(f)\otimes_A B \to \sL_f,
\]
called the \emph{Hurewicz map} associated with $f$ \cite[\sectsign 25.3.6]{SAG}.
If $f$ is connective (i.e., if $\Spec(f)$ is a closed immersion), then $\epsilon_f$ is $2$-connective \cite[Proposition 25.3.6.1]{SAG}. In particular, if $i\colon Z\to X$ is a closed immersion between derived affine schemes and $\sI$ is the fiber of $i^*\colon \sO(X)\to \sO(Z)$, we have a $1$-connective map
\[
\sI \otimes_{\sO(X)}\sO(Z) \to \sL_i[-1]=\sN_i.
\]

\begin{lem}\label{lem:connective}
Let
\[
\begin{tikzcd}
	Z_0 \ar{r}{i_0} \ar{d} & X_0 \ar{d} \\
	Z \ar{r}{i} & X
\end{tikzcd}
\]
be a Cartesian square  of derived affine schemes, where all arrows are closed immersions. Let $\sI$ and $\sI_0$ be the fibers of the maps $i^*\colon \sO(X)\to \sO(Z)$ and $i_0^*\colon \sO(X_0)\to \sO(Z_0)$. Then the canonical map
\[
\sI\to \sI_0\times_{\sN_{i_0}}\sN_{i}
\]
is connective.
\end{lem}

\begin{proof}
	We factor this map as follows:
	\[
	\sI\xrightarrow{\alpha} \sI_0 \times_{\sI_0\otimes_{\sO(X_0)}\sO(Z_0)}( \sI\otimes_{\sO(X)}\sO(Z))\xrightarrow{\beta} \sI_0\times_{\sN_{i_0}}\sN_{i}
	\]
	By \cite[Proposition 25.3.6.1]{SAG}, the canonical maps $\sI\otimes_{\sO(X)}\sO(Z) \to \sN_i$ and $\sI_0\otimes_{\sO(X_0)}\sO(Z_0) \to \sN_{i_0}$ are $1$-connective. Being a base change of the latter, the projection $(\sI_0\otimes_{\sO(X_0)}\sO(Z_0))\times_{\sN_{i_0}}\sN_i\to \sN_i$ is $1$-connective. It follows that the map
	\[
	\sI\otimes_{\sO(X)}\sO(Z) \to (\sI_0\otimes_{\sO(X_0)}\sO(Z_0))\times_{\sN_{i_0}}\sN_i
	\]
	is connective, hence that its base change $\beta$ is connective.
	It remains to show that $\alpha$ is connective. Let $T$ be the closed subscheme of $X$ obtained by gluing $Z$ and $X_0$ along $Z_0$. We can factor $\alpha$ as
	\[
	\sI\to \sI\otimes_{\sO(X)}\sO(T) \to \sI_0 \times_{\sI_0\otimes_{\sO(X_0)}\sO(Z_0)}( \sI\otimes_{\sO(X)}\sO(Z)).
	\]
	The first map is clearly connective.
	Since $\sI_0\simeq \sI\otimes_{\sO(X)}\sO(X_0)$, the second map is an equivalence by Milnor patching \cite[Theorem 16.2.0.1]{SAG}. Thus, $\alpha$ is connective.
\end{proof}

\begin{prop}\label{prop:extension}
Let $X,Y\in\dSch_{S}$, let $\xi\in \sVect(Y)$ be of rank $\geq 0$, and let $X_0\subset X$ be a closed subscheme. 
Suppose that $X$ is affine and that $Y$ admits an étale map to an affine bundle over $S$. Then the map
\[
\h^{\efr}_S(Y,\xi)(X) \to \h^{\efr}_S(Y,\xi)(X_0) \times_{\h^{\nfr}_S(Y,\xi)(X_0)} \h^{\nfr}_S(Y,\xi)(X)
\]
is an effective epimorphism (i.e., surjective on $\pi_0$).
\end{prop}

\begin{proof}
	Write $\xi=(\sE,m)$.
	An element in the right-hand side consists of:
	\begin{itemize}
		\item a span $X\xleftarrow{f} Z\xrightarrow{g} Y$ with $f$ finite quasi-smooth;
		\item a closed immersion $i\colon Z\to \A^n_X$ over $X$ with an equivalence $\tau\colon \sN_i\simeq \sO_Z^{n-m}\oplus g^*(\sE)$;
		\item an equational $\xi$-framing of the induced span $X_0\xleftarrow{f_0} Z_0\xrightarrow{g_0} Y$:
		\[
		Z_0\stackrel{i_0}\longrightarrow \A^n_{X_0} \leftarrow U_0 \xrightarrow{\varphi_0} \bV(\sO_Y^{n-m}\oplus \sE), \quad \alpha_0\colon Z_0\simeq \varphi_0^{-1}(Y),
		\]
		where $i_0$ is the pullback of $i$, $U_0\to \A^n_{X_0}$ is an affine étale neighborhood of $Z_0$, and $\phi_0$ extends $g_0$;
		\item an identification of the equivalence $\sN_{i_0}\simeq \sO_{Z_0}^{n-m}\oplus g_0^*(\sE)$ induced by $\tau$ with that induced by $\alpha_0$.
	\end{itemize}
	The goal is to construct an equational $\xi$-framing $(\phi,\alpha)$ of $X\xleftarrow{f} Z\xrightarrow{g} Y$ that simultaneously induces the normal framing $\tau$ and the equational $\xi$-framing $(\phi_0,\alpha_0)$.
	Using \cite[Lemma A.2.3]{EHKSY1}, we can lift the étale neighborhood $U_0$ of $Z_0$ in $\A^n_{X_0}$ to an étale neighborhood $U$ of $Z$ in $\A^n_X$. 
	Refining $U_0$ if necessary, we can assume that $U$ is affine (by \cite[Lemma A.1.2(ii)]{EHKSY1}).
	
	Let $h_0\colon U_0\to Y$ be the composition of $\varphi_0\colon U_0\to \bV(\sO_Y^{n-m}\oplus \sE)$ and the projection $\bV(\sO_Y^{n-m}\oplus \sE)\to Y$.
	We first construct a simultaneous extension $h\colon U\to Y$ of $h_0\colon U_0\to Y$ and $g\colon Z\to Y$. Suppose first that $Y$ is an affine bundle over $S$. Since $U$ is affine, $U\times_SY\to U$ is a vector bundle over $U$. It follows that the restriction map
		\[
		\Maps_S(U,Y) \to \Maps_S(U_0,Y) \times_{\Maps_S(Z_0,Y)} \Maps_S(Z,Y) \simeq \Maps_S(U_0\coprod_{Z_0}Z,Y)
		\]
		is surjective, so the desired extension exists. In general, let $p\colon Y\to A$ be an étale map where $A$ is an affine bundle over $S$. By the previous case, there exists an $S$-morphism $U\to A$ extending $p\circ h_0$ and $p\circ g$. Then the étale map $U\times_{A}Y\to U$ has a section over $U_0\coprod_{Z_0}Z$, so there exists an affine open subset $U'\subset U\times_{A}Y$ that is an étale neighborhood of $U_0\coprod_{Z_0}Z$ in $U$. We can therefore replace $U$ by $U'$, and the projection $U'\to Y$ gives the desired extension.
	
	It remains to construct a $Y$-morphism $\phi\colon U\to\bV(\sO_Y^{n-m}\oplus \sE)$ extending $\phi_0$ and an equivalence $\alpha\colon Z\simeq \phi^{-1}(Y)$ lifting $\alpha_0$ such that the induced trivialization $\sN_i\simeq \sO_Z^{n-m}\oplus g^*(\sE)$ is equivalent to $\tau$. Recall that $Y$-morphisms $U\to \bV(\sO_Y^{n-m}\oplus \sE)$ correspond to morphisms of $\sO_U$-modules $\sO_U^{n-m}\oplus h^*(\sE) \to \sO_U$.
	Let $\sI$ and $\sI_0$ be the fibers of the restrictions map $\sO(U)\to \sO(Z)$ and $\sO(U_0)\to \sO(Z_0)$.
	By Lemma~\ref{lem:connective}, the morphism of $\sO(U)$-modules
	\[
	\sI \to \sI_0\times_{\sN_{i_0}} \sN_i
	\]
	is connective.
	Since $\sO(U)^{n-m}\oplus h^*(\sE)$ is a projective object in connective $\sO(U)$-modules \cite[Proposition 7.2.2.7]{HA}, the morphism $\sO(U)^{n-m}\oplus h^*(\sE) \to \sI_0\times_{\sN_{i_0}} \sN_i$ induced by $\alpha_0$ and $\tau$ lifts to a morphism
	\[
	\sO(U)^{n-m}\oplus h^*(\sE) \to \sI.
	\]
	This defines a $Y$-morphism $\phi\colon U\to \bV(\sO^{n-m}_Y\oplus \sE)$ extending $\phi_0$ together with a factorization of $Z\to U$ through $\phi^{-1}(Y)$, i.e., a $U$-morphism $\alpha\colon Z\to \phi^{-1}(Y)$. By construction, $\alpha$ lifts $\alpha_0$ and induces the equivalence $\tau$ on conormal sheaves; since both $Z$ and $\phi^{-1}(Y)$ are regularly embedded in $U$, $\alpha$ is an étale closed immersion. Thus, there exists a function $a$ on $U$ such that $\alpha$ induces an equivalence $Z\simeq \phi^{-1}(Y)\cap U_a$. Replacing $U$ by $U_a$ concludes the proof.
\end{proof}

\begin{cor}\label{cor:efr-vs-nfr}
	Suppose that $Y\in\Sm_{S}$ is a finite sum of schemes admitting étale maps to affine bundles over $S$ and let $\xi\in\sVect_{\geq 0}(Y)$. Then the map
	\[
	\Lhtp \h^\efr_S(Y,\xi) \to \Lhtp \h^\nfr_S(Y,\xi)
	\]
	is an equivalence on derived affine schemes. In particular, it induces an equivalence
	\[
	L_\zar\Lhtp \h^\efr_S(Y,\xi) \simeq L_\zar\Lhtp \h^\nfr_S(Y,\xi).
	\]
\end{cor}

\begin{proof}
	By Proposition~\ref{prop:additivity}, we can assume that $Y$ admits an étale map to an affine bundle over $S$.
	By Proposition~\ref{prop:extension}, for every $n\geq 0$, the map
	\[
	\h^\efr_S(Y,\xi)^{\A^n} \to \h^\efr_S(Y,\xi)^{\partial\A^n} \times_{\h^\nfr_S(Y,\xi)^{\partial\A^n}} \h^\nfr_S(Y,\xi)^{\A^n}
	\]
	is surjective on affines. 
	By Propositions~\ref{prop:efr-descent}(ii) and~\ref{prop:nfr-descent}(ii), both $\h^\efr_S(Y,\xi)$ and $\h^\nfr_S(Y,\xi)$ satisfy closed gluing. It follows that the map
	\[
	\h^\efr_S(Y,\xi)^{\A^\bullet} \to \h^\nfr_S(Y,\xi)^{\A^\bullet}
	\]
	is a trivial Kan fibration of simplicial spaces when evaluated on any affine scheme, and we conclude using \cite[Theorem A.5.3.1]{SAG}.
\end{proof}

\begin{cor}\label{cor:efr-vs-nfr2}
	Let $Y$ be a smooth $S$-scheme and $\xi\in\sVect(Y)$ of rank $\geq 0$. Then the map
	\[
	\h_S^\efr(Y,\xi) \to \h_S^\nfr(Y,\xi)
	\]
	in $\Pre_\Sigma(\dSch_S)$ is a motivic equivalence.
\end{cor}

\begin{proof}
	The scheme $Y$ is the filtered union of its quasi-compact open subschemes, and on quasi-compact derived schemes the presheaves $\h_S^\efr(Y,\xi)$ and $\h_S^\nfr(Y,\xi)$ are the filtered colimits of the corresponding subpresheaves, so we can assume $Y$ quasi-compact. Let $\{U_1,\dotsc,U_k\}$ be an open cover of $Y$ by $S$-schemes admitting étale maps to affine bundles over $S$ \cite[Tag 054L]{stacks}. The map $U_1\sqcup\dotsb\sqcup U_k\to Y$ is a Zariski covering map; by Propositions~\ref{prop:efr-descent}(iv) and \ref{prop:nfr-descent}(iv), $L_\nis\h_S^\efr(-,\xi)$ and $L_\nis\h_S^\nfr(-,\xi)$ preserve the colimit of its Čech nerve. Thus, we can assume that $Y$ is a finite sum of schemes admitting étale maps to affine bundles over $S$. Then the claim follows from Corollary~\ref{cor:efr-vs-nfr}.
\end{proof}

For $Z\to X$ a finite morphism of derived schemes, we denote by $\Emb_X(Z,\A^n_X)$ the space of closed immersions $Z\to \A^n_X$ over $X$ (note that this is not a discrete space in general, because closed immersions of derived schemes are not monomorphisms). We let
\[
\Emb_X(Z,\A^\infty_X) = \colim_{n\to\infty} \Emb_X(Z,\A^n_X).
\]

\begin{prop}\label{prop:Emb-lifting}
	Let $X$ be a derived affine scheme, $Z\to X$ a finite morphism, $X_0\to X$ a closed immersion, and $Z_0=Z\times_XX_0$.
	Suppose that $(X_0)_\cl \to X_\cl$ is finitely presented.
	Then the pullback map
	\[
	\Emb_{X}(Z, \A^\infty_{X}) \to \Emb_{X_0}(Z_0, \A^\infty_{X_0})
	\]
	is an effective epimorphism (i.e., surjective on $\pi_0$).
\end{prop}

\begin{proof}
	Let $i_0\colon Z_0\to \A^n_{X_0}$ be a closed immersion over $X_0$, given by $n$ functions $g_1$, \dots, $g_n$ on $Z_0$. Let $g_1'$, \dots, $g_n'$ be lifts of these functions to $Z$.
	Note that
	\[
	\fib(\sO(Z)\to \sO(Z_0))\simeq \fib(\sO(X)\to \sO(X_0))\otimes_{\sO(X)}\sO(Z).
	\]
	Since $Z\to X$ is finite and $(X_0)_\cl\to X_\cl$ is finitely presented, the $\sO(X)$-module $\pi_0\fib(\sO(Z)\to \sO(Z_0))$ is finitely generated; let $h_1,\dotsc,h_m\in \sO(Z)$ be the images of a finite set of generators. Then the $n+m$ functions $g_1'$, \dots, $g_n'$, $h_1$, \dots, $h_m$ define a closed immersion $i\colon Z\to \A^{n+m}_X$ over $X$ whose pullback to $X_0$ is equivalent to $i_0$ in $\Emb_X(Z,\A^\infty_X)$.
\end{proof}

\begin{cor}\label{cor:Emb-contractible}
	Let $X$ be a derived affine scheme and $Z\to X$ a finite morphism. Then the presheaf
	\[
	\dAff_X^\op \to \Spc,\quad X'\mapsto \Emb_{X'}(Z\times_XX', \A^\infty_{X'})
	\]
	is $\A^1$-contractible.
\end{cor}

\begin{proof}
	This follows from Proposition~\ref{prop:Emb-lifting} as in \cite[Lemma 2.3.22]{EHKSY1}.
\end{proof}

Let $Z\to X$ be a finite quasi-smooth morphism of derived schemes of relative virtual codimension $c$. Any closed immersion $i\colon Z\to \A^n_X$ over $X$ is then quasi-smooth, hence has a finite locally free conormal sheaf $\sN_i=\sL_i[-1]$. Thus, we have a morphism
\[
\Emb_X(Z,\A^n_X) \to \Vect_{n+c}(Z), \quad i\mapsto \sN_i.
\]
Taking the colimit over $n$, we get a morphism
\[
\Emb_X(Z,\A^\infty_X) \to \sVect_c(Z)\subset \sVect(Z).
\]
We denote by $\Emb_X^\xi(Z,\A^\infty_X)$ its fiber over $\xi\in\sVect(Z)$. Note that there is a commutative square
\[
\begin{tikzcd}
	\Emb_X(Z,\A^\infty_X) \ar{r} \ar{d} & \sVect(Z) \ar{d} \\
	* \ar{r}{-\sL_f} & K(Z),
\end{tikzcd}
\]
inducing a canonical map
\[
\Emb_X^\xi(Z,\A^\infty_X) \to \Maps_{K(Z)}(\xi,-\sL_f)
\]
on the horizontal fibers over $\xi$.

\begin{prop}\label{prop:nfr-to-dfr}
	Let $f\colon Z\to X$ be a finite quasi-smooth morphism of derived affine schemes and let $\xi\in\sVect(Z)$.
	Then the morphism of simplicial spaces
	\[
	\Emb^{\xi}_{X\times\A^\bullet}(Z\times\A^\bullet,\A^\infty_{X\times\A^\bullet}) \to \Maps_{K(Z\times\A^\bullet)}(\xi,-\sL_f)
	\]
	induces an equivalence on geometric realizations.
\end{prop}

\begin{proof}
	Let us write $X^\bullet=X\times\A^\bullet$ and $Z^\bullet=Z\times\A^\bullet$ for simplicity.
	Recall that the given morphism comes from a natural transformation of Cartesian squares
	\[
   \begin{tikzcd}
     \Emb^\xi_{X^\bullet}(Z^\bullet,\A^\infty_{X^\bullet}) \ar{r} \ar{d}
       & \Emb_{X^\bullet}(Z^\bullet,\A^\infty_{X^\bullet}) \ar{d}
     \\
     * \ar{r}{\xi}
       & \sVect(Z^\bullet)
   \end{tikzcd}
	\longrightarrow
   \begin{tikzcd}
     \Maps_{K(Z^\bullet)}(\xi,-\sL_f) \ar{r} \ar{d}
       & * \ar{d}{-\sL_f}
     \\
     * \ar{r}{\xi}
       & K(Z^\bullet).
   \end{tikzcd}
	\]
	We consider two cases. If $[\xi]\neq [-\sL_f]$ in $\pi_0K(Z)$, then $\Maps_{K(Z)}(\xi,-\sL_f)$ is empty and the result holds trivially.
	Suppose that $[\xi]=[-\sL_f]$ in $\pi_0K(Z)$. 
	Then $-\sL_f$ lives in the connected component $K(Z)\langle \xi\rangle\subset K(Z)$ containing $\xi$.
	Since $Z$ is affine and $[\sL_f]=[\sO_Z^n]-[\sN_i]$ for any closed immersion $i\colon Z\to \A^n_X$ over $X$, the conormal sheaf $\sN_i$ is stably isomorphic to $\xi$. It follows that the map $\Emb_X(Z,\A^\infty_X)\to \sVect(Z)$ lands in the component $\sVect(Z)\langle\xi\rangle\subset \sVect(Z)$ containing $\xi$. 
	We may therefore rewrite the above Cartesian squares as follows:
	\[
   \begin{tikzcd}
     \Emb^\xi_{X^\bullet}(Z^\bullet,\A^\infty_{X^\bullet}) \ar{r} \ar{d}
       & \Emb_{X^\bullet}(Z^\bullet,\A^\infty_{X^\bullet}) \ar{d}
     \\
     * \ar{r}{\xi}
       & \sVect(Z^\bullet)\langle\xi\rangle
   \end{tikzcd}
	\longrightarrow
   \begin{tikzcd}
     \Maps_{K(Z^\bullet)}(\xi,-\sL_f) \ar{r} \ar{d}
       & * \ar{d}{-\sL_f}
     \\
     * \ar{r}{\xi}
       & K(Z^\bullet)\langle\xi\rangle.
   \end{tikzcd}
	\]
	Recall from \ssecref{ssec:sVect} that $\sVect(Z^\bullet)\langle\xi\rangle$ is equivalent to $\BGL(Z^\bullet)$.
	 The map $\sVect(Z^\bullet)\langle\xi\rangle\to K(Z^\bullet)\langle\xi\rangle$ between the lower right corners is acyclic in each degree by Lemma~\ref{lem:quillen}. Its geometric realization is an acyclic map whose domain has abelian fundamental groups (since the commutator subgroup of $\GL(Z)$ is generated by elementary matrices, which are $\A^1$-homotopic to the identity), hence it is an equivalence. The map between the upper right corners also induces an equivalence on geometric realizations by Corollary~\ref{cor:Emb-contractible}. Since the lower right corners are degreewise connected, it follows from \cite[Lemma 5.5.6.17]{HA} that geometric realization preserves these Cartesian squares, and we obtain the desired equivalence on the upper left corners.
\end{proof}

\begin{cor}\label{cor:nfr-vs-dfr}
	Let $Y\in \dSch_S$ and let $\xi\in\sVect_{\geq 0}(Y)$. Then the map
	\[
	\Lhtp\h^\nfr_S(Y,\xi) \to \Lhtp\h^\fr_S(Y,\xi)
	\]
	is an equivalence on derived affine schemes. In particular, it induces an equivalence
	\[
	L_\zar\Lhtp \h^\nfr_S(Y,\xi) \simeq L_\zar\Lhtp \h^\fr_S(Y,\xi).
	\]
\end{cor}

\begin{proof}
	This follows from Proposition~\ref{prop:nfr-to-dfr} using \cite[Lemma 2.3.12]{EHKSY1} (where one can harmlessly replace finite syntomic morphisms by finite quasi-smooth morphisms).
\end{proof}

Combining Corollaries~\ref{cor:efr-vs-nfr2} and \ref{cor:nfr-vs-dfr}, we obtain:

\begin{thm}\label{thm:fr-comparison}
	Let $Y$ be a smooth $S$-scheme and $\xi\in\sVect(Y)$ of rank $\geq 0$. Then the maps
	\[
	\h_S^\efr(Y,\xi) \to \h_S^\nfr(Y,\xi) \to \h_S^\fr(Y,\xi)
	\]
	in $\Pre_\Sigma(\dSch_S)$ are motivic equivalences.
\end{thm}

\ssec{Base change}
\label{ssec:base-change}

Let $X$ be a derived affine scheme and $Z\subset X$ a closed subscheme. We say that the pair $(X,Z)$ is \emph{henselian} if the underlying classical pair $(X_\cl,Z_\cl)$ is henselian \cite[Tag 09XD]{stacks}. By the topological invariance of the étale site \cite[Theorem 7.5.0.6]{HA}, $(X,Z)$ is henselian if and only if, for every étale affine $X$-scheme $Y$, the restriction map
\[
\Maps_X(X,Y) \to \Maps_X(Z,Y)
\]
is an effective epimorphism.

\begin{lem}\label{lem:lifting-isomorphisms}
	Let $(X,Z)$ be a henselian pair of derived affine schemes. Then the induced morphism of $\infty$-groupoids $\Vect(X) \to \Vect(Z)$ is $1$-connective.
\end{lem}

\begin{proof}
	The morphism $\Vect(X_\cl)\to \Vect(Z_\cl)$ is $0$-connective by \cite[Corollaire I.7]{Gruson}. If $P,Q\in \Vect_n(X_\cl)$, isomorphisms $P\simeq Q$ are sections of a $\GL_n$-torsor over $X_\cl$, which is in particular a smooth affine $X_\cl$-scheme. Using \cite[Théorème I.8]{Gruson}, this implies that $\Vect(X_\cl)\to \Vect(Z_\cl)$ is $1$-connective. It remains to observe that $\Vect(X) \to \Vect(X_\cl)$ is $2$-connective, because $\pi_0\Hom_R(P,Q) \simeq \Hom_{\pi_0(R)}(\pi_0(P),\pi_0(Q))$ when $P$ is a finite locally free $R$-module.
\end{proof}

\begin{lem}\label{lem:quasi-smooth-lift}
	Let $X$ be a derived affine scheme, $X_0\to X$ a closed immersion, and $Z_0$ an affine quasi-smooth $X_0$-scheme. Then there exists an affine quasi-smooth $X$-scheme $Z$ such that $Z\times_XX_0\simeq Z_0$.
\end{lem}

\begin{proof}
	Choose a smooth affine $X$-scheme $V$ and a closed immersion $Z_0\to V_0=V\times_XX_0$ over $X_0$ (for example, $V=\A^n_X$).
	The conormal sheaf $\sN$ of the immersion $Z_0\to V_0$ is finite locally free. By \cite[Corollaire I.7]{Gruson} and the topological invariance of the étale site, replacing $V$ by an étale neighborhood of $Z_0$ if necessary, there exists a finite locally free sheaf on $V_\cl$ lifting $\sN|(Z_0)_\cl$. Hence, by Lemma~\ref{lem:lifting-isomorphisms} applied to the pairs $(Z_0,(Z_0)_\cl)$ and $(V,V_\cl)$, there exists a finite locally free sheaf $\sE$ on $V$ lifting $\sN$. Let $\sE_0$ be the pullback of $\sE$ to $V_0$ and let $\sI$ be the fiber of $\sO_{V_0} \to \sO_{Z_0}$. Recall that there is a canonical surjection $\epsilon\colon \sI\to \sN$ in $\QCoh^\cn(V_0)$ (see \ssecref{ssec:fr-comparison}). Since $\sE$ and $\sE_0$ are projective in their respective $\infty$-categories of connective quasi-coherent sheaves \cite[Proposition 7.2.2.7]{HA}, we can find successive lifts
	\[
	\begin{tikzcd}
		\sE \ar{d} \ar[dashed]{rr}{\phi} & & \sO_V \ar[two heads]{dd} \\
		\sE_0 \ar[dashed]{dr}{\phi_0} \ar{d} & & \\
		\sN & \sI \ar[two heads]{l}{\epsilon} \ar{r}[swap]{\iota} & \sO_{V_0}\rlap.
	\end{tikzcd}
	\]
	 By Nakayama's lemma, the morphism $\phi_0\colon \sE_0\to \sI$ is surjective in a neighborhood of $Z_0$ in $V_0$; hence, the quasi-smooth closed subscheme of $V_0$ defined by $\iota\circ\phi_0$ (i.e., the zero locus of the corresponding section of the vector bundle $\bV(\sE_0) \to V_0$) has the form $Z_0\sqcup K$. Replacing $V$ by an affine open neighborhood of $Z_0$ if necessary, we can assume $K=\emptyset$. Let $Z\subset V$ be the quasi-smooth closed subscheme defined by $\phi$. By construction, $Z\times_XX_0$ is the quasi-smooth closed subscheme of $V_0$ defined by $\iota\circ\phi_0$, which is $Z_0$.
\end{proof}

\begin{lem}\label{lem:nfr-LKE}
	Let $S=\Spec R$ be an affine scheme, $Y$ a smooth $S$-scheme with an étale map to a vector bundle over $S$, and $\xi\in \sVect_{\geq 0}(Y)$. Then the functor $\h_S^{\nfr}(Y,\xi)\colon \CAlg_R^\Delta\to \Spc$ is left Kan extended from $\CAlg_R^\mathrm{sm}$.
\end{lem}

\begin{proof}
	We check conditions (1)–(3) of Proposition~\ref{prop:akhil}. Condition (1) holds by Proposition~\ref{prop:nfr-descent}(iii), and condition (3) is a special case of closed gluing (Proposition~\ref{prop:nfr-descent}(ii)).
	Let $(X,X_0)$ be a henselian pair of derived affine $R$-schemes, $f_0\colon Z_0\to X_0$ a finite quasi-smooth morphism, $i_0\colon Z_0\to \A^n_{X_0}$ a closed immersion over $X_0$, $g_0\colon Z_0\to Y$ an $S$-morphism, and $\tau_0$ an equivalence $\sN_{i_0}\simeq \sO_{Z_0}^{n-m}\oplus g_0^*(\sE)$, where $\xi=(\sE,m)$. We have to construct a lift of this data from $X_0$ to $X$. 
	Since $\h_S^\nfr(Y,\xi)$ is finitary (Proposition~\ref{prop:nfr-descent}(iii)) and $(X,X_1)$ is henselian for any $X_1$ containing $X_0$ \cite[Tag 0DYD]{stacks}, we can assume that $X_0\to X$ is finitely presented.
	By Lemma~\ref{lem:quasi-smooth-lift}, there exists an affine quasi-smooth lift $f\colon Z\to X$ of $f_0$. Replacing $Z$ by an open neighborhood of $Z_0$, we can assume $f$ quasi-finite. Moreover, by \cite[Corollary B.3.3.6]{SAG}, we have $Z=Z'\sqcup Z''$ where $Z'\to X$ is finite and $Z''\to X$ does not hit $X_0$; thus we can assume $f$ finite. By Proposition~\ref{prop:Emb-lifting}, increasing $n$ if necessary, we can also lift $i_0$ to a closed immersion $i\colon Z\to \A^n_X$ over $X$. 
	By assumption, there exists an étale map $h\colon Y \to V$ where $V$ is a vector bundle over $S$. Then the restriction map
	\[
	\Maps_S(Z, V) \to \Maps_S(Z_0,V)
	\]
	is an effective epimorphism (since $Z$ is affine), so the composite $h\circ g_0$ lifts to a map $Z\to V$. The projection $Y\times_VZ\to Z$ has a section over $Z_0$, hence over $Z$ since $h$ is affine étale and $(Z,Z_0)$ is henselian \cite[Proposition 18.5.6(i)]{EGA4-4}. If $g$ is the composite $Z\to Y\times_VZ\to Y$, then $g$ extends $g_0$. Finally, since $(Z,Z_0)$ is henselian, we can lift $\tau_0$ to an equivalence $\sN_{i}\simeq \sO_{Z}^{n-m}\oplus {g}^*(\sE)$ by Lemma~\ref{lem:lifting-isomorphisms}.
\end{proof}

\begin{thm}\label{thm:base-change}
	Let $f\colon S'\to S$ be a morphism of schemes, $Y$ a smooth $S$-scheme, and $\xi\in K(Y)$ of rank $\geq 0$. Then the canonical map
	\[
	f^*(\h_S^\fr(Y,\xi)|\Sm_S) \to \h_{S'}^\fr(Y_{S'},\xi_{S'})|\Sm_{S'}
	\]
	is a motivic equivalence.
\end{thm}

\begin{proof}
This is obvious if $f$ is smooth, so we may assume $S$ and $S'$ affine. 
Note that if we do not restrict these presheaves to smooth schemes, this map is obviously an equivalence. It therefore suffices to show that $L_\mot \h_S^\fr(Y,\xi)$, viewed as a presheaf on affine $S$-schemes, is the motivic localization of a colimit of presheaves represented by smooth affine $S$-schemes.
If $Y$ is the filtered colimit of quasi-compact open subschemes, then $\h_S^\fr(Y,\xi)$ is the filtered colimit of the corresponding subpresheaves, so we can assume $Y$ quasi-compact. Then there exists a finite open cover $\{U_1,\dotsc,U_k\}$ of $Y$ by $S$-schemes admitting étale maps to vector bundles over $S$. The map $U_1\sqcup\dotsb\sqcup U_k\to Y$ is a Zariski covering map; by Proposition~\ref{prop:fr-descent}(iii), $L_\nis\h_S^\fr(-,\xi)$ preserves the colimit of its Čech nerve. Together with Proposition~\ref{prop:fr-additivity}, we can assume that $Y$ admits an étale map to a vector bundle over $S$.
 In this case, we know from Lemma~\ref{lem:nfr-LKE} that $\h_S^\nfr(Y,\xi)$ is left Kan extended from smooth affine $S$-schemes. Since we have a motivic equivalence $\h_S^\nfr(Y,\xi)\to \h_S^\fr(Y,\xi)$ by Theorem~\ref{thm:fr-comparison}, we are done.
\end{proof}

\section{Geometric models of motivic Thom spectra}
\label{sec:thom}

The main result of this section, Theorem~\ref{thm:thom-general}, identifies the motivic Thom spectrum $M\beta$ of any $\beta\colon B\to K_{\geq 0}$ with the framed suspension spectrum of a concrete framed motivic space, namely the moduli stack of finite quasi-smooth schemes with $\beta$-structure. We obtain this result in several steps:
\begin{itemize}
	\item In \ssecref{ssec:thom-vector-bundles}, we prove the theorem for Thom spectra of vector bundles over smooth $S$-schemes. This is essentially a generalization to arbitrary base schemes of a theorem of Garkusha, Neshitov, and Panin over infinite fields (which is used also in the proof of the motivic recognition principle). However, it is necessary to reformulate their result using tangentially framed correspondences to obtain an identification that is both natural and multiplicative in the vector bundle.
	\item In \ssecref{ssec:thom-virtual-bundles}, we extend the theorem to Thom spectra of the form $\Th_{Y/S}(\xi)$ where $Y$ is a smooth $S$-scheme and $\xi\in K(Y)$ has rank $\geq 0$.
	\item Finally, in \ssecref{ssec:thom-general}, we introduce the notion of $\beta$-structure, we recall the formalism of motivic Thom spectra, and we deduce the general theorem.
\end{itemize}
At each step we also obtain a computation of the infinite $\P^1$-loop spaces of these Thom spectra over perfect fields, using the motivic recognition principle. In \ssecref{ssec:MGL}, we specialize the main theorem to the motivic Thom spectra $\MGL$ and $\MSL$. Finally, in \ssecref{ssec:hilbert}, we rephrase our computations in terms of Hilbert schemes.

\ssec{Thom spectra of vector bundles}
\label{ssec:thom-vector-bundles}

Let $(\Sm_S)_{/\Vect}\to \Sm_S$ denote the Cartesian fibration classified by $\Vect\colon \Sm_S^\op \to \Spc$ (as the notation suggests, $(\Sm_S)_{/\Vect}$ is also a full subcategory of the overcategory $\Pre(\Sm_S)_{/\Vect}$). An object of $(\Sm_S)_{/\Vect}$ is thus a pair $(Y,\sE)$ where $Y$ is a smooth $S$-scheme and $\sE$ is a finite locally free sheaf on $Y$, and a morphism $(Y,\sE) \to (Y',\sE)$ is a pair $(f,\phi)$ where $f\colon Y\to Y'$ is an $S$-morphism and $\phi\colon \sE\simeq f^*(\sE')$ is an isomorphism in $\Vect(Y)$. Similarly, we denote by $(\Sm_S)_{/K_{\geq 0}}$ the $\infty$-category of pairs $(Y,\xi)$ where $Y$ is a smooth $S$-scheme and $\xi\in K(Y)$ is of rank $\geq 0$.

Since $\Vect$ and $K_{\geq 0}$ are presheaves of $\Einfty$-spaces (under direct sum), the $\infty$-categories $(\Sm_S)_{/\Vect}$ and $(\Sm_S)_{/K_{\geq 0}}$ acquire symmetric monoidal structures with
\[
(Y_1,\xi_1) \otimes (Y_2,\xi_2) = (Y_1\times_SY_2, \pi_1^*(\xi_1) \oplus \pi_2^*(\xi_2))
\]
(see \cite[\sectsign 2.2.2]{HA}).
The assignment $(Y,\xi)\mapsto \h_S^\fr(Y,\xi)$ is a right-lax symmetric monoidal functor $(\Sm_S)_{/K_{\geq 0}}\to \Pre_\Sigma(\Span^\fr(\Sm_S))$ (see Appendix~\ref{app:category}).

\begin{constr}\label{constr:Theta}
	Let $(Y,\sE)\in (\Sm_S)_{/\Vect}$ and let $\bV^\times(\sE)\subset\bV(\sE)$ denote the complement of the zero section of the vector bundle $\bV(\sE)$. We define a morphism
\begin{equation*}\label{eqn:thom-comparsion}
\Theta_{Y/S,\sE}\colon \h_S^\fr(\bV(\sE)/\bV^\times(\sE)) \to \h_S^\fr(Y,\sE)
\end{equation*}
in $\Pre_\Sigma(\Span^\fr(\Sm_S))$ as follows. Let $z\colon Y\hook \bV(\sE)$ be the zero section. Then $z$ is a regular closed immersion with a canonical equivalence $\sL_z \simeq \sE[1]$, whence an equivalence $\tau\colon \sL_z\simeq -\sE$ in $K(Y)$. The span
\[
\begin{tikzcd}
   & Y \ar[swap]{ld}{z}\ar{rd}{\id} & \\
 \bV(\sE) &   & Y
\end{tikzcd}
\]
together with the equivalence $\tau$ defines a canonical element of $\h_S^\fr(Y,\sE)(\bV(\sE))$. Moreover, its image in $\h_S^\fr(Y,\sE)(\bV^\times(\sE))$ is the empty correspondence, which is the zero element. This defines the desired map $\Theta_{Y/S,\sE}$. 
The morphism $\Theta_{Y/S,\sE}$ is clearly natural and symmetric monoidal in the pair $(Y,\sE)\in (\Sm_S)_{/\Vect}$.
\end{constr}

We now consider the diagram of symmetric monoidal $\infty$-categories
\begin{equation}\label{eqn:main-diagram}
\begin{tikzcd}
	(\Sm_S)_{/\Vect} \ar{r}{\Th} \ar{d} & \Pre_\Sigma(\Sm_S)_* \ar{d}{\gamma^*} \ar{r}{L_\mot} \ar[dl,shorten <>=10pt,Rightarrow,"\Theta"'] & \H(S)_* \ar{d}[swap]{\gamma^*} \ar{r}{\Sigma^\infty_\T} & \SH(S) \ar{d}{\simeq}[swap]{\gamma^*} \\
	(\Sm_S)_{/K_{\geq 0}} \ar{r}[swap]{\h^\fr} & \Pre_\Sigma(\Span^\fr(\Sm_S)) \ar{r}[swap]{L_\mot} & \H^\fr(S) \ar{r}[swap]{\Sigma^\infty_{\T,\fr}} & \SH^\fr(S)\rlap,
\end{tikzcd}
\end{equation}
where:
\begin{itemize}
	\item $\Th$ sends $(Y,\sE)$ to the quotient $\bV(\sE)/\bV^\times(\sE)$;
	\item $\h^\fr$ sends $(Y,\xi)$ to the the presheaf $\h^\fr_S(Y,\xi)$;
	\item $\Theta$ is the natural transformation with components $\Theta_{Y/S,\sE}$.
\end{itemize}
Note that $\Th$ and $\h^\fr$ are only right-lax symmetric monoidal, but all the other functors in this diagram are strictly symmetric monoidal (and $\Th$ becomes strictly monoidal after Zariski sheafification). The fact that $\gamma^*\colon \SH(S)\to \SH^\fr(S)$ is an equivalence was proved in \cite[Theorem 18]{framed-loc}.

Our goal in this subsection is to prove the following theorem:

\begin{thm}\label{thm:thom}
	Let $S$ be a scheme, $Y$ a smooth $S$-scheme, and $\sE$ a finite locally free sheaf on $Y$. Then $\Sigma^\infty_{\T,\fr}\Theta_{Y/S,\sE}$ is an equivalence. In other words, the boundary of~\eqref{eqn:main-diagram} is strictly commutative.
\end{thm}

One of the main inputs is the following theorem of Garkusha–Neshitov–Panin:

\begin{thm}[Garkusha–Neshitov–Panin]
	\label{thm:GNP}
	Let $k$ be an infinite field, $Y$ a smooth separated $k$-scheme of finite type, and $n\geq 0$. Then the canonical map
	\[
	\h^{\efr}_k(\A^n_Y /(\A^n_Y-0)) \to \h^{\efr}_k(Y,\sO^n)
	\]
	of presheaves on $\Sm_k$ is a motivic equivalence.
\end{thm}

Here, the left-hand side uses the formal extension of $\h^\efr_k$ to $\Pre_\Sigma(\Sm_k)_*$ (see \cite[2.1.10]{EHKSY1}). The ``canonical map'' sends an equationally framed correspondence $(Z,U,\phi,g)$ from $X$ to $\A^n_Y$, where $g=(g_0,g_1)\colon U\to \A^n\times Y$, to the correspondence $(Z\cap g_0^{-1}(0),U,(\phi,g_0),g_1)$.

\begin{proof}[Proof of Theorem~\ref{thm:GNP}]
	Modulo the notation, this follows from the level $0$ part of \cite[Theorem 1.1]{gnp}, which assumes that $k$ is an infinite perfect field. The result was generalized by Druzhinin in \cite{Druzhinin}, where it is made clear that it holds as stated here over any infinite field (the perfectness assumption only being needed to ensure that $L_\mot=L_\nis \Lhtp$ when $n>0$).
\end{proof}

\begin{lem}\label{lem:finite-fields}
	Let $k$ be a field and let $\sF\in\Pre_{\Sigma,\A^1}(\Span^\fr(\Sm_k))$. Suppose that $L_\nis(\sF_{K})$ is a grouplike presheaf of $\Einfty$-spaces on $\Sm_K$ for some separable algebraic field extension $K/k$. Then $L_\nis\sF$ is grouplike.
\end{lem}

\begin{proof}
	Let $X$ be the henselization of a point in a smooth $k$-scheme and let $\alpha\in \pi_0(\sF(X))$.
	Note that $X_K$ is a finite sum of henselian local schemes. By the assumption and a continuity argument, there exists a finite separable extension $k'/k$ such that the image of $\alpha$ in $\pi_0(\sF(X_{k'}))$ has an additive inverse $\beta$. By \cite[Proposition B.1.4]{EHKSY1}, there exists a morphism $\phi$ from $\Spec k$ to $\Spec k'$ in $\Span^\fr(\Sm_k)$ such that $\phi^*(\alpha_{k'})=d_\epsilon\alpha$, where $d=[k':k]$. Hence, $d_\epsilon \alpha+\phi^*(\beta)=0$.
	Since $1$ is a summand of $d_\epsilon$, this implies that $\alpha$ has an additive inverse.
\end{proof}

\begin{lem}\label{lem:Nis-connected}
	Let $k$ be a field, $Y$ a smooth $k$-scheme, $\sE$ a finite locally free sheaf on $Y$ of rank $\geq 1$, and $\xi\in K(Y)$ of rank $\geq 1$. Then the Nisnevich sheaves
	\[
	L_\nis\Lhtp\h_k^\fr(\bV(\sE)/\bV^\times(\sE)) \quad\text{and}\quad L_\nis\Lhtp \h_k^\fr(Y,\xi)
	\]
	on $\Sm_k$ are grouplike. If $k$ is infinite, they are connected.
\end{lem}

In fact, these sheaves are connected even if $k$ is finite, see Remark~\ref{rem:connectivity} below.

\begin{proof}
	By Lemma~\ref{lem:finite-fields}, it suffices to prove the last statement.
	We must show that any section of $\h_k^\fr(\bV(\sE)/\bV^\times(\sE))$ or $\h_k^\fr(Y,\xi)$ over a henselian local scheme $X$ is $\A^1$-homotopic to $0$. We can assume that the finite $X$-scheme in such a section is connected and hence has a unique closed point. Thus, we can shrink $Y$ so that it admits an étale map to an affine space and so that $\sE$ and $\xi$ are trivial.
	 By Corollaries~\ref{cor:efr-vs-nfr} and \ref{cor:nfr-vs-dfr}, it then suffices to show that the sheaves
	\[
	L_\nis\Lhtp\h_k^\efr(\A^n_Y/(\A^n_Y-0)) \quad\text{and}\quad L_\nis\Lhtp \h_k^\efr(Y,\sO^n)
	\]
	are connected when $n\geq 1$. This is precisely \cite[Lemma A.1]{gnp}.
\end{proof}

\begin{prop}\label{prop:GNP}
	Let $k$ be a field, $Y$ a smooth $k$-scheme, and $\sE$ a finite locally free sheaf on $Y$. 
	Then the map $L_\mot\Theta_{Y/k,\sE}$ is an equivalence in $\H^{\fr}(k)$.
\end{prop}

\begin{proof}
	As in the proof of Theorem~\ref{thm:base-change}, we can assume $Y$ separated of finite type and $\sE$ trivial. We can also assume $\sE$ of rank $\geq 1$, since the statement is tautological when $\sE=0$.
	If $k$ is infinite, the result follows by combining Theorems~\ref{thm:GNP} and \ref{thm:fr-comparison}, noting that the square
\[
\begin{tikzcd}
	\h^{\efr}_k(\A^n_Y /(\A^n_Y-0)) \ar{r} \ar{d} & \h^{\efr}_k(Y,\sO^n) \ar{d} \\
	\h^{\fr}_k(\A^n_Y /(\A^n_Y-0)) \ar{r}{\Theta} & \h^{\fr}_k(Y,\sO^n)
\end{tikzcd}
\]
is commutative.
	In light of Lemma~\ref{lem:Nis-connected}, the result for $k$ finite follows immediately from \cite[Corollary B.2.5(2)]{EHKSY1}.
\end{proof}

\begin{proof}[Proof of Theorem~\ref{thm:thom}]
	The source and target of $\Sigma^\infty_{\T,\fr}L_\mot\Theta_{Y/S,\sE}$ are both stable under base change: this is obvious for the source, and for the target it follows from Theorem~\ref{thm:base-change}.
	The question is in particular local on $S$, so we can assume $S$ qcqs. We can also assume $Y$ qcqs as in the proof of Theorem~\ref{thm:base-change}. By Noetherian approximation, we can then assume $S$ of finite type over $\Spec \Z$. In this case, equivalences in $\SH(S)$ are detected pointwise on $S$ \cite[Proposition B.3]{norms}, so we can assume that $S$ is the spectrum of a field. Now the claim follows from Proposition~\ref{prop:GNP}.
\end{proof}

\ssec{Thom spectra of nonnegative virtual vector bundles}
\label{ssec:thom-virtual-bundles}

\begin{thm}\label{thm:thom-bundles}
	Let $Y$ be a smooth $S$-scheme and $\xi\in K(Y)$ of rank $\geq 0$. Then there is an equivalence
	\[
	\Th_{Y/S}(\xi) \simeq \Sigma^\infty_{\T,\fr} \h_S^\fr(Y,\xi)
	\]
	in $\SH(S)\simeq \SH^\fr(S)$, natural and symmetric monoidal in $(Y,\xi)$.
\end{thm}

\begin{proof}
	We consider the following restriction of the diagram~\eqref{eqn:main-diagram}:
	\[
	\begin{tikzcd}
		\Vect(S) \ar{r}{\Th} \ar{d} & \Pre_\Sigma(\Sm_S)_* \ar{d}{\gamma^*} \ar{r}{L_\mot} \ar[dl,shorten <>=10pt,Rightarrow,"\Theta"'] & \H(S)_* \ar{d}[swap]{\gamma^*} \ar{r}{\Sigma^\infty_\T} & \SH(S) \ar{d}{\simeq}[swap]{\gamma^*} \\
		K_{\geq 0}(S) \ar{r}[swap]{\h^\fr} & \Pre_\Sigma(\Span^\fr(\Sm_S)) \ar{r}[swap]{L_\mot} & \H^\fr(S) \ar{r}[swap]{\Sigma^\infty_{\T,\fr}} & \SH^\fr(S)\rlap.
	\end{tikzcd}
	\]
	This is a diagram of symmetric monoidal $\infty$-categories and right-lax symmetric monoidal functors, which is natural in $S$.
	By Theorem~\ref{thm:thom}, the boundary of this diagram is strictly commutative. The composite of the top row is strictly symmetric monoidal and lands in $\Pic(\SH(S))$, hence it extends uniquely to a symmetric monoidal functor $\Vect(S)^\gp\to \SH(S)$.
	We claim that the composite of the bottom row is also strictly symmetric monoidal, i.e., that for $\xi,\eta\in K_{\geq 0}(S)$ the structural map 
	\[
	\Sigma^\infty_{\T,\fr} \h_S^\fr(S,\xi) \otimes \Sigma^\infty_{\T,\fr} \h_S^\fr(S,\eta) \to \Sigma^\infty_{\T,\fr} \h_S^\fr(S,\xi+\eta)
	\]
	is an equivalence. Indeed, this assertion is local on $S$, so we can assume that $\xi$ and $\eta$ are finite locally free sheaves on $S$, in which case the claim follows from the commutativity of the diagram. Similarly, the composite of the bottom row lands in $\Pic(\SH^\fr(S))$, as can be checked locally on $S$. Thus, the bottom row extends uniquely to a symmetric monoidal functor $K_{\geq 0}(S)^\gp \simeq K(S) \to \SH^\fr(S)$, and we have an induced commutative diagram
	\[
	\begin{tikzcd}
		\Vect(S)^\gp \ar{r}{\Th} \ar{d} & \SH(S) \ar{d}{\simeq}[swap]{\gamma^*} \\
		K(S) \ar{r}[swap]{\h^\fr} & \SH^\fr(S)\rlap,
	\end{tikzcd}
	\]
	still natural in $S$.
	Since the canonical map $\Vect(-)^\gp\to K(-)$ is a Zariski equivalence and $\SH(-)$ is a Zariski sheaf,
	the top horizontal map factors uniquely through $K(-)$, giving rise to the motivic J-homomorphism $\Th_{S/S}(-)\colon K(S)\to \SH(S)$ \cite[\sectsign 16.2]{norms}. Hence, we obtain a symmetric monoidal equivalence
	\[
	\Th_{S/S}(-) \simeq \Sigma^\infty_{\T,\fr} \h_S^\fr(S,-) \colon K_{\geq 0}(S) \to \SH(S),
	\]
	natural in $S$. Unstraightening over $\Sm_S$ and composing with the symmetric monoidal functor
	\[
	(\Sm_S)_{/\SH^\simeq} \to \SH(S), \quad  (f\colon Y\to S, E\in\SH(Y)) \mapsto f_\sharp E,
	\]
	constructed in \cite[\sectsign 16.3]{norms}, we obtain an equivalence
	\[
	f_\sharp\Th_{Y/Y}(\xi) \simeq f_\sharp \Sigma^\infty_{\T,\fr} \h_Y^\fr(Y,\xi)
	\]
	which is natural and symmetric monoidal in $(f\colon Y\to S,\xi)\in(\Sm_S)_{/K_{\geq 0}}$.
	We now define the desired symmetric monoidal natural transformation $\Th_{Y/S}(\xi)\to \Sigma^\infty_{\T,\fr}\h_S^\fr(Y,\xi)$ by the commutative square
	\[
	\begin{tikzcd}
	f_\sharp\Th_{Y/Y}(\xi) \ar{r}{\simeq} \ar{d}[swap]{\simeq} & \Th_{Y/S}(\xi) \ar[dashed]{d} \\
	f_\sharp \Sigma^\infty_{\T,\fr} \h_Y^\fr(Y,\xi) \ar{r} & \Sigma^\infty_{\T,\fr}\h_S^\fr(Y,\xi)\rlap.
	\end{tikzcd}
	\]
It remains to show that the lower horizontal map is an equivalence.
	The assertion is local on $Y$ (by Propositions~\ref{prop:fr-descent}(iii) and \ref{prop:fr-additivity}), so we can assume that $\xi$ is a finite locally free sheaf. In this case, we know that the right vertical map is an equivalence by Theorem~\ref{thm:thom}, which concludes the proof.
\end{proof}

\begin{cor}\label{cor:thom-bundles}
	Let $k$ be a perfect field, $Y$ a smooth $k$-scheme, and $\xi\in K(Y)$ of rank $\geq 0$. Then there is an equivalence
	\[
	\Omega^\infty_{\T,\fr}\Th_{Y/k}(\xi) \simeq L_\zar \Lhtp\h_k^\fr(Y,\xi)^\gp
	\]
	in $\H^\fr(k)$, natural and symmetric monoidal in $(Y,\xi)$. Moreover, if the rank of $\xi$ is $\geq 1$, then $L_\nis \Lhtp\h_k^\fr(Y,\xi)$ is already grouplike.
\end{cor}

\begin{proof}
	The first statement follows from Theorem~\ref{thm:thom-bundles}, the fact that the functor 
	\[
	\Sigma^\infty_{\T,\fr}\colon \H^\fr(k)^\gp \to \SH^\fr(k)
	\]
	 is fully faithful \cite[Theorem 3.5.13(i)]{EHKSY1}, and the fact that the motivic localization functor $L_\mot$ can be computed as $L_\zar\Lhtp$ on $\Pre_\Sigma(\Span^\fr(\Sm_k))^\gp$ \cite[Theorem 3.4.11]{EHKSY1}.
	If $\xi$ has rank $\geq 1$, then $L_\nis \Lhtp\h_k^\fr(Y,\xi)$ is grouplike by Lemma~\ref{lem:Nis-connected}.
\end{proof}

\begin{rem}\label{rem:connectivity}
	In the setting of Corollary~\ref{cor:thom-bundles}, if $\xi$ has rank $\geq n$, then $L_\nis\Lhtp\h^\fr_k(Y,\xi)$ is $n$-connective (as a Nisnevich sheaf). This is obvious if $n=0$. If $n\geq 1$, then it is grouplike and hence equivalent to $\Omega^\infty_\T\Th_{Y/k}(\xi)$, which is $n$-connective by Morel's stable $\A^1$-connectivity theorem.
\end{rem}

\begin{cor}\label{cor:thom-bundles-efr}
	Let $k$ be a perfect field, $Y$ a smooth $k$-scheme, and $\xi \in \sVect(Y)$ of rank $\geq 0$. Then there are equivalences
	\[
	\Omega^\infty_{\T}\Th_{Y/k}(\xi) \simeq L_\zar (\Lhtp\h_k^\nfr(Y,\xi))^\gp \simeq L_\zar (\Lhtp\h_k^\efr(Y,\xi))^\gp,
	\]
 natural in $(Y,\xi)$.
\end{cor}

\begin{proof}
The first equivalence follows by combining Corollaries~\ref{cor:thom-bundles} and~\ref{cor:nfr-vs-dfr}. 
To deduce the second equivalence from Corollary~\ref{cor:efr-vs-nfr2}, it is enough to show that
\[
L_\zar (\Lhtp\h_k^\efr(Y,\xi))^\gp \simeq (L_\mot\h_k^\efr(Y,\xi))^\gp.
\]
This follows from \cite[Remark~3.4.12]{EHKSY1}, since $\Lhtp\h_k^\efr(Y,\xi)$ is a presheaf on $\Span^\efr_*(\Sm_k)$ and, for any $X\in\Sm_k$, the endomorphism $\sigma_X^*$ of $(\Lhtp\h_k^\efr(Y,\xi))(X)$ is homotopic to the identity \cite[Lemma~3.1.4]{MuraThesis}.
\end{proof}

\begin{rem}
	One can give a more direct proof of a less structured version of Corollary~\ref{cor:thom-bundles} if $\xi=(\sE,m)\in\sVect_{\geq 0}(Y)$.
	Voevodsky's lemma (Remark~\ref{rem:voev-lemma}) provides a map
	\[
	L_\mot \h^{\efr}_k(Y,\xi)^\gp \to \Omega^m_\T\Omega^\infty_\T\Sigma^\infty_\T(\bV(\sE)/ \bV^\times (\sE)).
	\]
	To show that it is an equivalence, we can assume that $\xi=\sO^n_Y$. In this case, it follows from \cite[Corollary 3.3.8]{EHKSY2} that this map is inverse to the equivalence of Theorem~\ref{thm:GNP} (which is tautological if $\rk\xi=0$). Note however that there is no hope of obtaining the monoidal equivalence of Corollary~\ref{cor:thom-bundles} in this way, because $\sVect(Y)$ has no monoidal structure.
\end{rem}

\begin{cor}\label{cor:sphere}
	Let $S$ be pro-smooth over a field. For every $\xi\in K(S)$ of rank $\geq 0$, there is an equivalence
	\[
	\Omega^\infty_{\T,\fr}\Sigma^\xi\1_S \simeq L_\zar\Lhtp\h^\fr_S(S,\xi)^\gp
	\]
	in $\H^\fr(S)$.
	Moreover, if $\xi$ has rank $\geq 1$, we can replace the group completion on the right-hand side by Nisnevich sheafification.
\end{cor}

\begin{proof}
	By Theorem~\ref{thm:thom-bundles}, we have an equivalence
	\[
	\Sigma^\xi\1_S \simeq \Sigma^\infty_{\T,\fr}\h^\fr_S(S,\xi)
	\]
	for any scheme $S$. By adjunction, we get a map
	\[
	L_\zar\Lhtp\h^\fr_S(S,\xi)^\gp \to \Omega^\infty_{\T,\fr}\Sigma^\xi_\T\1_S.
	\]
	To prove that it is an equivalence when $S$ is pro-smooth over a field, we can assume $\xi$ trivial since the question is local on $S$. We are then reduced to the case of a perfect field, which follows from Corollary~\ref{cor:thom-bundles}.
\end{proof}

\ssec{General nonnegative Thom spectra}
\label{ssec:thom-general}

\begin{defn}
	Let $S$ be a scheme. A \emph{stable tangential structure} over $S$ is a morphism $\beta\colon B\to K$ in $\Pre_\Sigma(\dSch_S)$. We say that $\beta$ has \emph{rank $n$} if $\beta$ lands in the rank $n$ summand of K-theory.
Given a quasi-smooth morphism $f\colon Z\to X$ with $Z\in \dSch_S$, a \emph{$\beta$-structure} on $f$ is a lift of $-\sL_f$ to $B(Z)$.
\end{defn}

Given a stable tangential structure $\beta$ over $S$, we denote by $\FQSM_S^\beta\colon \dSch_S^\op \to \Spc$ the moduli stack of $\beta$-structured finite quasi-smooth schemes over $S$:
\[
\FQSM_S^\beta(X) = \{\text{$\beta$-structured finite quasi-smooth derived $X$-schemes}\}.
\]

\begin{ex}\label{ex:iota_n}
	Let $\iota_n$ be the inclusion of the rank $n$ summand of $K$. A $\iota_n$-structure on a quasi-smooth morphism $f$ is simply the property that $f$ has relative virtual dimension $-n$.
	We will also write
		\[
		\FQSM_S^{n} = \FQSM_S^{\iota_n}
		\]
		for the moduli stack of finite quasi-smooth schemes of dimension $-n$. If $n=0$, this is the moduli stack $\FSYN_S$ of finite syntomic schemes (which is a smooth quasi-separated algebraic stack over $S$), by Lemma~\ref{lem:classical}.
\end{ex}

\begin{ex}\label{ex:framing}
	If $\beta\colon *\to K$ is the zero section, then a $\beta$-structure on a quasi-smooth morphism $f\colon Z\to X$ is an equivalence $\sL_f\simeq 0$ in $K(Z)$, i.e., a (stable) $0$-dimensional \emph{framing} of $f$. The presheaf $\FQSM_S^\beta$ coincides with the presheaf $\FSYN_S^\fr$ considered in \cite[3.5.17]{EHKSY1}.
\end{ex}

\begin{ex}\label{ex:SL-orientation}
	If $\beta$ is the fiber of $\det\colon K\to \Pic$, then a $\beta$-structure on a quasi-smooth morphism $f\colon Z\to X$ is an equivalence $\det(\sL_{f})\simeq \sO_Z$. We call this structure an \emph{orientation} of $f$, and we write
	\[
	\FQSM_S^{\ornt} = \FQSM_S^{\beta}
	\]
	for the moduli stack of oriented finite quasi-smooth schemes.
	We note that $\det^{-1}(\sO)$ coincides (on qcqs derived schemes) with the presheaf $K^\SL$ from Example~\ref{ex:LKE}, defined as the right Kan extension from derived affine schemes of the group completion of the monoidal groupoid $\Vect^\SL$ of locally free sheaves with trivialized determinant. Indeed, on derived affine schemes, $\sVect^\SL\to K^\SL$ is a plus construction (see Remark~\ref{rem:KSL-plus}), and in particular it is acyclic. The map $\sVect^\SL\to \det^{-1}(\sO)$ is also acyclic, being a pullback of $\sVect\to K$. It follows that $K^\SL \to \det^{-1}(\sO)$ is an acyclic map, whence an equivalence since the source has abelian fundamental groups.
\end{ex}


\begin{ex}
	If $\beta$ is the fiber of the motivic J-homomorphism $K\to \Pic(\SH)$, then a $\beta$-structure on a quasi-smooth morphism $f\colon Z\to X$ is an equivalence $\Th_{Z}(\sL_f)\simeq \1_Z$ in $\SH(Z)\simeq\SH(Z_\cl)$ (see \cite{KhanThesis} for the extension of $\SH(-)$ to derived schemes).
\end{ex}

\begin{ex}\label{ex:twisted-frames}
	Let $Y\in\dSch_S$ and $\xi\in K_{\geq 0}(Y)$. If $\beta\colon Y\to K$ classifies $\xi$, then $\FQSM_S^\beta=\h_S^\fr(Y,\xi)$ as defined in \ssecref{ssec:twisted-frames}.
\end{ex}

For $\beta$ a stable tangential structure over $S$, we formally have
\[
\FQSM_S^\beta \simeq \colim_{\substack{Y\in\dSch_S\\b\in B(Y)}} \h_S^\fr(Y,\beta(b))
\]
in $\Pre(\dSch_S)$, where the colimit is indexed by the source of the Cartesian fibration classified by $B\colon \dSch_S^\op \to \Spc$. Note that this $\infty$-category has finite sums (because $B$ is a $\Sigma$-presheaf) and hence is sifted, so that the above formula is also valid in $\Pre_\Sigma(\Span^\fr(\dSch_S))$.

\begin{prop}\label{prop:beta-descent}
	Let $S$ be a scheme.
	The functor 
	\[
	\Pre_\Sigma(\dSch_S)_{/K_{\geq 0}} \to \Pre_\Sigma(\Span^\fr(\dSch_S)), \quad \beta\mapsto \FQSM^\beta_S,
	\]
	preserves Nisnevich equivalences and étale equivalences.
\end{prop}

\begin{proof}
	Note that this functor preserves sifted colimits. By \cite[Corollary 5.1.6.12]{HTT}, there is a canonical equivalence $\Pre_\Sigma(\dSch_S)_{/K_{\geq 0}}\simeq \Pre_\Sigma((\dSch_S)_{/K_{\geq 0}})$. By \cite[Lemma 2.10]{norms}, it therefore suffices to prove the following: for any $Y\in\dSch_S$, any $\xi\in K_{\geq 0}(Y)$, and any Nisnevich (resp.\ étale) covering sieve $R\hook Y$, the induced map $\h_S^\fr(L_\Sigma R,\xi) \to \h_S^\fr(Y,\xi)$ is a Nisnevich (resp.\ étale) equivalence.
	If $R$ is a finitely generated sieve, this follows from Proposition~\ref{prop:fr-descent}(iii) since $L_\Sigma R$ is a sieve generated by a single map. If $Y$ is quasi-compact, then $R$ admits a finitely generated refinement, so we are done in this case. In general, write $Y$ as a filtered colimit of quasi-compact open subschemes $Y_\alpha\subset Y$ and let $R_\alpha=R\times_YY_\alpha$. Then the canonical map $\colim_\alpha \h^\fr_S(Y_\alpha,\xi) \to \h^\fr_S(Y,\xi)$ is an equivalence on quasi-compact derived schemes, and in particular it is a Nisnevich equivalence. Similarly, $\colim_\alpha \h^\fr_S(L_\Sigma R_\alpha,\xi) \to \h^\fr_S(L_\Sigma R,\xi)$ is a Nisnevich equivalence, and we conclude by 2-out-of-3.
\end{proof}

The following somewhat technical definition plays a crucial role in the sequel.

\begin{defn}
	A stable tangential structure $\beta\colon B\to K$ over $S$ is called \emph{smooth} if the counit map $\tilde B\to B$ is a Nisnevich equivalence, where $\tilde B$ is the left Kan extension of $B|\Sm_S$ to $\dSch_S$.
\end{defn}

\begin{lem}\label{lem:local-LKE}
	Let $S$ be a scheme and $\beta\colon B\to K$ a stable tangential structure over $S$.
	Suppose that $B$ is left Kan extended along $\SmAff_T \subset \dAff_T$ for every $T$ in some affine Nisnevich cover of $S$. Then $\beta$ is smooth.
\end{lem}

\begin{proof}
	Consider the square of adjunctions
	\[
	\begin{tikzcd}
		\Pre_{\nis}(\Sm_S) \ar[shift left=1]{r}{\mathrm{LKE}} \ar[shift right=1]{d}[swap]{\mathrm{res}} & \Pre_{\nis}(\dSch_S) \ar[shift right=1]{d}[swap]{\mathrm{res}} \ar[shift left=1]{l}{\mathrm{res}} \\
		\Pre_\nis(\SmAff_T) \ar[shift left=1]{r}{\mathrm{LKE}} \ar[shift right=1]{u}[swap]{\mathrm{RKE}} & \Pre_\nis(\dAff_T)\rlap. \ar[shift left=1]{l}{\mathrm{res}} \ar[shift right=1]{u}[swap]{\mathrm{RKE}}
	\end{tikzcd}
	\]
	It is easy to show that the square of right adjoints commutes, because the inclusions $\SmAff_T\subset \Sm_T$ and $\dAff_T\subset \dSch_T$ induce equivalences of Nisnevich $\infty$-topoi. Hence, the square of left adjoints commutes as well. This shows that the counit map $\tilde B\to B$ is a Nisnevich equivalence when restricted to $\dAff_T$ for all $T$ in the cover, hence it is a Nisnevich equivalence.
\end{proof}

In Appendix~\ref{app:LKE}, we provide many examples of stable tangential structures satisfying the assumption of Lemma~\ref{lem:local-LKE}, which are therefore smooth. In particular, K-theory itself has this property. More generally, if $X$ is a smooth algebraic stack over $S$ with quasi-affine diagonal and with a structure of $\sE_1$-monoid over $\Vect$, and if $I\subset \Z$ is any subset, then the stable tangential structure $X^\gp\times_{L_\Sigma\Z}L_\Sigma I\to K$ is smooth (by Proposition~\ref{prop:akhil2} and Lemmas \ref{lem:gp} and~\ref{lem:PB-LKE}). For example, the stable tangential structures in Examples \ref{ex:iota_n}, \ref{ex:framing}, and \ref{ex:SL-orientation} are smooth, while the one in Example~\ref{ex:twisted-frames} is smooth if and only if $Y$ is smooth.

We briefly recall the formalism of motivic Thom spectra from \cite[Section 16]{norms}. To a morphism $\beta\colon B\to \Pic(\SH)$ in $\Pre(\Sm_S)$ one can associate a Thom spectrum $M\beta\in \SH(S)$. As in topology, it is given by a formal colimit construction:
\[
M\beta = \colim_{\substack{f\colon Y \to S\text{ smooth}\\ b\in B(Y)}} f_\sharp\beta(b).
\]
Moreover, it has good multiplicative properties: one has a symmetric monoidal functor
\[
M\colon \Pre(\Sm_S)_{/\Pic(\SH)} \to \SH(S).
\]
In particular, if $\beta$ is an $\sE_n$-morphism, then $M\beta$ is an $\sE_n$-ring spectrum.

Recall also that the motivic J-homomorphism is a morphism of $\Einfty$-spaces
\[
K(S) \to \Pic(\SH(S)), \quad \xi\mapsto \Th_{S/S}(\xi),
\]
natural in $S$.
Restricting $M$ along the J-homomorphism, we obtain a symmetric monoidal functor
\[
M\colon \Pre(\Sm_S)_{/K} \to \SH(S).
\]
For $n\in\Z$, the shifted algebraic cobordism spectrum $\Sigma^n_{\T}\MGL_S\in\SH(S)$ is the Thom spectrum of the restriction of the J-homomorphism to the rank $n$ summand of K-theory \cite[Theorem 16.13]{norms}.

\begin{thm}\label{thm:thom-general}
	Let $S$ be a scheme and $\beta\colon B\to K_{\geq 0}$ a smooth stable tangential structure over $S$. Then there is an equivalence
	\[
	M\beta \simeq \Sigma^\infty_{\T,\fr}\FQSM_S^{\beta}
	\]
	in $\SH(S)\simeq\SH^\fr(S)$, natural and symmetric monoidal in $\beta$. 
	In particular, if $\beta$ is $\sE_n$ for some $0\leq n\leq \infty$, then this is an equivalence of $\sE_n$-ring spectra.
\end{thm}

\begin{rem}
	In the statement of Theorem~\ref{thm:thom-general}, $M\beta$ depends only on the restriction of $\beta$ to $\Sm_S$. If we start with a morphism $\beta_0\colon B_0\to K_{\geq 0}$ in $\Pre_\Sigma(\Sm_S)$, we can always apply the theorem with $\beta\colon B\to K_{\geq 0}$ the left Kan extension of $\beta_0$ to obtain an equivalence $M\beta_0 \simeq \Sigma^\infty_{\T,\fr}\FQSM_S^{\beta}$. (Note that $B$ is a $\Sigma$-presheaf on $\dSch_S$, because for $X_1,\dotsc,X_k\in\dSch_S$ the sum functor $\prod_{i}(\Sm_S)_{X_i/} \to (\Sm_S)_{\bigcoprod_i X_i/}$ is left adjoint, hence coinitial.)
\end{rem}

\begin{proof}
	This is a formal consequence of Theorem~\ref{thm:thom-bundles}. We have
	\[
	M\beta = \colim_{\substack{Y\in\Sm_S\\b\in B(Y)}} \Th_{Y/S}(\beta(b)) \simeq \colim_{\substack{Y\in\Sm_S\\b\in B(Y)}}\Sigma^\infty_{\T,\fr} \h_S^\fr(Y,\beta(b)).
	\]
	We therefore want to show that, when $\beta$ is smooth, the map
	\[
	\colim_{\substack{Y\in\Sm_S\\b\in B(Y)}}\Sigma^\infty_{\T,\fr} \h_S^\fr(Y,\beta(b))
	\to
	\colim_{\substack{Y\in\dSch_S\\b\in B(Y)}}\Sigma^\infty_{\T,\fr} \h_S^\fr(Y,\beta(b))
	\]
	induced by the inclusion $(\Sm_S)_{/B} \subset (\dSch_S)_{/B}$ is an equivalence.
	By Proposition~\ref{prop:beta-descent}, we may as well assume that $B$ is the left Kan extension of $B|\Sm_S$. In that case, we claim that the inclusion $(\Sm_S)_{/B} \subset (\dSch_S)_{/B}$ is right adjoint and hence cofinal. Indeed, we have
	\[
	B(Y) = \colim_{\substack{Y'\in\Sm_S\\Y\to Y'}}B(Y'),
	\]
	hence $(\dSch_S)_{/B}\simeq C_{/B\circ d_1}$ where $C\subset \Fun(\Delta^1,\dSch_S)$ is the full subcategory of morphisms whose codomain is smooth. Forgetting the domain is then left adjoint to the above inclusion.
\end{proof}

\begin{cor}\label{cor:thom-general}
	Let $k$ be a perfect field and $\beta\colon B\to K_{\geq 0}$ a smooth stable tangential structure over $k$. Then there is an equivalence
	\[
	\Omega^\infty_{\T,\fr}M\beta \simeq L_\zar \Lhtp(\FQSM_k^{\beta})^\gp
	\]
	in $\H^\fr(k)$, natural and symmetric monoidal in $\beta$.
	Moreover, if $\beta$ has rank $\geq 1$, $L_\nis\Lhtp \FQSM_k^{\beta}$ is already grouplike.
\end{cor}

\begin{proof}
	The first statement follows from Theorem~\ref{thm:thom-general} as in the proof of Corollary~\ref{cor:thom-bundles}.
	Let $\tilde\beta\colon \tilde B\to K_{\geq 0}$ be the left Kan extension of $\beta|\Sm_k$. Since $\beta$ is smooth, the map
	\[
	\FQSM_k^{\tilde \beta} \to \FQSM_k^\beta
	\]
	is a Nisnevich equivalence by Proposition~\ref{prop:beta-descent}. After applying $\Lhtp$, it remains an effective epimorphism on Nisnevich stalks. To prove that $L_\nis\Lhtp \FQSM_k^{\beta}$ is grouplike when $\beta$ has rank $\geq 1$, we may therefore replace $\beta$ by $\tilde\beta$ and assume that $B$ is left Kan extended along $\Sm_k\subset \dSch_k$. In this case, we have an equivalence
	\[
	\FQSM_k^{\beta} \simeq \colim_{\substack{Y\in\Sm_k\\b\in B(Y)}} \h_k^\fr(Y,\beta(b)).
	\]
	Since $\Lhtp$ preserves colimits as an endofunctor of $\Pre(\Sm_k)$, we have
	\[
	L_\nis\Lhtp \FQSM_k^{\beta} \simeq L_\nis \colim_{\substack{Y\in\Sm_k\\b\in B(Y)}} \Lhtp\h_k^\fr(Y,\beta(b))
	\simeq L_\nis \colim_{\substack{Y\in\Sm_k\\b\in B(Y)}} L_\nis\Lhtp\h_k^\fr(Y,\beta(b)).
	\]
	Since this colimit is sifted, it can be computed in $\Pre_\Sigma(\Span^\fr(\Sm_k))$. 
	By Corollary~\ref{cor:thom-bundles}, each sheaf $L_\nis\Lhtp\h_k^\fr(Y,\beta(b))$ is grouplike, so we conclude using the fact that grouplike objects are stable under colimits.
\end{proof}

\begin{cor}
	Let $k$ be a perfect field and $\beta\colon B\to K_{\geq 0}$ a smooth stable tangential structure over $k$.
	For any $n\geq 1$, there is an equivalence
	\[
	\Omega^n_\T L_\mot\FQSM_k^{\beta+n} \simeq L_\mot (\FQSM_k^\beta)^\gp
	\]
	in $\Pre_\Sigma(\Sm_k)$.
\end{cor}

\begin{proof}
	This follows immediately from Corollary~\ref{cor:thom-general} since $M(\beta+n)\simeq \Sigma^n_\T M\beta$.
\end{proof}

\begin{rem}\label{rem:fund-class-comparison}
	The equivalence of Theorem~\ref{thm:thom-general} admits a conditional description in terms of virtual fundamental classes as follows. For $\beta\colon B\to K$ a smooth stable tangential structure over $S$, let $\mathcal{PQS}\mathrm{m}_S^\beta$ denote the moduli stack of proper quasi-smooth $S$-schemes with $\beta$-structure. 
	There should exist a canonical morphism
	\[
	\fc_\beta\colon \mathcal{PQS}\mathrm m_S^{\beta} \to \Omega^\infty_{\T,\fr}M\beta
	\]
	in $\Pre_\Sigma(\Span^\fr(\Sm_S))$ sending $f\colon Z\to X$ to the image by the Gysin transfer $f_!\colon M\beta(Z,\sL_f) \to M\beta(X)$ of the Thom class $t_\beta(-\sL_f)\in M\beta(Z,\sL_f)$ determined by the lift of $-\sL_f$ to $B(Z)$. Assuming the existence of $\fc_\beta$ with evident naturality and multiplicativity properties, one can show that the equivalence of Theorem~\ref{thm:thom-general} is adjoint to $\fc_\beta|\FQSM_S^\beta$. Indeed, it is enough to prove this when $\beta$ is the class of a finite locally free sheaf $\sE$ on a smooth $S$-scheme $Y$. In this case, it is easy to check that the composite
	\[
	\h^\fr_S(\bV(\sE)/\bV^\times(\sE)) \xrightarrow{\Theta_{Y/S,\sE}} \h^\fr_S(Y,\sE) \xrightarrow{\fc_\sE} \Omega^\infty_{\T,\fr}\Sigma^\infty_\T(\bV(\sE)/\bV^\times(\sE))
	\]
	is adjoint to the identity.
	
	This can partially be made precise for $\beta$ of rank $0$. Gysin transfers for regular closed immersions between classical schemes are constructed in \cite{DJKFundamental}. Using the canonical factorization $Z\hook \bV(f_*\sO_Z)\to X$ of a finite syntomic morphism $f\colon Z\to X$, one can construct a morphism
	\begin{equation*}
		\fc_\beta\colon \FSYN_S^{\beta} \to \Omega^\infty_{\T}M\beta
	\end{equation*}
	in $\CMon(\Pre_\Sigma(\Sm_S))$, which is natural in $\beta$ (cf.\ \cite[\sectsign 3.1]{EHKSY2}). Unfortunately, this does not suffice to carry out the above argument.
	When $\beta=0\in K(Y)$ for some smooth $S$-scheme $Y$, the assertion that $\fc_\beta$ is induced by the equivalence of Theorem~\ref{thm:thom-general} is nevertheless verified by \cite[Theorem 3.3.10]{EHKSY2}.
\end{rem}

\begin{rem}\label{rem:quick-proof}
	For $\beta$ a smooth stable tangential structure of rank $0$ over a perfect field $k$, one can obtain more directly an equivalence $\Omega^\infty_\T M\beta \simeq L_\zar \Lhtp(\FSYN_k^\beta)^\gp$ using the morphism
	\[
	\fc_\beta\colon \FSYN_k^\beta \to \Omega^\infty_\T M\beta
	\]
	from Remark~\ref{rem:fund-class-comparison}. By \cite[Theorem 3.3.10]{EHKSY2} and \cite[Corollary 3.5.16]{EHKSY1}, $\fc_\beta^\gp$ is a motivic equivalence when $\beta=0\in K(Y)$ for some smooth $k$-scheme $Y$. It follows that $\fc_\beta^\gp$ is a motivic equivalence for any $\beta\in K(Y)$ of rank $0$, since the question is local on $Y$. Finally, since $\fc_\beta$ is natural in $\beta$ and $\Omega^\infty_\T$ preserves sifted colimits \cite[Corollary 3.5.15]{EHKSY1}, we deduce that $\fc_\beta^\gp$ is a motivic equivalence for any $\beta$ of rank $0$. However, this approach does not suffice to understand $\MGL$-modules over perfect fields (as in Theorem~\ref{thm:MGL-modules2} below), because we do not yet know if the morphism $\fc_\beta$ can be made symmetric monoidal in $\beta$, nor if it can be promoted to a morphism of presheaves with framed transfers.
\end{rem}

\begin{quest}
	Following the discussion in Remark~\ref{rem:fund-class-comparison}, it is natural to ask the following questions:
	\begin{enumerate}
		\item For $\beta$ a smooth stable tangential structure of rank $\geq 0$ over a perfect field $k$, is the inclusion $\FQSM_k^\beta\subset \mathcal{PQS}\mathrm m_k^\beta$ a motivic equivalence after group completion?
		\item For $\beta$ a smooth stable tangential structure of arbitrary rank over $k$, is there an equivalence $\Omega^\infty_\T M\beta \simeq L_\mot(\mathcal{PQS}\mathrm m_k^\beta)^\gp$ ?
	\end{enumerate}
	Regarding (1), Remark~\ref{rem:fund-class-comparison} implies that $L_\mot(\FQSM_k^\beta)^\gp$ is a direct factor of $L_\mot(\mathcal{PQS}\mathrm m_k^\beta)^\gp$. Moreover, for $\beta=\id_K$, one can show that the $\A^1$-localization $\Lhtp\mathcal{PQS}\mathrm m_k$ is already grouplike (the proof will appear elsewhere). An affirmative answer to (1) would therefore imply that $L_\mot\mathcal{PQS}\mathrm m_k^0$ is the group completion of $L_\mot\FSYN_k$.
\end{quest}

\ssec{Algebraic cobordism spectra}
\label{ssec:MGL}

\begin{thm}\label{thm:MGL}
	Let $S$ be a scheme. 
	
	\noindent{\em(i)}
	There is an equivalence of $\Einfty$-ring spectra
	\[
	\MGL_S \simeq \Sigma^\infty_{\T,\fr} \FSYN_S
	\]
	in $\SH(S)\simeq \SH^\fr(S)$.
	
	\noindent{\em(ii)}
	For every $n\geq 1$, there is an equivalence of $\MGL_S$-modules
	\[
	\Sigma^n_\T\MGL_S \simeq \Sigma^\infty_{\T,\fr}\FQSM_S^{n}
	\]
	in $\SH(S)\simeq \SH^\fr(S)$.
	
	\noindent{\em(iii)}
	There is an equivalence of $\Einfty$-ring spectra
	\[
	\bigvee_{n\geq 0}\Sigma_\T^n\MGL_S \simeq \Sigma^\infty_{\T,\fr}\FQSM_S
	\]
	in $\SH(S)\simeq \SH^\fr(S)$.
\end{thm}

\begin{proof}
	These are instances of Theorem~\ref{thm:thom-general}, where $\beta$ is the inclusion of the rank $n$ summand of $K$-theory (for (i) and (ii)) or the identity map $K_{\geq 0}\to K_{\geq 0}$ (for (iii)). Indeed, these summands of K-theory satisfy the assumption of Lemma~\ref{lem:local-LKE} by Example~\ref{ex:LKE-rank}.
\end{proof}

\begin{cor}\label{cor:MGL}
	Let $S$ be a pro-smooth scheme over a field. 
	
	\noindent{\em(i)}
	There is an equivalence of $\Einfty$-ring spaces
	\[
	\Omega^\infty_{\T,\fr}\MGL_S \simeq L_\zar\Lhtp (\FSYN_S)^\gp
	\]
	in $\H^\fr(S)$.
	
	\noindent{\em(ii)}
	For every $n\geq 1$, there are equivalences of $\FSYN_S$-modules
	\[
	\Omega^\infty_{\T,\fr}\Sigma^n_\T\MGL_S \simeq L_\zar\Lhtp(\FQSM_S^{n})^\gp\simeq L_\nis\Lhtp \FQSM_S^{n}
	\]
	in $\H^\fr(S)$.
	
	\noindent{\em(iii)}
	There is an equivalence of $\Einfty$-ring spaces
	\[
	\Omega^\infty_{\T,\fr}\left(\bigvee_{n\geq 0}\Sigma_\T^n\MGL_S\right) \simeq L_\zar\Lhtp(\FQSM_S)^\gp
	\]
	in $\H^\fr(S)$.
\end{cor}

\begin{proof}
	When $S$ is the spectrum of a perfect field, these statements are instances of Corollary~\ref{cor:thom-general}. 
	In general, we can choose a pro-smooth morphism $f\colon S\to \Spec k$ where $k$ is a perfect field, and the results over $k$ pull back to the results over $S$.
\end{proof}

Let us also spell out the specializations of Theorem~\ref{thm:thom-general} and Corollary~\ref{cor:thom-general} to the smooth stable tangential structure of Example~\ref{ex:SL-orientation}.

\begin{thm}\label{thm:MSL}
	Let $S$ be a scheme. 
	
	\noindent{\em(i)}
	There is an equivalence of $\Einfty$-ring spectra
	\[
	\MSL_S \simeq \Sigma^\infty_{\T,\fr} \FSYN_S^{\ornt}
	\]
	in $\SH(S)\simeq \SH^\fr(S)$.
	
	\noindent{\em(ii)}
	For every $n\geq 1$, there is an equivalence of $\MSL_S$-modules
	\[
	\Sigma^n_\T\MSL_S \simeq \Sigma^\infty_{\T,\fr}\FQSM_S^{\ornt,n}
	\]
	in $\SH(S)\simeq \SH^\fr(S)$.
	
	\noindent{\em(iii)}
	There is an equivalence of $\sE_1$-ring spectra
	\[
	\bigvee_{n\geq 0}\Sigma_\T^{n}\MSL_S \simeq \Sigma^\infty_{\T,\fr}\FQSM_S^{\ornt}
	\]
	restricting to an equivalence of $\Einfty$-ring spectra
	\[
	\bigvee_{n\geq 0}\Sigma_\T^{2n}\MSL_S \simeq \Sigma^\infty_{\T,\fr}\FQSM_S^{\ornt,\ev}
	\]
	in $\SH(S)\simeq \SH^\fr(S)$.
\end{thm}

\begin{cor}\label{cor:MSL}
	Let $S$ be a pro-smooth scheme over a field. 
	
	\noindent{\em(i)}
	There is an equivalence of $\Einfty$-ring spaces
	\[
	\Omega^\infty_{\T,\fr}\MSL_S \simeq L_\zar\Lhtp (\FSYN_S^{\ornt})^\gp
	\]
	in $\H^\fr(S)$.
	
	\noindent{\em(ii)}
	For every $n\geq 1$, there are equivalences of $\FSYN_S^\ornt$-modules
	\[
	\Omega^\infty_{\T,\fr}\Sigma^n_\T\MSL_S \simeq L_\zar\Lhtp(\FQSM_S^{\ornt,n})^\gp\simeq L_\nis\Lhtp \FQSM_S^{\ornt,n}
	\]
	in $\H^\fr(S)$.
	
	\noindent{\em(iii)}
	There is an equivalence of $\sE_1$-ring spaces
	\[
	\Omega^\infty_{\T,\fr}\left(\bigvee_{n\geq 0}\Sigma_\T^n\MSL_S\right) \simeq L_\zar\Lhtp(\FQSM_S^{\ornt})^\gp
	\]
	restricting to an equivalence of $\Einfty$-ring spaces
	\[
	\Omega^\infty_{\T,\fr}\left(\bigvee_{n\geq 0}\Sigma_\T^{2n}\MSL_S\right) \simeq L_\zar\Lhtp(\FQSM_S^{\ornt,\ev})^\gp
	\]
	in $\H^\fr(S)$.
\end{cor}

\ssec{Hilbert scheme models}
\label{ssec:hilbert}

Using the $\A^1$-contractibility of the space of embeddings of a finite scheme into $\A^\infty$ (Corollary~\ref{cor:Emb-contractible}), we can recast our models for $\Omega^\infty_\T\MGL$ and $\Omega^\infty_\T\MSL$ (and others) in terms of Hilbert schemes, at the cost of losing the identification of the framed transfers and of the multiplicative structures.

Let $X$ be an $S$-scheme. We define the functor $\Hilb^{\fqs}(X/S)\colon \Sch_S^\op \to \Spc$ by
\[
\Hilb^{\fqs}(X/S)(T) = \{\text{closed immersions $Z\to X_T$ such that $Z\to T$ is finite and quasi-smooth}\},
\]
and we denote by $\Hilb^{\fqs,n}(X/S)$ the subfunctor where $Z\to T$ has relative virtual dimension $-n$ (which is contractible unless $n\geq 0$). By Lemma~\ref{lem:classical}, we have
\[
\Hilb^{\fqs,0}(X/S) = \Hilb^{\flci}(X/S).
\]
In particular, if $X$ is smooth and quasi-projective over $S$, then $\Hilb^{\fqs,0}(X/S)$ is representable by a smooth $S$-scheme (see \cite[Lemma 5.1.3]{EHKSY1}).

We also define
\[
\Hilb^{\fqs}(\A^\infty_S/S) = \colim_{n\to\infty} \Hilb^{\fqs}(\A^n_S/S),
\]
and similarly for $\Hilb^{\fqs,n}(\A^\infty_S/S)$.

\begin{lem}\label{lem:Hilb}
	Let $S$ be a scheme. Then the forgetful map
	\[
	\Hilb^{\fqs}(\A^\infty_S/S) \to \FQSM_S
	\]
	is a universal $\A^1$-equivalence on affine schemes (i.e., any pullback of this map in $\Pre(\Sch_S)$ is an $\A^1$-equivalence on affine schemes).
\end{lem}

\begin{proof}
	Let $f\colon T\to S$ and let $Z\in\FQSM_S(T)$. Form the Cartesian square
	\[
	\begin{tikzcd}
		P_Z \ar{r} \ar{d} & T \ar{d}{Z} \\
		\Hilb^{\fqs}(\A^\infty_S/S) \ar{r} & \FQSM_S\rlap.
	\end{tikzcd}
	\]
	By universality of colimits, it suffices to show that $P_Z\to T$ is an $\A^1$-equivalence on affine schemes. By inspection, $P_Z$ is $f_\sharp$ of the presheaf $T'\mapsto \Emb_{T'}(Z\times_TT',\A^\infty_{T'})$ on $\Sch_T$, and the latter is $\A^1$-contractible on affine schemes by Corollary~\ref{cor:Emb-contractible}.
\end{proof}

\begin{thm}\label{thm:MGL-Hilb}
	Suppose $S$ is pro-smooth over a field. 
	
	\noindent{\em(i)}
	There is an equivalence
	\[
	\Omega^\infty_\T\MGL_S \simeq L_\zar(\Lhtp\Hilb^\flci(\A^\infty_S/S))^\gp.
	\]
	
	\noindent{\em(ii)}
	For every $n\geq 1$, there are equivalences
	\[
	\Omega^\infty_\T\Sigma^n_\T\MGL_S \simeq L_\zar(\Lhtp\Hilb^{\fqs,n}(\A^\infty_S/S))^\gp\simeq L_\nis \Lhtp\Hilb^{\fqs,n}(\A^\infty_S/S).
	\]
\end{thm}

\begin{proof}
	This follows immediately from Corollary~\ref{cor:MGL} and Lemma~\ref{lem:Hilb}.
\end{proof}

Define the functor $\Hilb^{\ornt,n}(X/S)\colon \Sch_S^\op\to \Spc$ and the forgetful map 
\[
\Hilb^{\ornt,n}(X/S)\to \Hilb^{\fqs,n}(X/S)
\] 
so that the fiber over $Z\in \Hilb^{\fqs,n}(X/S)(T)$ is the $\infty$-groupoid of equivalences $\det(\sL_{Z/T})\simeq \sO_Z$.
In other words, $\Hilb^{\ornt,n}(X/S)$ is the Weil restriction of the $\G_m$-torsor $\Isom(\det(\sL_{\sZ}),\sO_{\sZ})$ over the universal $\sZ$. In particular, if $X$ is smooth and quasi-projective over $S$, then $\Hilb^{\ornt,0}(X/S)$ is representable by a smooth $S$-scheme.

\begin{thm}\label{thm:MSL-Hilb}
	Suppose $S$ is pro-smooth over a field. 
	
	\noindent{\em(i)}
	There is an equivalence
	\[
	\Omega^\infty_\T\MSL_S \simeq L_\zar(\Lhtp\Hilb^{\ornt,0}(\A^\infty_S/S))^\gp.
	\]
	
	\noindent{\em(ii)}
	For every $n\geq 1$, there are equivalences
	\[
	\Omega^\infty_\T\Sigma^n_\T\MSL_S \simeq L_\zar(\Lhtp\Hilb^{\ornt,n}(\A^\infty_S/S))^\gp\simeq L_\nis \Lhtp\Hilb^{\ornt,n}(\A^\infty_S/S).
	\]
\end{thm}

\begin{proof}
	This follows immediately from Corollary~\ref{cor:MSL} and Lemma~\ref{lem:Hilb}, noting that
	\[
	\Hilb^{\ornt,n}(\A^\infty_S/S) \simeq \Hilb^{\fqs}(\A^\infty_S/S) \times_{\FQSM_S} \FQSM_S^{\ornt,n}.\qedhere
	\]
\end{proof}

For any smooth stable tangential structure $\beta\colon B\to K_{\geq 0}$, Lemma~\ref{lem:Hilb} gives a description of $\Omega^\infty_\T M\beta$ in terms of the functor classifying derived subschemes $Z$ of $\A^\infty$ with some structure on the image of the shifted cotangent complex $\sL_Z[-1]$ in K-theory. However, it is perhaps more natural to classify derived subschemes $Z$ of $\A^\infty$ with some structure on the conormal sheaf $\sN_{Z/\A^\infty}\in \sVect(Z)$.
If $\beta\colon B\to \sVect_{\geq 0}$ is a morphism in $\Pre_\Sigma(\dSch_S)$, we define the functor $\Hilb^\beta(\A^\infty_S/S)\colon \Sch_S^\op\to\Spc$ by the Cartesian squares
\[
\begin{tikzcd}
	B(Z)\times_{\sVect(Z)}\{\sN_{Z/\A^\infty_T}\} \ar{r} \ar{d} & T \ar{d}{Z} \\
	\Hilb^{\beta}(\A^\infty_S/S)(T) \ar{r} & \Hilb^{\fqs}(\A^\infty_S/S)\rlap.
\end{tikzcd}
\]

\begin{lem}\label{lem:Hilb-beta}
	Let $S$ be a scheme and $\beta\colon B\to \sVect_{\geq 0}$ a stable tangential structure over $S$. Then the forgetful map
	\[
	\Hilb^\beta(\A^\infty_S/S) \to \FQSM_S^\beta
	\]
	is an $\A^1$-equivalence on affine schemes.
\end{lem}

\begin{proof}
	This map is the colimit of the maps
		\[
		\h^\nfr_S(Y,\beta(b)) \to \h^\fr_S(Y,\beta(b))
		\]
	over $Y\in\dSch_S$ and $b\in B(Y)$, which are $\A^1$-equivalences on affine schemes by Corollary~\ref{cor:nfr-vs-dfr}.
\end{proof}

\begin{thm}\label{thm:Mbeta-Hilb}
	Let $k$ be a perfect field and $\beta\colon B\to \sVect_{\geq 0}$ a smooth stable tangential structure over $k$. Then there is an equivalence
	\[
	\Omega^\infty_\T M\beta \simeq L_\zar(\Lhtp \Hilb^\beta(\A^\infty_k/k))^\gp.
	\]
	Moreover, if $\beta$ has rank $\geq 1$, $L_\nis\Lhtp\Hilb^\beta(\A^\infty_k/k)$ is already grouplike.
\end{thm}

\begin{proof}
	This follows immediately from Lemma~\ref{lem:Hilb-beta} and Corollary~\ref{cor:thom-general}.
\end{proof}

One can recover Theorems~\ref{thm:MGL-Hilb} and \ref{thm:MSL-Hilb} from Theorem~\ref{thm:Mbeta-Hilb} using the motivic equivalences $\sVect\to K$ and $\sVect^\SL\to K^\SL$ and the fact that the functor $M\colon \Pre(\Sm_S)_{/K} \to \SH(S)$ inverts motivic equivalences \cite[Remark 16.11]{norms}.

\section{Modules over algebraic cobordism}
\label{sec:modules}

In this section, we show that modules over motivic Thom ring spectra can be described as motivic spectra with certain transfers.
We first treat the case of $\MGL$ in \ssecref{ssec:MGL-modules}, where we construct a symmetric monoidal equivalence between $\MGL$-modules and motivic spectra with finite syntomic transfers. We then treat the case of $\MSL$ and explain the general case in \ssecref{ssec:MSL-modules}. Finally, in \ssecref{ssec:HZ}, we describe the motivic cohomology spectrum $H\Z$, which is an $\MGL$-module, as a motivic spectrum with finite syntomic transfers: it is the suspension spectrum of the constant sheaf $\Z$ equipped with canonical finite syntomic transfers.

It is worth pointing out that, although the theorems in this section do not involve any derived algebraic geometry, their proofs use derived algebraic geometry in an essential way (via Section~\ref{sec:thom}).

\ssec{Modules over MGL}
\label{ssec:MGL-modules}

Let $\Span^\fsyn(\Sm_S)$ denote the symmetric monoidal $(2,1)$-category whose objects are smooth $S$-schemes and whose morphisms are spans
\[
\begin{tikzcd}
   & Z \ar[swap]{ld}{f}\ar{rd}{g} & \\
  X &   & Y
\end{tikzcd}
\]
where $f$ is finite syntomic. Let $\H^\fsyn(S)$ denote the full subcategory of $\Pre_\Sigma(\Span^\fsyn(\Sm_S))$ spanned by the $\A^1$-invariant Nisnevich sheaves, and let $\SH^\fsyn(S)$ be the symmetric monoidal $\infty$-category of $\T$-spectra in $\H^\fsyn(S)$. We have the usual adjunction
\[
\Sigma^\infty_{\T,\fsyn} : \H^\fsyn(S) \rightleftarrows \SH^\fsyn(S) : \Omega^\infty_{\T,\fsyn}.
\]
The symmetric monoidal forgetful functor
\[
\epsilon\colon \Span^\fr(\Sm_S) \to \Span^\fsyn(\Sm_S)
\]
(see \cite[4.3.15]{EHKSY1}) induces symmetric monoidal colimit-preserving functors 
\[
\epsilon^*\colon \H^\fr(S)\to \H^\fsyn(S)\quad\text{and}\quad \epsilon^*\colon \SH^\fr(S)\to \SH^\fsyn(S).
\]
For clarity, we will denote the tensor products in $\H^\fr(S)$ and $\SH^\fr(S)$ by $\otimes^\fr$ and the ones in $\H^\fsyn(S)$ and $\SH^\fsyn(S)$ by $\otimes^\fsyn$.

We denote by $\h^\fsyn_S(X)$ the presheaf on $\Span^\fsyn(\Sm_S)$ represented by $X\in\Sm_S$.

\begin{lem}\label{lem:epsilon1}
	The forgetful functor $\epsilon_*\colon \H^\fsyn(S)\to \H^\fr(S)$ is a strict $\H^\fr(S)$-module functor. In other words, for any $A\in \H^\fr(S)$ and $B\in\H^\fsyn(S)$, the canonical map
	\[
	A\otimes^\fr \epsilon_*(B) \to \epsilon_*(\epsilon^*(A)\otimes^\fsyn B)
	\]
	is an equivalence.
\end{lem}

\begin{proof}
Since $\epsilon_*$ preserves colimits, we can assume that $A=\h^\fr_S(X)$ and $B=\h^\fsyn_S(Y)$ for some smooth $S$-schemes $X$ and $Y$. Since the stable tangential structure $\iota_0$ is smooth, we have by Proposition~\ref{prop:beta-descent} a Nisnevich equivalence
	\[
	\colim_{(Z,\xi)} \h_S^\fr(Y\times_S Z,\pi_Z^*(\xi))\to \h_S^\fsyn(Y),
	\]
	where the colimit is over all $Z\in\Sm_S$ and $\xi\in K(Z)$ of rank $0$. Hence, it suffices to show that the map
	\[
	\h^\fr_S(X) \otimes^\fr \h_S^\fr(Y\times_SZ, \pi_Z^*(\xi)) \to \h^\fr_S(X\times_SY\times_SZ,\pi_Z^*(\xi))
	\]
	is a motivic equivalence for all such pairs $(Z,\xi)$. Since the question is local on $Z$ (by Propositions \ref{prop:fr-descent}(iii) and \ref{prop:fr-additivity}), we can assume $\xi=0$, in which case it is obvious.
\end{proof}

\begin{lem}\label{lem:epsilon2}
	$\Sigma^\infty_{\T,\fr}\epsilon_* \simeq \epsilon_*\Sigma^\infty_{\T,\fsyn}$.
\end{lem}

\begin{proof}
	By Lemma~\ref{lem:epsilon1}, the $\T$-stable adjunction
	\[
	\epsilon^*: \SH^\fr(S) \rightleftarrows \SH^\fsyn(S) : \epsilon_*
	\]
	is obtained from the unstable one by extending scalars along $\Sigma^\infty_{\T,\fr}\colon \H^\fr(S)\to \SH^\fr(S)$. This immediately implies the result.
\end{proof}

\begin{thm}\label{thm:MGL-modules}
	Let $S$ be a scheme. There is an equivalence of symmetric monoidal $\infty$-categories
	\[
	\Mod_{\MGL}(\SH(S)) \simeq \SH^{\fsyn}(S),
	\]
	natural in $S$ and compatible with the forgetful functors to $\SH(S)$.
\end{thm}

\begin{proof}
	By Theorem~\ref{thm:MGL}(i), we have an equivalence of motivic $\sE_\infty$-ring spectra
	\[
	\MGL_S \simeq \Sigma^\infty_{\T,\fr}\h^\fsyn_S(S).
	\]
	By Lemma~\ref{lem:epsilon2}, the right-hand side is $\epsilon_*\Sigma^\infty_{\T,\fsyn}\h^\fsyn_S(S)$, which means that $\MGL_S$ is the image of the unit by the forgetful functor $\SH^\fsyn(S)\to\SH(S)$.
	We therefore obtain an adjunction
	\[
	\begin{tikzcd}
		\Mod_{\MGL}(\SH(S)) \ar[shift left=1]{r}{\Phi} & \SH^\fsyn(S) \ar[shift left=1]{l}{\Psi}
	\end{tikzcd}
	\]
	where $\Phi$ is symmetric monoidal and $\Psi$ is conservative. It remains to show that the unit map
	\[
	\MGL_S\otimes \Sigma^\infty_\T Y_+ \to \Psi \Phi(\MGL_S\otimes \Sigma^\infty_\T Y_+)\simeq \Psi\Sigma^\infty_{\T,\fsyn} \h^\fsyn_S(Y)
	\]
	is an equivalence for every $Y\in \Sm_S$.
	By Lemma~\ref{lem:epsilon2} again, this map is $\Sigma^\infty_{\T,\fr}$ of the map
	\[
	\h_S^\fsyn(S) \otimes^\fr \h^\fr_S(Y) \to \h^\fsyn_S(Y),
	\]
	which is an equivalence by Lemma~\ref{lem:epsilon1}.
\end{proof}

\begin{thm}\label{thm:MGL-modules2}
	Let $k$ be a perfect field. 
	
	\noindent{\em(i)}
	There is an equivalence of symmetric monoidal $\infty$-categories
	\[
	\Mod_{\MGL}(\SH^\veff(k)) \simeq \H^\fsyn(k)^\gp
	\]
	under $\SH^\veff(k)\simeq \H^\fr(k)^\gp$.
	
	\noindent{\em(ii)}
	There is an equivalence of symmetric monoidal $\infty$-categories
	\[
	\Mod_{\MGL}(\SH^\eff(k)) \simeq \SH^{S^1,\fsyn}(k)
	\]
	under $\SH^\eff(k)\simeq \SH^{S^1,\fr}(k)$.
\end{thm}

\begin{proof}
	The proof of (i) is exactly the same as that of Theorem~\ref{thm:MGL-modules}, using Corollary~\ref{cor:MGL} instead of Theorem~\ref{thm:MGL}. We obtain (ii) from (i) by stabilizing.
\end{proof}

As a corollary, we obtain a cancellation theorem for $\A^1$-invariant sheaves with finite syntomic transfers over perfect fields:

\begin{cor}\label{cor:fsyn-cancellation}
	Let $k$ be a perfect field. Then the $\infty$-category $\H^\fsyn(k)^\gp$ is prestable and the functor $\Sigma_{\G}\colon \H^\fsyn(k)^\gp \to \H^\fsyn(k)^\gp$ is fully faithful.
\end{cor}

\ssec{Modules over MSL}
\label{ssec:MSL-modules}

We have completely analogous results for $\MSL$ instead of $\MGL$. Consider the symmetric monoidal $(2,1)$-category $\Span^\ornt(\Sm_S)$ whose objects are smooth $S$-schemes and whose morphisms are spans
\[
\begin{tikzcd}
   & Z \ar[swap]{ld}{f}\ar{rd}{g} & \\
  X &   & Y
\end{tikzcd}
\]
where $f$ is finite syntomic together with an isomorphism $\omega_f\simeq \sO_Z$. We can form as usual the symmetric monoidal $\infty$-categories $\H^\ornt(S)$ and $\SH^\ornt(S)$. The following results are proved in the same way as the corresponding results from \ssecref{ssec:MGL-modules}.

\begin{thm}\label{thm:MSL-modules}
	Let $S$ be a scheme. There is an equivalence of symmetric monoidal $\infty$-categories
	\[
	\Mod_{\MSL}(\SH(S)) \simeq \SH^{\ornt}(S),
	\]
	natural in $S$ and compatible with the forgetful functors to $\SH(S)$.
\end{thm}

\begin{thm}\label{thm:MSL-modules2}
	Let $k$ be a perfect field. 
	
	\noindent{\em(i)}
	There is an equivalence of symmetric monoidal $\infty$-categories
	\[
	\Mod_{\MSL}(\SH^\veff(k)) \simeq \H^\ornt(k)^\gp
	\]
	under $\SH^\veff(k)\simeq \H^\fr(k)^\gp$.
	
	\noindent{\em(ii)}
	There is an equivalence of symmetric monoidal $\infty$-categories
	\[
	\Mod_{\MSL}(\SH^\eff(k)) \simeq \SH^{S^1,\ornt}(k)
	\]
	under $\SH^\eff(k)\simeq \SH^{S^1,\fr}(k)$.
\end{thm}

\begin{cor}\label{cor:or-cancellation}
	Let $k$ be a perfect field. Then the $\infty$-category $\H^\ornt(k)^\gp$ is prestable and the functor $\Sigma_{\G}\colon \H^\ornt(k)^\gp \to \H^\ornt(k)^\gp$ is fully faithful.
\end{cor}

\begin{rem}
	There are analogs of the above results for any $\sE_1$ smooth stable tangential structure $\beta$ of rank $0$ over $S$. Indeed, one can construct an $\infty$-category $\Span^\beta(\Sm_S)$ of $\beta$-structured finite syntomic correspondences using the formalism of labeling functors from \cite[\sectsign 4.1]{EHKSY1}, in a manner similar to the construction of $\Span^\fr(\Sm_S)$. Then for $S$ arbitrary and $k$ a perfect field, we have equivalences of $\infty$-categories
	\begin{gather*}
		\Mod_{M\beta}(\SH(S)) \simeq \SH^\beta(S),\\
		\Mod_{M\beta}(\SH^\veff(k)) \simeq \H^\beta(k)^\gp,\\
		\Mod_{M\beta}(\SH^\eff(k)) \simeq \SH^{S^1,\beta}(k),
	\end{gather*}
	which are symmetric monoidal if $\beta$ is $\Einfty$. Moreover, the $\infty$-category $\H^\beta(k)^\gp$ is prestable and the functor $\Sigma_{\G}\colon \H^\beta(k)^\gp \to \H^\beta(k)^\gp$ is fully faithful.
\end{rem}

\ssec{Motivic cohomology as an MGL-module}
\label{ssec:HZ}

For $A$ a commutative monoid, let $A_S$ denote the corresponding constant sheaf on $\Sm_S$, which is an $\A^1$-invariant Nisnevich sheaf. The sheaf $A_S$ has canonical finite locally free transfers \cite[Proposition 13.13]{norms}, and in particular finite syntomic transfers.

Let $H\Z_S\in\SH(S)$ be the motivic cohomology spectrum defined by Spitzweck \cite{SpitzweckHZ}, which is an $\Einfty$-algebra in $\Mod_{\MGL}(\SH(S))$ \cite[Remark 10.2]{SpitzweckHZ}. The following theorem is a refinement of \cite[Theorem 21]{framed-loc}:

\begin{thm}\label{thm:HZ}
	For any scheme $S$, there is an equivalence of $\Einfty$-algebras
	\[
	H\Z_S \simeq \Sigma^\infty_{\T,\fsyn} \Z_S
	\]
	in $\Mod_{\MGL}(\SH(S)) \simeq \SH^\fsyn(S)$.
\end{thm}

\begin{proof}
	We first note that the right-hand side is stable under base change, since $\Sigma^\infty_{\T,\fr}\epsilon_*(\Z_S)$ is \cite[Lemma 20]{framed-loc}, $\Sigma^\infty_{\T,\fr}\epsilon_* \simeq \epsilon_*\Sigma^\infty_{\T,\fsyn}$ (Lemma~\ref{lem:epsilon2}), and $\epsilon_*$ commutes with base change (by Theorem~\ref{thm:MGL-modules}).
	We can therefore assume that $S$ is a Dedekind domain. In this case, $\Omega^\infty_{\T,\fsyn}H\Z_S$ is the constant sheaf of rings $\Z_S$ with some finite syntomic transfers. As shown in the proof of \cite[Theorem 21]{framed-loc}, these transfers are the canonical ones for framed finite syntomic correspondences. Since $\Z_S$ is a discrete constant sheaf and every finite syntomic morphism $Z\to X$ can be framed Zariski-locally on $X$, we deduce that $\Omega^\infty_{\T,\fsyn}H\Z_S$ is $\Z_S$ with its canonical finite syntomic transfers. By adjunction, we obtain a morphism of $\Einfty$-algebras
	\[
	\phi_S\colon \Sigma^\infty_{\T,\fsyn} \Z_S \to H\Z_S
	\]
	in $\Mod_{\MGL}(\SH(S))$, which is stable under base change. It thus suffices to show that $\phi_S$ is an equivalence when $S$ is the spectrum of a perfect field, but this follows from Theorem~\ref{thm:MGL-modules2}.
\end{proof}

Arguing as in \cite[Corollary 22]{framed-loc}, we obtain the following corollary:

\begin{cor}
	Let $S$ be a scheme and $A$ an abelian group (resp.\ a ring; a commutative ring). Then there is an equivalence of $H\Z_S$-modules (resp.\ of $\sE_1$-$H\Z_S$-algebras; of $\Einfty$-$H\Z_S$-algebras)
	\[
	HA_S \simeq \Sigma^\infty_{\T,\fsyn} A_S
	\]
	in $\Mod_{\MGL}(\SH(S)) \simeq \SH^\fsyn(S)$.
\end{cor}

\begin{rem}
	It follows from Theorem~\ref{thm:HZ} that the canonical morphism of $\Einfty$-ring spectra $\MGL_S\to H\Z_S$ is $\Sigma^\infty_{\T,\fsyn}$ of the degree map $\deg\colon \FSYN_S\to\N_S$.
\end{rem}

\appendix
\section{Functors left Kan extended from smooth algebras}
\label{app:LKE}

A surprising observation due to Bhatt and Lurie is that algebraic K-theory, as a functor on commutative rings, is left Kan extended from smooth rings. In this appendix, we present a general criterion for a functor on commutative rings to be left Kan extended from smooth rings, which we learned from Akhil Mathew, and we apply it to deduce some variants of the result of Bhatt and Lurie that are relevant for the applications of Theorem~\ref{thm:thom-general}.

A morphism of derived commutative rings $f\colon A\to B$ is called a \emph{henselian surjection} if $\pi_0(f)$ is surjective and $(\pi_0(A),\ker\pi_0(f))$ is a henselian pair \cite[Tag 09XD]{stacks}.

\begin{prop}[Mathew]\label{prop:akhil}
	Let $R$ be a commutative ring (resp.\ a derived commutative ring) and $F\colon \CAlg_R^\heart \to \Spc$ (resp.\ $F\colon \CAlg_R^\Delta \to \Spc$) a functor. Suppose that:
	\begin{enumerate}
		\item $F$ preserves filtered colimits;
		\item for every henselian surjection $A\to B$, the map $F(A)\to F(B)$ is an effective epimorphism (i.e., surjective on $\pi_0$);
		\item for every henselian surjections $A\to C \leftarrow B$, 
		the square
		\[
		\begin{tikzcd}
			F(A\times_CB) \ar{r} \ar{d} & F(B) \ar{d} \\
			F(A) \ar{r} & F(C)
		\end{tikzcd}
		\]
		is Cartesian.
	\end{enumerate}
	Then $F$ is left Kan extended from $\CAlg_R^\sm$.
\end{prop}

\begin{rem}
	Conditions (1) and (2) of Proposition~\ref{prop:akhil} are also necessary, since they hold when $F=\Maps_R(S,-)$ for some smooth $R$-algebra $S$ \cite[Théorème I.8]{Gruson}. Condition (3), on the other hand, is not (for example, it fails for K-theory).
\end{rem}

\begin{proof}
	Let $\tilde F$ be the left Kan extension of $F|\CAlg_R^\sm$. Then $\tilde F$ is a colimit of functors satisfying conditions (1)–(3) and in particular it satisfies conditions (1) and (2).
	The canonical map $\tilde F\to F$ is an equivalence on smooth $R$-algebras, hence on ind-smooth $R$-algebras.
	For any $A\in\CAlg_R^\heart$ (resp.\ $A\in \CAlg_R^\Delta$), we can inductively construct an augmented simplicial object $B$ such that $B[\emptyset]=A$ and, for each $n\geq 0$, $B[\Delta^n]$ is ind-smooth and $B[\Delta^n]\to B[\partial\Delta^n]$ is a henselian surjection.
	To conclude, we prove that both $\tilde F$ and $F$ send $B$ to a colimit diagram. Since $\tilde F$ is a colimit of functors that satisfy (1)–(3), it will suffice to show that $F(B)$ is a colimit diagram.
	Henselian surjections are stable under pullback, so the map $B[L]\to B[K]$ is a henselian surjection for any inclusion of finite simplicial sets $K\subset L$. In particular, by (2), $F(B[L])\to F(B[K])$ is an effective epimorphism.
	
	Let $K$ be a finite nonsingular simplicial set. Then $K$ can be built from $\emptyset$ and simplices $\Delta^n$ by a finite sequence of pushouts, which are transformed by $B[-]$ into Cartesian squares of henselian surjections. By (3), we deduce that $F(B[K]) \simeq F(B)[K]$ for such $K$, since this is trivially true for $K=\emptyset$ and $K=\Delta^n$. 
	Applying this to $K=\partial\Delta^n$, we conclude that $F(B)[\Delta^n] \to F(B)[\partial\Delta^n]$ is an effective epimorphism, hence that $F(B)$ is a colimit diagram \cite[Lemma A.5.3.7]{SAG}.
\end{proof}

\begin{ex}
	$\SH^\omega\colon \CAlg_R^\heart\to \InftyCat$ is left Kan extended from $\CAlg_R^\sm$ (apply Proposition~\ref{prop:akhil} to $\Fun(\Delta^n,\SH(-)^\omega)^\simeq$ for $n\geq 0$).
\end{ex}

\begin{prop}[Mathew]\label{prop:akhil2}
	Let $R$ be a derived commutative ring and $X$ a smooth algebraic stack over $R$ with quasi-affine diagonal (e.g., a smooth quasi-separated algebraic space), viewed as a functor $X\colon \CAlg_R^\Delta\to \Spc$. Then $X$ is left Kan extended from $\CAlg_R^\sm$.
\end{prop}

\begin{proof}
	We check conditions (1)–(3) of Proposition~\ref{prop:akhil}.
	 Condition (1) holds because $X$ is locally of finite presentation, and condition (3) holds because $\Spec(A\times_CB)$ is the pushout of $\Spec(B)\leftarrow \Spec(C)\to\Spec(A)$ in the $\infty$-category of derived algebraic stacks \cite[Example 17.3.1.3]{SAG}.
	 It remains to check condition (2).
	 Let $A\to B$ be a henselian surjection between derived commutative $R$-algebras, and let $A_0 = \pi_0A \times_{\pi_0B}B$ be the relative $0$-truncation of $A$ over $B$. Then $A\to A_0$ induces an isomorphism on $\pi_0$, so we can write $A\simeq \lim_{n\geq 0} A_n$ where $A\to A_n$ is $n$-connective and $A_{n+1}\to A_n$ is a square-zero extension \cite[Lemma 17.3.6.4]{SAG}. Since $X$ is smooth, each induced map $X(A_{n+1})\to X(A_n)$ is an effective epimorphism \cite[Remark 17.3.9.2]{SAG}. 
	 Moreover, as $X$ is nilcomplete \cite[Proposition 5.3.7]{LurieThesis}, we have $X(A)\simeq \lim_n X(A_n)$ \cite[Proposition 17.3.2.4]{SAG}. Hence, the induced map $X(A) \to X(A_0)$ is an effective epimorphism.
	 Since $X(A_0) \simeq X(\pi_0A) \times_{X(\pi_0B)} X(B)$ by (3), it remains to show that $X(\pi_0A) \to X(\pi_0B)$ is an effective epimorphism. In other words, we can assume $A$ discrete and $B=A/I$ for some ideal $I\subset A$ such that $(A,I)$ is a henselian pair.
	
	Let $f\colon \Spec(A/I)\to X$ be a morphism over $R$. We must show that $f$ can be extended to $\Spec (A)$. Replacing $X$ by $X\otimes_RA$, we may as well assume that $R=A$; in particular, since $A$ is discrete, $X$ is classical.
	 Since $\Spec(A/I)$ is quasi-compact, we can replace $X$ by a quasi-compact open substack \cite[Tags 06FJ and 0DQQ]{stacks} and assume $X$ finitely presented over $A$.
	By condition (1), we can also assume that $I\subset A$ is a finitely generated ideal. Then the pair $(A,I)$ is a filtered colimit of pairs that are henselizations of pairs of finite type over $\Z$. By \cite[Proposition B.2]{RydhApprox}, we are reduced to the case where $(A,I)$ is the henselization of a pair of finite type over $\Z$. In this case, the map $A\to A^\wedge_I$ is regular \cite[Tags 0AH2, 0AH3, 07PX]{stacks}, hence is a filtered colimit of smooth morphisms $A\to B_\alpha$ by Popescu \cite[Tag 07GC]{stacks}. Since $X$ is smooth over $A$, we can compatibly extend $f$ to $\Spec(A/I^n)$ for all $n$. As $X$ is now Noetherian with quasi-affine diagonal, we can extend $f$ to $\Spec(A^\wedge_I)$ by Grothendieck's algebraization theorem (as generalized by Bhatt and Halpern-Leistner \cite[Corollary 1.5]{BhattH-L} or Lurie \cite[Corollary 9.5.5.3]{SAG}), hence to $\Spec(B_\alpha)$ for some $\alpha$. We are thus in the following situation:
	\[
	\begin{tikzcd}
		& & X \ar{dd} \\
		& & \\
		\Spec(A/I) \ar{r} \ar[bend left=15]{uurr}{f} & \Spec (B_\alpha) \ar{r} \ar[dashed]{uur} & \Spec (A)\rlap.
	\end{tikzcd}
	\]
	Since $(A,I)$ is henselian and $B_\alpha$ is a smooth $A$-algebra, the morphism $\Spec(B_\alpha)\to\Spec(A)$ admits a section fixing $\Spec(A/I)$ \cite[Théorème I.8]{Gruson}, so we are done.
\end{proof}

\begin{lem}\label{lem:gp}
	Let $R$ be a derived commutative ring and $X\colon \CAlg_R^\Delta\to \Spc$ a functor left Kan extended from $\CAlg_R^\sm$. For any monoid structure on $X$, the group completion $X^\gp\colon \CAlg_R^\Delta\to \Spc$ is also left Kan extended from $\CAlg_R^\sm$.
\end{lem}

\begin{proof}
	For every $A\in \CAlg_R^\Delta$, the $\infty$-category $(\CAlg_R^\mathrm{sm})_{/A}$ has finite coproducts and hence is sifted. It follows that the forgetful functors $\Mon^{\gp}(\Spc) \to \Mon(\Spc)\to \Spc$ commute with left Kan extension along the inclusion $\CAlg^\mathrm{sm}_R\subset \CAlg_R^\Delta$ (since they preserve sifted colimits). It therefore suffices to show that the functor
	\[
	X^\gp\colon \CAlg_R^\Delta\to \Mon^{\gp}(\Spc)
	\]
	is left Kan extended from $\CAlg_R^\mathrm{sm}$. This functor is the composition of $X\colon \CAlg_R^\Delta\to \Mon(\Spc)$, which is left Kan extended, with the group completion functor $\Mon(\Spc)\to \Mon^{\gp}(\Spc)$, which preserves colimits. 
\end{proof}

\begin{ex}\label{ex:LKE}
	Let $R$ be a derived commutative ring.
	Proposition~\ref{prop:akhil2} and Lemma~\ref{lem:gp} imply that the following functors $\CAlg_R^\Delta\to\Spc$ are left Kan extended from $\CAlg_R^\sm$, being the group completions of smooth algebraic stacks with affine diagonal (defined over $\Z$):
	\begin{enumerate}
		\item algebraic K-theory $K$, which is the group completion of the stack of finite locally free sheaves;
		\item oriented K-theory $K^\SL$, which is the group completion of the stack of finite locally free sheaves with trivialized determinant;
		\item symplectic K-theory $K^\Sp$, which is the group completion of the stack of finite locally free sheaves (necessarily of even rank) with a nondegenerate alternating bilinear form;
		\item quadratic Grothendieck–Witt theory $\GW^\mathrm q$, which is the group completion of the stack of finite locally free sheaves with a nondegenerate quadratic form;
		\item symmetric Grothendieck–Witt theory $\GW^\mathrm s$, which is the group completion of the stack of finite locally free sheaves with a nondegenerate symmetric bilinear form.
	\end{enumerate}
	All these examples are presheaves of $\Einfty$-spaces, except $K^\SL$ which is $\sE_1$. However, the even rank summand $K^\SL_\ev\subset K^\SL$ is $\Einfty$. Indeed, the stack in (2) is the the pullback $\Vect\times_{\Pic^\dagger}L_\Sigma\Z$, where $\Pic^\dagger(R)=\Pic(\Mod_R(\Spt))$; the map $\Z\to \Pic^\dagger$ is only $\sE_1$, but its restriction to $2\Z\subset \Z$ is $\Einfty$.
\end{ex}

\begin{lem}\label{lem:PB-LKE}
	Let $F\to H\stackrel f\leftarrow G$ be a diagram in $\Fun(\CAlg_R^\Delta,\Spc)$. If $F$ is left Kan extended from $\CAlg_R^\sm$ and $f$ is relatively representable by smooth affine schemes, then $F\times_HG$ is left Kan extended from $\CAlg_R^\sm$.
\end{lem}

\begin{proof}
	We have a square of adjunctions
	\[
	\begin{tikzcd}
		\Pre(\SmAff_R)_{/G} \ar[shift left=1]{r}{\mathrm{LKE}} \ar[shift left=1]{d}{f_*} & \Pre(\dAff_R)_{/G} \ar[shift left=1]{d}{f_*} \ar[shift left=1]{l}{\mathrm{res}} \\
		\Pre(\SmAff_R)_{/H} \ar[shift left=1]{r}{\mathrm{LKE}} \ar[shift left=1]{u}{f^*} & \Pre(\dAff_R)_{/H} \ar[shift left=1]{l}{\mathrm{res}} \ar[shift left=1]{u}{f^*}
	\end{tikzcd}
	\]
	and we wish to prove that the square of left adjoints commutes. 
	Using the identification $\Pre(C)_{/X}\simeq \Pre(C_{/X})$ \cite[Corollary 5.1.6.12]{HTT} and the assumption on $f$, we see that the functors $f_*$ are precomposition with the pullback functors $f^*\colon (\SmAff_R)_{/H} \to (\SmAff_R)_{/G}$ and $f^*\colon (\dAff_R)_{/H} \to (\dAff_R)_{/G}$. It is then obvious that the square of right adjoints commutes.
\end{proof}

\begin{ex}\label{ex:LKE-rank}
	Let $F\colon \CAlg^\Delta_R\to \Spc$ be one of the functors from Example~\ref{ex:LKE}. Then there is a rank map $F\to L_\Sigma\Z$. For any subset $I\subset \Z$, let $F_I\subset F$ be the subfunctor consisting of elements with ranks in $I$. Then $F_I$ is left Kan extended from $\CAlg_R^\sm$. This follows from Lemma~\ref{lem:PB-LKE}, since $L_\Sigma I\subset L_\Sigma\Z$ is relatively representable by smooth affine schemes.
\end{ex}

\begin{rem}
	The proof of Proposition~\ref{prop:akhil2} is quite nonelementary. For the algebraic stacks from Example~\ref{ex:LKE}, it is possible to prove more directly that they are left Kan extended from $\CAlg_R^\sm$. Let us give such a proof for $\Vect$ itself.
	 Since the rank of a vector bundle on a derived affine scheme is bounded, we have
	\[
	\Vect = \colim_n \Vect_{\leq n}.
	\]
	Let
	\[
	\Gr_{\leq n} = \colim_{k\to\infty}\; \h(\Gr_0(\A^k)\sqcup \dotsb\sqcup \Gr_n(\A^k) ),
	\]
	where $\h$ is the Yoneda embedding.
	The canonical map $\Gr_{\leq n}\to \Vect_{\leq n}$ is an effective epimorphism of presheaves on $\dAff_R$, since every vector bundle on a derived affine scheme is generated by its global sections.
	For every $n\leq k$, choose a vector bundle torsor $U_{n,k}\to \Gr_n(\A^k)$ where $U_{n,k}$ is affine, and choose maps $U_{n,k}\to U_{n,k+1}$ compatible with $Gr_n(\A^k)\to Gr_n(\A^{k+1})$.
	Let
	\[
	U_{\leq n} = \colim_{k\to\infty}\; \h(U_{0,k} \sqcup\dotsb \sqcup U_{n,k}).
	\]
	Then the map $U_{\leq n} \to \Gr_{\leq n}$ is an effective epimorphism of presheaves on $\dAff_R$, because every vector bundle torsor over a derived affine scheme admits a section. 
	Thus, $U_{\leq n}\to \Vect_{\leq n}$ is an effective epimorphism, and hence $\Vect_{\leq n}$ is the colimit of the simplicial diagram
	\[
	\begin{tikzcd}
		\dotsb \arrow[r, shift left=2] \arrow[r] \arrow[r, shift right=2] & U_{\leq n}\times_{\Vect_{\leq n}} U_{\leq n} \arrow[r, shift left] \arrow[r, shift right] & U_{\leq n}\rlap.
	\end{tikzcd}
	\]
	Since $\Vect_{\leq n}$ is an Artin stack with smooth and affine diagonal, each term in this simplicial object is a filtered colimit of smooth affine $R$-schemes.
\end{rem}

\section{The \texorpdfstring{$\infty$}{∞}-category of twisted framed correspondences}
\label{app:category}

In this appendix, we construct a symmetric monoidal $\infty$-category $\Span^{\sL}((\dSch_S)_{/K})$ whose objects are pairs $(X,\xi)$ where $X$ is a derived $S$-scheme and $\xi\in K(X)$, and whose morphisms are spans
\[
\begin{tikzcd}
   & Z \ar[swap]{ld}{f} \ar{rd}{g} & \\
  (X,\xi) &   & (Y,\eta)
\end{tikzcd}
\]
where $\sL_f$ is perfect together with an equivalence $f^*(\xi)+\sL_f \simeq g^*(\eta)$ in $K(Z)$. 
We also construct symmetric monoidal functors
\begin{gather*}
\Span^\fr(\dSch_S) \to \Span^\sL((\dSch_S)_{/K}), \quad X\mapsto (X,0),\\
\gamma\colon (\dSch_{S})_{/K} \to \Span^\sL((\dSch_S)_{/K}),\quad (X,\xi)\mapsto (X,\xi),
\end{gather*}
where $\Span^\fr(\dSch_S)$ is the $\infty$-category of framed correspondences constructed in \cite[Section 4]{EHKSY1} (with ``scheme'' replaced by ``derived scheme'' and ``finite syntomic'' by ``finite quasi-smooth''), and $\gamma$ extends $\gamma\colon \dSch_{S}\to \Span^\fr(\dSch_S)$.
These constructions are used several times in the paper. For example, the presheaf $\h^\fr_S(Y,\xi)$ on $\Span^\fr(\dSch_S)$ is the restriction of the presheaf represented by $(Y,-\xi)$ on the wide subcategory of $\Span^\sL((\dSch_S)_{/K})$ whose morphisms have a finite quasi-smooth left leg, and the right-lax symmetric monoidal structure on the functor $(Y,\xi)\mapsto \h^\fr_S(Y,\xi)$ is induced by the symmetric monoidal functor $\gamma$.

We will construct $\Span^\sL((\dSch_S)_{/K})$ using the formalism of labeling functors developed in \cite[Section 4]{EHKSY1}, although we need a minor generalization of that formalism allowing objects to be labeled as well. The key is to generalize the notion of Segal presheaf as follows:
 
\begin{defn}\label{defn:Segal-presheaf}
	Let $X_\bullet$ be a simplicial $\infty$-category. A \emph{Segal presheaf} on $X_\bullet$ is a functor
\[
F\colon \int_{\Delta^\op} X_\bullet^\op \to \Spc
\]
such that, for every $n\geq 0$ and $\sigma \in X_n$, the map 
\[
F(\sigma) \rightarrow F(\rho_1^*(\sigma)) \times_{F(\sigma_1)} \cdots \times_{F(\sigma_{n-1})} F(\rho_n^*(\sigma))
\]
 induced by the Segal maps $\rho_i\colon [1] \rightarrow [n]$ is an equivalence.
 It is called \emph{reduced} if $F|X_0^\op$ is contractible, and it is called \emph{complete} if, for every $v\in X_0$, the map $F(v) \to F(\iota^*(v))$ induced by the unique map $\iota\colon [1]\to [0]$ is an equivalence (equivalently, if $F$ sends cocartesian edges over $\Delta_\mathrm{surj}^\op$ to equivalences).
\end{defn}

Here, $\int_\sC F\to \sC$ denotes the coCartesian fibration classified by a functor $F\colon \sC\to\InftyCat$.
Note that a Segal presheaf in the sense of \cite[Definition 4.1.14]{EHKSY1} is exactly a reduced Segal presheaf in the sense of Definition~\ref{defn:Segal-presheaf}. 

Let $C$ be an $\infty$-category and $M$ and $N$ two classes of morphisms in $C$ that are closed under composition and under pullback along one another. Recall from \cite[4.1.11]{EHKSY1} that we can associate to such a triple $(C,M,N)$ a simplicial $\infty$-category $\Phi_\bullet(C,M,N)\subset \Fun((\Delta^{\bullet})^\op, C)$, where $\Phi_n(C,M,N)$ is the subcategory of $\Fun((\Delta^n)^\op,C)$ whose objects are the functors sending every edge of $(\Delta^n)^\op$ to $M$ and whose morphisms are the Cartesian transformations with components in $N$. We now repeat \cite[Definition 4.1.15]{EHKSY1} with our generalized notion of Segal presheaf:

\begin{defn}\label{defn:labeling-functor}
	Let $(C,M,N)$ be a triple. A \emph{labeling functor} on $(C,M,N)$ is a Segal presheaf on $\Phi_\bullet(C,M,N)$.
\end{defn}

Given a triple with labeling functor $(C,M,N;F)$ and $n\geq 0$, we define the space $\Span^F_n(C,M,N)$ by applying the Grothendieck construction to the functor
\[
\Span_n(C,M,N)  \to \Spc
\]
sending an $n$-span $\sigma\colon (\Sigma_n,\Sigma_n^L,\Sigma_n^R) \to (C,M,N)$ to the limit of the composite
\[
\int_{\Delta^\op}\Phi_\bullet(\Sigma_n,\Sigma_n^L,\Sigma_n^R)^\op \xrightarrow{\sigma} \int_{\Delta^\op}\Phi_\bullet(C,M,N)^\op \xrightarrow{F}\Spc.
\]
As in \cite[4.1.18]{EHKSY1}, we obtain a functor
\[
\Lab\Trip \to \Fun(\Delta^\op,\Spc), \quad (C,M,N;F) \mapsto \Span^F_\bullet(C,M,N).
\]
Let us unpack the simplicial space $\Span^F_\bullet(C,M,N)$ in degrees $\leq 1$:
\begin{itemize}
	\item $\Span^F_0(C,M,N)$ is the space of pairs $(X,\alpha)$ where $X\in C$ and $\alpha\in F(X)$;
	\item $\Span^F_1(C,M,N)$ is the space of spans
\[
\begin{tikzcd}
   & Z \ar[swap]{ld}{f} \ar{rd}{g} & \\
  X &   & Y
\end{tikzcd}
\]
where $f\in M$ and $g\in N$, together with $\phi\in F(f)$, $\beta\in F(Y)$, and an equivalence $\delta_1^*(\phi)\simeq g^*(\beta)$;
\item the degeneracy map $s_0\colon\Span^F_0(C,M,N) \to \Span^F_1(C,M,N)$ sends $(X,\alpha)$ to the identity span on $X$ with $\phi=\iota^*(\alpha)$, $\beta=\alpha$, and $\delta_1^*\iota^*(\alpha)\simeq \alpha$ the canonical equivalence;
\item the face map $d_0\colon \Span^F_1(C,M,N) \to \Span^F_0(C,M,N)$ sends a span as above to $(X,\delta_0^*(\phi))$;
\item the face map $d_1\colon \Span^F_1(C,M,N) \to \Span^F_0(C,M,N)$ sends a span as above to $(Y,\beta)$.
\end{itemize}

\begin{prop}\label{prop:Corr^F}
	Let $(C,M,N;F)$ be a triple with labeling functor. Then $\Span^F_\bullet(C,M,N)$ is a Segal space. If $F$ is complete, then $\Span^F_\bullet(C,M,N)$ is a complete Segal space.
\end{prop}

\begin{proof}
	The proof of the first statement is exactly the same as the proof of \cite[Theorem 4.1.23]{EHKSY1}. The second statement is obvious from the description of $1$-simplices above.
\end{proof}

\begin{rem}
	One can show that the above construction subsumes Haugseng's $\infty$-categories of spans with local systems \cite[Definition 6.8]{haugseng}. Indeed, to a presheaf of (complete) Segal spaces $\sF$ on an $\infty$-category $C$ with pullbacks, one can associate a (complete) labeling functor $F$ on $C$ such that $\Span^F_\bullet(C)$ is the Segal space of spans in $C$ with local systems valued in $\sF$.
\end{rem}

Let us denote by ``$\perf$'' the class of morphisms of derived schemes with perfect cotangent complex.
We now seek to construct a labeling functor $\sK$ on the pair $(\dSch,\perf)$, such that the restriction of $\sK$ to $\Phi_0(\dSch,\perf)^\op=\dSch^\op$ is the K-theory presheaf $K$. Moreover, for an $n$-simplex
\[
\sigma=(X_0 \leftarrow X_1 \leftarrow\dotsb\leftarrow X_n)
\]
in $\Phi_n(\dSch,\perf)$, the first vertex map $\sK(\sigma) \to K(X_0)$ should be an equivalence (in particular, $\sK$ should be complete), and for $0\leq i\leq n$ the $i$th vertex map $K(X_0)\simeq \sK(\sigma)\to K(X_i)$ should be $\xi\mapsto f_{i}^*(\xi)+\sL_{f_{i}}$ where $f_{i}\colon X_i\to X_0$.

Let $p\colon X\to S$ be a coCartesian fibration classified by a functor $S\to \InftyCat^{\mathrm{pt,rex}}$. In \cite[Definition 4.2.8]{EHKSY1} we introduced the $\infty$-category $\Gap_S(n,X)$ of \emph{relative $n$-gapped objects} of $X$, which is equivalent to the full subcategory of $\Fun(\Delta^n,X)$ spanned by the functors sending $0$ to a $p$-relative zero object (i.e., a zero object in its fiber). 
Let
\[
\Filt_S(n,X) = \Fun(\Delta^n,X).
\]
The simplicial $\infty$-category $\Filt_S(\bullet,X)$ classifies a coCartesian fibration $\Filt_S(X) \to \Delta^\op$. Let $\Gap_S(X)\subset \Filt_S(X)$ be the full subcategory on those functors $\Delta^n\to X$ sending $0$ to a $p$-relative zero object. Then $\Gap_S(X) \to \Delta^\op$ is a coCartesian fibration classified by $\Gap_S(\bullet,X)$. Moreover, the inclusion $\Gap_S(X)\subset \Filt_S(X)$ has a left adjoint preserving coCartesian edges. By straightening, it gives rise to a morphism of simplicial $\infty$-categories
\[
\Filt_S(\bullet,X) \to \Gap_S(\bullet,X)
\]
over $\Fun(\Delta^\bullet,S)$, sending $x_0\to x_1\to\dotsb \to x_n$ to the relative $n$-gapped object $0_{p(x_0)}\to x_1/x_0\to\dotsb \to x_n/x_0$, where $x_i/x_0$ denotes a $p$-relative cofiber. 

\begin{rem}
	When $S=*$, the simplicial map $\Filt_S(\bullet,X)\to \Gap_S(\bullet,X)$ is the one constructed by Barwick in \cite[Corollary 5.20.1]{BarwickKtheory}. However, for general $S$, our notion of ``relative'' is essentially different.
\end{rem}

We specialize to the coCartesian fibration $p\colon \Perf \to \dSch^\op$. For $\sigma\colon \Delta^n\to \dSch^\op$, let $\Gap_\sigma(\Perf)$ be the fiber of the coCartesian fibration $p_*\colon \Gap_{\dSch^\op}(n,\Perf) \to \Fun(\Delta^n,\dSch^\op)$ over $\sigma$, and let $\Filt_\sigma(\Perf)$ be the fiber of the coCartesian fibration $p_*\colon\Filt_{\dSch^\op}(n,\Perf) \to \Fun(\Delta^n,\dSch^\op)$ over $\sigma$. Then, by the additivity property of K-theory, we have canonical equivalences
\begin{gather*}
	K(\Gap_\sigma(\Perf)) \simeq K(\sigma_1)\times \dotsb\times K(\sigma_n),\\
	K(\Filt_\sigma(\Perf)) \simeq K(\sigma_0) \times K(\sigma_1) \times \dotsb \times K(\sigma_n),
\end{gather*}
such that the map $\Filt_\sigma(\Perf) \to \Gap_\sigma(\Perf)$ induces the projection onto the last $n$ factors (cf.\ \cite[4.2.20]{EHKSY1}).

As in \cite[4.2.22]{EHKSY1}, we can take K-theory fiberwise to obtain a morphism of simplicial coCartesian fibrations in spaces
\[
 \begin{tikzcd}
 	 K\Filt_{\dSch^\op}(\bullet,\Perf) \ar{rr} \ar{dr} &[-4em] &[-4em]  K\Gap_{\dSch^\op}(\bullet,\Perf) \ar{dl} \\
	 & \Fun(\Delta^\bullet,\dSch^\op)\rlap. & 
 \end{tikzcd}
\]
Moreover, the horizontal arrow is itself a simplicial coCartesian fibration in spaces, by \cite[Lemma 1.4.14]{LurieGoodwillie}.
In \cite[4.2.17]{EHKSY1}, we packaged the cotangent complex into a section
\[
\begin{tikzcd}
	& K\Gap_{\dSch^\op}(\bullet,\Perf) \ar{d} \\
	\Phi_\bullet(\dSch,\perf)^\op \ar[hook]{r} \ar[dashed]{ur}{\sL} & \Fun(\Delta^\bullet,\dSch^\op)\rlap.
\end{tikzcd}
\]
We now form the Cartesian square
\begin{equation}\label{eqn:labeling}
\begin{tikzcd}
	P_\bullet \ar{r} \ar{d} & K\Filt_{\dSch^\op}(\bullet,\Perf) \ar{d} \\
	\Phi_{\bullet}(\dSch,\perf)^\op \ar{r}{\sL} & K\Gap_{\dSch^\op}(\bullet,\Perf)\rlap,
\end{tikzcd}
\end{equation}
where the vertical arrows are simplicial coCartesian fibrations in spaces.
By \cite[Lemma 1.4.14]{LurieGoodwillie}, the left vertical map induces a coCartesian fibration in spaces
\[
\int_{\Delta^\op} P_\bullet \to \int_{\Delta^\op} \Phi_{\bullet}(\dSch,\perf)^\op,
\]
which is classified by a functor
\[
\sK\colon \int_{\Delta^\op} \Phi_{\bullet}(\dSch,\perf)^\op \to \Spc.
\]
It is not difficult to show that $\sK$ is a labeling functor on $(\dSch,\perf)$ with the desired properties (cf.\ \cite[Proposition 4.2.31]{EHKSY1}).

\begin{defn}
	Let $S$ be a derived scheme.
	The $\infty$-category $\Span^\sL((\dSch_S)_{/K})$ is the complete Segal space $\Span^\sK_\bullet(\dSch_S,\perf)$.
\end{defn}

The labeling functor
\[
\fr\colon \int_{\Delta^\op} \Phi_{\bullet}(\dSch,\perf)^\op \to \Spc
\]
constructed in \cite[4.2.29]{EHKSY1} is obtained from~\eqref{eqn:labeling} by pulling back $\sL$ further along the zero section $\Phi_{\bullet}(\dSch,\perf)^\op \to K\Filt_{\dSch^\op}(\bullet,\Perf)$, so there is a canonical natural transformation $\fr\to\sK$ of labeling functors on $(\dSch,\perf)$, inducing a functor
\[
\Span^\fr(\dSch_S) \to \Span^\sL((\dSch_S)_{/K}), \quad X\mapsto (X,0).
\]
The functor
\[
\gamma\colon (\dSch_{S})_{/K} \to \Span^\sL((\dSch_S)_{/K})
\]
is simply the inclusion of the wide subcategory on those spans whose left leg is an equivalence.

Finally, we can equip $\Span^\sL((\dSch_S)_{/K})$ with a symmetric monoidal structure, where
\[
(X,\xi) \otimes (Y,\eta) = (X\times_SY, \pi_X^*(\xi)+\pi_Y^*(\eta)).
\]
To that end we must promote $\sK$ to a symmetric monoidal labeling functor \cite[Definition 4.3.5]{EHKSY1}. This is done exactly as in \cite[\sectsign 4.3]{EHKSY1}, starting with the functor $\Fin_*\to \coCart$ sending $I_+$ to the coCartesian fibration $\Perf_S^I \to (\dSch_S^\op)^I$.


\bibliographystyle{alphamod}

\begin{thebibliography}{GMTW09}

\bibitem[AGP18]{agp}
A.~Ananyevskiy, G.~Garkusha, and I.~Panin, \emph{Cancellation theorem for
  framed motives of algebraic varieties}, 2018,
  \href{http://arxiv.org/abs/1601.06642v2}{{\sf arXiv:1601.06642v2}}

\bibitem[Bar16]{BarwickKtheory}
C.~Barwick, \emph{On the algebraic $K$-theory of higher categories}, J. Topol.
  \textbf{9} (2016), no.~1, pp.~245--347

\bibitem[BEH{\etalchar{+}}20]{deloop4}
T.~Bachmann, E.~Elmanto, M.~Hoyois, A.~A. Khan, V.~Sosnilo, and M.~Yakerson,
  \emph{On the infinite loop spaces of algebraic cobordism and the motivic
  sphere}, 2020, \href{http://arxiv.org/abs/1911.02262v3}{{\sf
  arXiv:1911.02262v3}}

\bibitem[BF18]{BachmannFasel}
T.~Bachmann and J.~Fasel, \emph{On the effectivity of spectra representing
  motivic cohomology theories}, 2018,
  \href{http://arxiv.org/abs/1710.00594v3}{{\sf arXiv:1710.00594v3}}

\bibitem[BH20]{norms}
T.~Bachmann and M.~Hoyois, \emph{Norms in motivic homotopy theory}, to appear
  in Ast{\'e}risque, 2020, \href{http://arxiv.org/abs/1711.03061}{{\sf
  arXiv:1711.03061}}

\bibitem[BHL17]{BhattH-L}
B.~Bhatt and D.~Halpern-Leistner, \emph{Tannaka duality revisited}, Adv. Math.
  \textbf{316} (2017), pp.~576--612

\bibitem[CD15]{CDintegral}
D.-C. Cisinski and F.~D{\'e}glise, \emph{Integral mixed motives in equal
  characteristic}, Doc. Math., Extra Volume: Alexander S. Merkurjev's Sixtieth
  Birthday (2015), pp.~145--194

\bibitem[CMNN20]{CMNN}
D.~Clausen, A.~Mathew, N.~Naumann, and J.~Noel, \emph{Descent in algebraic
  $K$-theory and a conjecture of Ausoni--Rognes}, J. Eur. Math. Soc.
  \textbf{22} (2020), no.~4, pp.~1149--1200

\bibitem[D{\'e}g18]{DegliseOrientation}
F.~D{\'e}glise, \emph{Orientation theory in arithmetic geometry}, $K$-theory
  (V.~Srinivas, S.~K. Roushon, R.~A. Rao, A.~J. Parameswaran, and A.~Krishna,
  eds.), Tata Institute of Fundamental Research Publications, vol.~19, 2018,
  pp.~241--350, preprint \href{http://arxiv.org/abs/1111.4203}{{\sf
  arXiv:1111.4203}}

\bibitem[DJK20]{DJKFundamental}
F.~D{\'e}glise, F.~Jin, and A.~A. Khan, \emph{Fundamental classes in motivic
  homotopy theory}, to appear in J. Eur. Math. Soc., 2020,
  \href{http://arxiv.org/abs/1805.05920}{{\sf arXiv:1805.05920}}

\bibitem[Dru20]{Druzhinin}
A.~Druzhinin, \emph{Framed motives of smooth affine pairs}, 2020,
  \href{http://arxiv.org/abs/1803.11388v8}{{\sf arXiv:1803.11388v8}}

\bibitem[EHK{\etalchar{+}}19]{EHKSY1}
E.~Elmanto, M.~Hoyois, A.~A. Khan, V.~Sosnilo, and M.~Yakerson, \emph{Motivic
  infinite loop spaces}, 2019, \href{http://arxiv.org/abs/1711.05248v5}{{\sf
  arXiv:1711.05248v5}}

\bibitem[EHK{\etalchar{+}}20]{EHKSY2}
\bysame, \emph{Framed transfers and motivic fundamental classes}, J. Topol.
  \textbf{13} (2020), no.~2, pp.~460--500, preprint
  \href{http://arxiv.org/abs/1809.10666v1}{{\sf arXiv:1809.10666v1}}

\bibitem[EK20]{ElmantoKolderup}
E.~Elmanto and H.~Kolderup, \emph{On modules over motivic ring spectra}, Ann.
  K-Theory \textbf{5} (2020), no.~2, pp.~327--355, preprint
  \href{http://arxiv.org/abs/1708.05651}{{\sf arXiv:1708.05651}}

\bibitem[GMTW09]{GMTW}
S.~Galatius, I.~Madsen, U.~Tillmann, and M.~Weiss, \emph{The homotopy type of
  the cobordism category}, Acta Math. \textbf{202} (2009), no.~2, pp.~195--239

\bibitem[GN18]{GarkushaNeshitov}
G.~Garkusha and A.~Neshitov, \emph{Fibrant resolutions for motivic Thom
  spectra}, 2018, \href{http://arxiv.org/abs/1804.07621v1}{{\sf
  arXiv:1804.07621v1}}

\bibitem[GNP18]{gnp}
G.~Garkusha, A.~Neshitov, and I.~Panin, \emph{Framed motives of relative
  motivic spheres}, 2018, \href{http://arxiv.org/abs/1604.02732v3}{{\sf
  arXiv:1604.02732v3}}

\bibitem[GP20a]{garkusha2014framed}
G.~Garkusha and I.~Panin, \emph{Framed motives of algebraic varieties (after V.
  Voevodsky)}, to appear in J. Amer. Math. Soc., 2020,
  \href{http://arxiv.org/abs/1409.4372}{{\sf arXiv:1409.4372}}

\bibitem[GP20b]{hitr}
\bysame, \emph{Homotopy invariant presheaves with framed transfers}, Cambridge
  J. Math. \textbf{8} (2020), no.~1, pp.~1--94, preprint
  \href{http://arxiv.org/abs/1504.00884}{{\sf arXiv:1504.00884}}

\bibitem[Gro67]{EGA4-4}
A.~Grothendieck, \emph{{\'E}l{\'e}ments de {G}{\'e}om{\'e}trie
  {A}lg{\'e}brique: {IV.} {{\'E}}tude locale des sch{\'e}mas et des morphismes
  de sch{\'e}mas, {Q}uatri{\`e}me partie}, Publ. Math. I.H.{\'E}.S. \textbf{32}
  (1967)

\bibitem[Gru72]{Gruson}
L.~Gruson, \emph{Une propri{\'e}t{\'e} des couples hens{\'e}liens}, Colloque
  d'alg{\`e}bre commutative, exp.~n\textsuperscript{o}~10, Publications des
  s{\'e}minaires de math{\'e}matiques et informatique de Rennes, 1972

\bibitem[Hau18]{haugseng}
R.~Haugseng, \emph{Iterated spans and classical topological field theories},
  Math. Z. \textbf{289} (2018), no.~3, pp.~1427--1488, preprint
  \href{http://arxiv.org/abs/1409.0837}{{\sf arXiv:1409.0837}}

\bibitem[Hoy20]{framed-loc}
M.~Hoyois, \emph{The localization theorem for framed motivic spaces}, to appear
  in Compos. Math., 2020, \href{http://arxiv.org/abs/1807.04253}{{\sf
  arXiv:1807.04253}}

\bibitem[Kha16]{KhanThesis}
A.~A. Khan, \emph{Motivic homotopy theory in derived algebraic geometry}, Ph.D.
  thesis, Universit{\"a}t Duisburg-Essen, 2016, available at
  \url{https://www.preschema.com/thesis/thesis.pdf}

\bibitem[KR19]{KhanVCD}
A.~A. Khan and D.~Rydh, \emph{Virtual Cartier divisors and blow-ups}, 2019,
  \href{http://arxiv.org/abs/1802.05702v2}{{\sf arXiv:1802.05702v2}}

\bibitem[LM07]{Levine:2007}
M.~Levine and F.~Morel, \emph{Algebraic Cobordism}, Springer, 2007

\bibitem[LS16]{LowreySchurg}
P.~Lowrey and T.~Sch{\"u}rg, \emph{Derived algebraic cobordism}, J. Inst. Math.
  Jussieu \textbf{15} (2016), pp.~407--443

\bibitem[Lur04]{LurieThesis}
J.~Lurie, \emph{Derived Algebraic Geometry}, Ph.D. Thesis, 2004,
  \url{https://www.math.ias.edu/~lurie/papers/DAG.pdf}

\bibitem[Lur09]{LurieGoodwillie}
\bysame, \emph{$(\infty,2)$-Categories and the Goodwillie Calculus I}, 2009,
  \url{https://www.math.ias.edu/~lurie/papers/GoodwillieI.pdf}

\bibitem[Lur17a]{HA}
\bysame, \emph{Higher Algebra}, September 2017,
  \url{https://www.math.ias.edu/~lurie/papers/HA.pdf}

\bibitem[Lur17b]{HTT}
\bysame, \emph{Higher Topos Theory}, April 2017,
  \url{https://www.math.ias.edu/~lurie/papers/HTT.pdf}

\bibitem[Lur18]{SAG}
\bysame, \emph{Spectral Algebraic Geometry}, February 2018,
  \url{https://www.math.ias.edu/~lurie/papers/SAG-rootfile.pdf}

\bibitem[Nav16]{NavarroGysin}
A.~Navarro, \emph{Riemann--Roch for homotopy invariant $K$-theory and Gysin
  morphisms}, 2016, \href{http://arxiv.org/abs/1605.00980v1}{{\sf
  arXiv:1605.00980v1}}

\bibitem[Nik17]{Nikolaus}
T.~Nikolaus, \emph{The group completion theorem via localizations of ring
  spectra}, 2017,
  \url{https://www.uni-muenster.de/IVV5WS/WebHop/user/nikolaus/papers.html}

\bibitem[Pan09]{PaninOrientedII}
I.~Panin, \emph{Oriented cohomology theories of algebraic varieties II},
  Homology Homotopy Appl. \textbf{11} (2009), no.~1, pp.~349--405

\bibitem[Qui71]{Quillen:1971}
D.~Quillen, \emph{Elementary Proofs of Some Results of Cobordism Theory Using
  Steenrod Operations}, Adv. Math. \textbf{7} (1971), no.~1, pp.~29--56

\bibitem[Rap19]{Raptis}
G.~Raptis, \emph{Some characterizations of acyclic maps}, J. Homotopy Relat.
  Struct. \textbf{14} (2019), pp.~773--785, preprint
  \href{http://arxiv.org/abs/1711.08898}{{\sf arXiv:1711.08898}}

\bibitem[R{\O}08]{Rondigs:2008}
O.~R{\"o}ndigs and P.~A. {\O}stv{\ae}r, \emph{Modules over motivic cohomology},
  Adv. Math. \textbf{219} (2008), no.~2, pp.~689--727

\bibitem[RW13]{ORW}
O.~Randal-Williams, \emph{``Group-completion'', local coefficient systems and
  perfection}, Q. J. Math. \textbf{64} (2013), no.~3, pp.~795--803

\bibitem[Ryd15]{RydhApprox}
D.~Rydh, \emph{Noetherian approximation of algebraic spaces and stacks}, J.
  Algebra \textbf{422} (2015), pp.~105--147

\bibitem[Spi18]{SpitzweckHZ}
M.~Spitzweck, \emph{A commutative $\mathbb{P}^1$-spectrum representing motivic
  cohomology over {D}edekind domains}, M{\'e}m. Soc. Math. Fr. \textbf{157}
  (2018), preprint \href{http://arxiv.org/abs/1207.4078}{{\sf arXiv:1207.4078}}

\bibitem[Stacks]{stacks}
{The Stacks Project Authors}, \emph{The Stacks Project},
  \url{http://stacks.math.columbia.edu}

\bibitem[TV08]{HAG2}
B.~To{\"e}n and G.~Vezzosi, \emph{Homotopical Algebraic Geometry. II. Geometric
  stacks and applications}, Mem. Amer. Math. Soc. \textbf{193} (2008), no.~902,
  preprint \href{http://arxiv.org/abs/math/0404373}{{\sf arXiv:math/0404373}}

\bibitem[Voe01]{voevodsky2001notes}
V.~Voevodsky, \emph{Notes on framed correspondences}, unpublished, 2001,
  \url{http://www.math.ias.edu/vladimir/files/framed.pdf}

\bibitem[Yak19]{MuraThesis}
M.~Yakerson, \emph{Motivic stable homotopy groups via framed correspondences},
  Ph.D. thesis, University of Duisburg-Essen, 2019, available at
  \url{https://duepublico2.uni-due.de/receive/duepublico_mods_00070044?q=iakerson}

\end{thebibliography}

\let\mathbb=\mathbf

{\small
\newcommand{\etalchar}[1]{$^{#1}$}
\providecommand{\bysame}{\leavevmode\hbox to3em{\hrulefill}\thinspace}

}

\parskip 0pt

\end{document}